\documentclass[11pt]{article}

\usepackage[utf8]{inputenc}
\usepackage[T1]{fontenc}
\usepackage[margin=1in]{geometry}
\usepackage{amsthm,amsmath,amsfonts,amssymb}
\usepackage[authoryear]{natbib}
\bibpunct{(}{)}{;}{a}{,}{,}
\usepackage{xurl}
\usepackage[colorlinks,citecolor=blue,urlcolor=blue,linkcolor=blue]{hyperref}
\usepackage{graphicx}
\usepackage{tikz}
\usepackage{booktabs}
\usepackage{bm}
\usepackage{bbm}
\usepackage{siunitx}
\usepackage{enumitem}

\hypersetup{
  pdftitle={Combining Concurrent and Historical Functional Linear Regression},
  pdfauthor={Alois Kneip, Dominik Liebl, and Sven Otto},
  pdfsubject={Functional data analysis},
  pdfkeywords={functional data analysis, functional linear model, concurrent effect, identifiability, gait analysis}
}

\numberwithin{equation}{section}

\theoremstyle{plain}
\newtheorem{theorem}{Theorem}[section]
\newtheorem{lemma}[theorem]{Lemma}

\theoremstyle{definition}

\newtheorem{remark}[theorem]{Remark}
\newtheorem*{I-assumption}{Identification Assumptions}
\newtheorem*{I-assumption-cont}{Identification Assumptions (continued)}
\newtheorem*{A-assumption}{Asymptotic Assumptions}
\newtheorem*{A-assumption-cont}{Asymptotic Assumptions (continued)}

\DeclareMathOperator{\E}{\mathbb{E}}
\DeclareMathOperator{\Cov}{Cov}

\DeclareMathOperator{\V}{\mathbb{V}}

\DeclareMathOperator*{\argmin}{arg\,min}
\newcommand{\dd}{\,\mathrm{d}}

\title{Combining Concurrent and Historical Functional Linear Regression}
\author{
  Alois Kneip\textsuperscript{1}, Dominik Liebl\textsuperscript{1}, and Sven Otto\textsuperscript{2}\\[0.6em]
  \small\textsuperscript{1}Institute of Finance and Statistics, University of Bonn\\
  \small\textsuperscript{2}Department of Econometrics and Statistics, University of Cologne
}
\date{}

\begin{document}

\maketitle

\begin{abstract}
We study a function-on-function linear regression model in which the response at time $t$ depends on both the past trajectory of a predictor and its concurrent value. The model combines an $L^2$-historical effect with a point-evaluation effect, and these two coefficient functions are not automatically identifiable. We characterize the resulting non-identifiability and show that the concurrent and historical effects are separately identifiable whenever the covariance eigenfunctions of the covariance operator of the predictor are not pointwise square-summable. This mild and novel condition prevents the concurrent point evaluation from being represented by an $L^2$-historical effect.

Building on an orthogonalized representation of the predictor process, we propose a smoothing-spline estimator for both coefficient functions and establish consistency rates. The rates reveal an interesting trade-off between path regularity and eigenvalue decay: smoother predictor trajectories lead to faster convergence of the historical-effect estimator, whereas rougher trajectories lead to faster convergence of the concurrent-effect estimator. Simulation studies and two real-data applications demonstrate the practical importance of disentangling concurrent and historical effects.
\end{abstract}


\section{Introduction}\label{sec:intro}

In many function-on-function regression applications, both the response and the predictor are observed as functions over time.
A natural temporal restriction is that the response $Y(t)$ at time $t\in[0,1]$ may depend on the predictor trajectory observed up to time $t$, but not on future values $X(s)$ with $s>t$.
In this context, \cite{malfait2003} proposed the \emph{historical functional linear model},
\begin{align}\label{eq:hist}
    Y(t)&=\alpha_0(t)+\int_0^t\beta(t,s)X(s)\, ds +\varepsilon(t),
\end{align}
where $\alpha_0(t)$ denotes the intercept function, $\beta(t,s)$ the bivariate regression coefficient function with the restriction $\beta(t,s)=0$ for all $s>t$, and $\varepsilon(t)$ the statistical error function.
The historical model can be viewed as a temporally restricted version of the functional linear regression model, whose asymptotic theory has been developed in important contributions such as \citet{Yao_Mueller_Wang_2005}, \citet{Hall_Horowitz_2007}, and \citet{crambes2009}.
Using the unit interval $[0,1]$ as the standardized time domain is, of course, without loss of generality.

The historical functional linear model has become one of the most widely used functional linear regression approaches, with successful applications in biomechanics, medicine, and environmental sciences \citep[see, for instance,][among many others]{Pomann_et_al_2016,Wang_Review_2016,Boudreault_et_al_2019,Christgau_AOS_2023,Meyer_HFDA_Review_2024}.
Because the integral in \eqref{eq:hist} extends to the current time $t$, it is sometimes tempting to interpret the historical functional linear model as including a concurrent contribution of $X(t)$ to $Y(t)$ \citep[cf.][]{Bernardi_et_al_JCGS_2022,Meyer_HFDA_Review_2024}.
However, the integral term cannot represent a separate pointwise effect at $s=t$, since a single time point has Lebesgue measure zero.
Accordingly, \eqref{eq:hist} captures effects accumulated over intervals of past predictor values, including values arbitrarily close to $t$, but it does not include a separate pointwise concurrent component.
In many applications of historical functional regression, such a concurrent effect, distinct from the effect accumulated over past values, is nonetheless plausible---as shown in our applications in Section \ref{sec:applications}.

This limitation is particularly problematic because past values $X(s)$ with $s<t$ close to $t$ are typically strongly correlated with the concurrent value $X(t)$.
As a consequence, if a concurrent effect is present but not modeled explicitly, the historical coefficient surface $\beta(t,s)$ estimated under the historical-only model \eqref{eq:hist} is generally subject to an omitted-variable bias absorbing the concurrent effect so that it need not represent the historical effect of interest.
Conversely, if a historical effect is present but not modeled explicitly, the coefficient function $\alpha(t)$ estimated under the concurrent-only model
\begin{equation}\label{eq:conc}
    Y(t)=\alpha_0(t)+\alpha(t)X(t)+\varepsilon(t),
\end{equation}
is generally also subject to an omitted-variable bias that combines the direct concurrent effect with the historical contributions of $X$.
The statistical problem is therefore not merely predictive.
It concerns the possibility to separately interpret the two different types of functional effects: an integral effect over the past and a pointwise effect at the present time. This issue is particularly relevant in applications where the concurrent effect is of substantive interest, for example, in our biomechanical application in Section~\ref{sec:appl2}, where the concurrent effect quantifies the direct influence of the predictor on the response at the same time point, while the historical effect captures the influence of past values of the predictor on the response at time $t$.

To address this problem, we consider the following functional linear regression model, which combines the historical-only model \eqref{eq:hist} with the concurrent-only model \eqref{eq:conc}:
\begin{align}\label{eq:combined}
    Y(t)&=\alpha_0(t)+\alpha(t)X(t)+\int_0^t\beta(t,s)X(s)\, ds+\varepsilon(t),
\end{align}
where the concurrent coefficient function $\alpha(t)$ quantifies the effect of $X(t)$ on $Y(t)$ at time $t\in[0,1]$, while $\beta(t,s)$ quantifies the historical effect of $X(s)$ on $Y(t)$ for $s<t$.
At first sight, model~\eqref{eq:combined} may look like a simple extension of the historical functional linear model by an additional concurrent regressor.
This view is misleading.
The term $X(t)$ is a point evaluation of the functional predictor, whereas the historical component is an $L^2$-integral effect.
Separate identifiability of the two effects is not automatic and belongs to the core challenges of this work.

Indeed, without identifiability, there may exist an alternative coefficient surface $\breve\beta(t,s)$ such that, for each $t$,
$\int_0^t \breve\beta(t,s)X(s)\,ds
=
\alpha(t)X(t)+\int_0^t\beta(t,s)X(s)\,ds$,
so that different parameter pairs produce the same fitted values.
The central difficulty is therefore not only the estimation of a larger model, but whether the decomposition in model \eqref{eq:combined} is actually identifiable. That is, whether a point-evaluation effect can be distinguished from an $L^2$-historical effect generated by the same predictor process.
In this paper, we derive a novel identification theory for this decomposition and show when the concurrent and historical components are separately identifiable.
In particular, we show that the concurrent coefficient function $\alpha(t)$ and the historical coefficient function $\beta(t,s)$ are identified if the eigenfunctions of the covariance operator of the predictor function $X$ are not pointwise square-summable.

Model~\eqref{eq:combined} is related to the point of impact regression models proposed by \cite{mckeague2010} and \cite{kneip2016}. While existing works in the points of impact literature, however, use a Reproducing Kernel Hilbert Space (RKHS) to express the functional effect and the points of impact effect in a single, \emph{joint} RKHS-valued parameter \citep[see, e.g.,][]{Kupresanin2010,Berrendero_RKHS_2018,Berrendero_RKHS_2019,kneip2020}, our focus is different.
Rather than seeking a joint representation, we establish conditions under which both effects are \emph{separately} identifiable.
This general type of identification problem was so far, to the best of our knowledge, only considered by \citet{kneip2016}, who introduce \emph{specific local variation} as an identification condition---a condition, which is particularly plausible for rough functional predictors.
In this work, we revisit the identification problem and show that identification is achievable under a pointwise non-square-summability condition on the eigenfunctions of the covariance operator of the predictor function $X$.
This new and rather mild condition differs from the local-variation assumption and readily applies to both smooth and rough functional predictors. Moreover, by contrast to \cite{kneip2016}, we consider a functional response $Y(t)$ and develop identification and estimation theory for arbitrary time points $t\in[0,1]$.

In the literature on functional linear regression, there are several approaches that consider the historical effects with a focus on the recent past. \citet{Sentuerk_Mueller_2010} and \citet{Kim_Sentuerk_Li_2011} propose a functional varying coefficient model of the form
$\mathbb{E}(Y(t)|X)=\alpha_0(t)+b(t)\int_0^{r}\gamma(u)X(t-u)\,du$
that treats the recent (small $r>0$) historical effect process $\int_0^{r}\gamma(u)X(t-u)\,du$ as the functional predictor.
This construction captures effects of the recent past, but it does not introduce a separate parameter for a concurrent contribution of $X(t)$ beyond the recent-history term, so that the concurrent effect and the effect of the recent past may be difficult to disentangle in interpretation.
More recently, \citet{Luo_et_al_Biometrics_2022} and \citet{Bernardi_et_al_JCGS_2022} suggest regularized historical functional linear regression models in which the coefficient surface $\beta(t,s)$ is sparse, for example with nonzero effects restricted to a band $|t-s|\le r$ for some possibly small $r>0$.
\citet{Bernardi_et_al_JCGS_2022} interpret such banded historical effects as a hybrid between historical and concurrent regression. In contrast, our identification results show that the concurrent component $\alpha(t)$ in \eqref{eq:combined} can be identified \emph{separately} from the historical surface $\beta(t,s)$.

The combined model \eqref{eq:combined} contains the \emph{varying-coefficient model} \citep{hastie1993}, also known as the \emph{concurrent functional linear regression model} \citep[][Ch.~14]{RS_2005_book} as the special case $\beta(t,s)=0$ for all $s,t\in[0,1]$. The concurrent model \eqref{eq:conc} is widely used in applications \citep[see, for instance,][]{Zhang_et_al_2011,Li_et_al_2016,Chu_Runze_Reimherr_2016,Ghosal_et_al_2020,Torti_et_al_2021,Petrovich_et_al_2023}.
Recent methodological extensions include concurrent regression with dependent and noisy functional predictors \citep{Ding_Yao_Zhang_2023}, simultaneous inference for concurrent regressions with non-smooth features and monotonicity \citep{Chang_McKeague_2024}, and models allowing concurrent and nonlinear dynamic effects from multiple predictors \citep{Jia_et_al_2024}.
However, if the true relationship contains a historical component as in \eqref{eq:combined}, then fitting the purely concurrent model \eqref{eq:conc} leads to a coefficient function that generally combines the direct concurrent effect with historical contributions of $X$.
This misspecification issue and its implications for the interpretation of $\alpha(t)$ appear not to have been analyzed explicitly in the concurrent regression literature.

Beyond identifiability, we develop estimation and asymptotic theory for model~\eqref{eq:combined}.
The concurrent term $\alpha(t)X(t)$ involves point evaluation, while the historical term is an integral effect, and this distinction is reflected in the estimation theory.
Our estimation strategy is based on an orthogonalized representation of the predictor that separates the concurrent direction $X(t)$ from the remaining historical variation.
This leads to a scalar-on-function regression problem for the historical coefficient at each fixed time point $t$, followed by a plug-in estimation of the concurrent coefficient function $\alpha(t)$.
Although the resulting estimator uses smoothing spline methodology from \citet{crambes2009}, our asymptotic theory is fundamentally different from existing functional linear regression theory, because the orthogonalized historical predictor is itself unknown and must be estimated in a pre-estimation step. This renders the asymptotic theory highly nonstandard as the usual assumptions made on the original predictor process $X$ affect the actually relevant, empirically estimated orthogonalized predictor in a highly nontrivial way. Our asymptotic analysis shows that the resulting convergence rates depend on the path regularity of $X$.
Rougher trajectories yield faster convergence for the concurrent component, whereas smoother trajectories yield faster convergence for the historical component.

The paper and its contributions are structured as follows.
Section~\ref{sec:identification} develops our novel identification results.
Section~\ref{sec:estimation} introduces the estimators for the coefficient functions $\alpha_0(t)$, $\alpha(t)$ and $\beta(t,s)$ of model \eqref{eq:combined}.
Section~\ref{sec:consistency} presents the asymptotic theory showing that estimating the historical-effect function profits from smooth predictor trajectories, while estimating the concurrent-effect function benefits from rougher trajectories.
Section~\ref{sec:simulations} investigates finite-sample performance in simulations, and Section~\ref{sec:applications} presents two biomechanical applications to gait analysis data.
Section~\ref{sec:conclusion} concludes.
Proofs of the theoretical results are provided in the Supplementary Material \cite{SupplementaryPaper}.

\section{Identifiability of the Model Parameters} \label{sec:identification}

In this section, we discuss conditions under which the model components of model~\eqref{eq:combined} are identified.
Throughout the paper, $\langle f,g\rangle=\int_0^1 f(s)g(s)\,ds$ denotes the $L^2[0,1]$ inner product with induced norm $\|f\|=\langle f,f\rangle^{1/2}$.
Let $\Delta=\{(t,s):0\leq s\leq t\leq 1\}$ denote the historical domain of the coefficient surface $\beta$.
For each $t\in[0,1]$, define the historical coefficient at time $t$, extended to $[0,1]$, by
$$
    \beta_t(s)=
    \begin{cases}
        \beta(t,s), & 0\leq s\leq t,\\
        0, & t<s\leq 1.
    \end{cases}
$$
This allows us to write the historical component as the $L^2[0,1]$ inner product $\langle \beta_t,X^\mu\rangle$.
For convenience, we consider the centered curves
$$
X^\mu(t) := X(t) - \E(X(t))
\quad\text{and}\quad
Y^\mu(t) := Y(t) - \E(Y(t)).
$$
Under the exogeneity condition imposed below, $\E(\varepsilon(t))=0$, so centering removes the intercept component $\alpha_0(t)$ from model~\eqref{eq:combined}.
Using the inner product notation, we can write model~\eqref{eq:combined} as
\setlength{\arraycolsep}{2pt}
\begin{equation}\label{eq:operator}
\begin{array}{rcccl}
Y^\mu(t) &=& \alpha(t) X^\mu(t) + \langle \beta_t, X^\mu \rangle &+& \varepsilon(t) \\
         &=& \mathcal{L}(X^\mu)(t) &+& \varepsilon(t),
\end{array}
\end{equation}
where
$
    \mathcal{L}(X^\mu)(t)
    :=
    \alpha(t)X^\mu(t)+\langle\beta_t,X^\mu\rangle
$
denotes the linear model operator.
Centering is, of course, without loss of information, since the intercept parameter can be recovered as
$
\alpha_0(t)
=
\E(Y(t))-\alpha(t)\E(X(t))-\langle\beta_t,\E(X)\rangle
$.

Let
$
\Gamma_X(x)(t)
=
\E(\langle X^\mu,x\rangle X^\mu(t))
=
\int_0^1 \sigma_X(t,s)x(s)\,ds
$
denote the covariance operator of $X$, where
$
    \sigma_X(t,s)=\E(X^\mu(t)X^\mu(s))
$
is the covariance kernel.
We write $\lambda_1\geq\lambda_2\geq\cdots>0$ for the positive eigenvalues of $\Gamma_X$ and $\{\psi_1,\psi_2,\ldots\}$ for corresponding orthonormal eigenfunctions spanning $\operatorname{Ker}(\Gamma_X)^\perp$.
This sequence is finite if $\Gamma_X$ has finite rank and infinite otherwise.
Throughout this section, measurability of deterministic functions on $[0,1]$ or $\Delta$ is understood with respect to the corresponding Borel $\sigma$-fields.
Our identification theory is developed under the following setup.

\begin{I-assumption}\
\begin{enumerate}[label=\textbf{I.\arabic*}]
    \item\label{I-ass:rvs}
    $X$ and $Y$ are jointly measurable stochastic processes on $[0,1]$ with sample paths in $L^2[0,1]$.
\item\label{I-ass:X-structure}
    The mean function $\E(X(t))$ is continuous on $[0,1]$, and the covariance kernel
    $\sigma_X(t,s)=\operatorname{Cov}(X(t),X(s))$ exists, is continuous on $[0,1]^2$, and satisfies $\sigma_X(t,t)>0$ for all $t\in[0,1]$.
    \item\label{I-ass:param}
    The parameter functions $\alpha_0,\alpha:[0,1]\to\mathbb R$ and $\beta:\Delta\to\mathbb R$ are measurable and satisfy $\int_0^1 \alpha_0(t)^2\,dt<\infty$, $\sup_{t\in[0,1]}|\alpha(t)|<\infty$, and $\int_\Delta \beta(t,s)^2\,ds\,dt<\infty$.
    Moreover, for every $t\in[0,1]$, $\beta_t\in\operatorname{Ker}(\Gamma_X)^\perp\subseteq L^2[0,1]$.
    \item\label{I-ass:eps-structure}
    The error process satisfies $\E(\varepsilon(t)\mid X)=0$ and $\E(\varepsilon(t)^2)<\infty$ for all $t\in[0,1]$.
\end{enumerate}
\end{I-assumption}

Assumption~\ref{I-ass:rvs} allows us to use $X$ and $Y$ simultaneously as stochastic processes with well-defined point evaluations and as $L^2[0,1]$-valued random elements.
More explicitly, joint measurability means that $(t,\omega)\mapsto X(t,\omega)$ and $(t,\omega)\mapsto Y(t,\omega)$ are measurable with respect to the product $\sigma$-field on $[0,1]\times\Omega$.
This distinction is important because point evaluations are not defined for generic elements of $L^2[0,1]$, which are equivalence classes of functions.
Joint measurability provides a standard way to connect the process representation with the Hilbert-space representation (see Section~7.4 of \citealt{Hsing_2015_Book}).
A convenient sufficient condition for Assumption~\ref{I-ass:rvs} is that $X$ and $Y$ have continuous sample paths (see Theorem~7.4.2 of \citealt{Hsing_2015_Book}).

Assumption~\ref{I-ass:X-structure} ensures that the covariance operator $\Gamma_X$ is well defined and that Mercer's theorem applies to the eigenvalue-eigenfunction expansion introduced above.
Moreover, the Karhunen-Loève expansion of $X$ converges in mean square uniformly in $t\in[0,1]$.

The restriction $\beta_t\in\operatorname{Ker}(\Gamma_X)^\perp$ in Assumption~\ref{I-ass:param} removes components of the historical coefficient that are not identifiable from the distribution of $X$.
Indeed, if $x\in\operatorname{Ker}(\Gamma_X)$, then
$
    \operatorname{Var}(\langle X,x\rangle)
    =
    \langle x,\Gamma_X(x)\rangle
    =
    0
$,
so the coordinate $\langle\beta_t,x\rangle$ has no effect on the regression.
This restriction is the functional analogue of eliminating linearly dependent columns in a finite-dimensional regression model with a rank-deficient design matrix \citep[see, for instance,][]{CardotMasSarda_2007}.

The measurability and integrability conditions in Assumptions~\ref{I-ass:rvs} and \ref{I-ass:param} ensure that the right-hand side of model~\eqref{eq:operator} is jointly measurable and has sample paths in $L^2[0,1]$.
Consequently, $\varepsilon$ is also jointly measurable with sample paths in $L^2[0,1]$.
If, in addition, $X$ and $Y$ are $C[0,1]$-valued, $\alpha_0,\alpha\in C[0,1]$, and $\beta\in C(\Delta)$, then the map
$
    t\mapsto \int_0^t \beta(t,s)X(s)\,ds
$
is continuous for every continuous sample path of $X$, and hence $\varepsilon$ is also $C[0,1]$-valued.

Finally, Assumption~\ref{I-ass:eps-structure} implies $\E(\varepsilon(t))=0$ and $\E(Y(t)^2)<\infty$ for all $t\in[0,1]$.
Moreover, the exogeneity condition $\E(\varepsilon(t)\mid X)=0$ yields
$
    \E(Y^\mu(t)\mid X)=\mathcal L(X^\mu)(t)
$,
which allows us to interpret $\mathcal L$ as a functional linear regression operator.

\begin{remark}[RKHS representation of $\mathcal L$]\label{rem:RKHSrep}
Under Assumptions~\ref{I-ass:rvs}--\ref{I-ass:eps-structure}, the functional linear regression operator $\mathcal L$ can be represented using the covariance RKHS of $X$.
Let $H$ denote the RKHS with reproducing kernel $\sigma_X(t,s)=\sum_k\lambda_k\psi_k(t)\psi_k(s)$.
Thus, $H=\{f\in L^2[0,1]:\sum_k\lambda_k^{-1}\langle f,\psi_k\rangle^2<\infty\}$ with inner product $\langle f,g\rangle_H=\sum_k\lambda_k^{-1}\langle f,\psi_k\rangle\langle g,\psi_k\rangle$ and induced norm $||f||_H^2=\langle f,f \rangle_H$.
For fixed $t\in[0,1]$, define $\omega_t=\sum_k\lambda_k(\alpha(t)\psi_k(t)+\langle\beta_t,\psi_k\rangle)\psi_k$.
Then $\omega_t\in H$, since $\|\omega_t\|_H^2\leq 2\alpha(t)^2\operatorname{Var}(X(t))+2\lambda_1\|\beta_t\|^2<\infty$, and we have
$\mathcal L(X^\mu)(t)=\langle\omega_t,X^\mu\rangle_H$.
\end{remark}

The RKHS representation in Remark~\ref{rem:RKHSrep} gives a compact joint representation of the concurrent and historical coefficients through a single RKHS-valued parameter.
For interpretation, however, such a joint representation is not sufficient.
The central question is whether the concurrent parameter $\alpha(t)$ and the historical parameter $\beta_t$ in
$
\mathcal{L}(X^\mu)(t)=\alpha(t)X^\mu(t)+\langle\beta_t,X^\mu\rangle
$
are separately identified for every $t\in[0,1]$.
Indeed, there may exist a coefficient function $\breve\beta_t\in\operatorname{Ker}(\Gamma_X)^\perp\subseteq L^2[0,1]$ such that
$
\alpha(t)X^\mu(t)=\langle\breve\beta_t,X^\mu\rangle
$.
In that case,
\[
\mathcal L(X^\mu)(t)
=
\alpha(t)X^\mu(t)+\langle\beta_t,X^\mu\rangle
=
\langle \breve\beta_t+\beta_t,X^\mu\rangle,
\]
and the concurrent effect can be absorbed into the historical component.
The following lemma characterizes exactly when such an ambiguous representation is possible.

\begin{lemma}\label{lem:identification}
Let Assumptions~\ref{I-ass:rvs}--\ref{I-ass:eps-structure} hold true.
For any $t\in[0,1]$ with $\alpha(t)\neq 0$, the following two statements are equivalent:
\begin{itemize}
    \item[(i)] There exists a function $\breve\beta_t\in\operatorname{Ker}(\Gamma_X)^\perp\subseteq L^2[0,1]$ such that
    $\alpha(t)X^\mu(t)=\langle\breve\beta_t,X^\mu\rangle$;
    \item[(ii)] $\sum_{k=1}^\infty \psi_k(t)^2<\infty$.
\end{itemize}
\end{lemma}

The condition $\alpha(t)\neq0$ in Lemma~\ref{lem:identification} is not restrictive for this characterization.
If $\alpha(t)=0$, the concurrent component is absent, and the ambiguity cannot arise.
Motivated by Lemma~\ref{lem:identification}, we impose the following condition to rule out such ambiguous representations.

\begin{I-assumption-cont}\
    \begin{enumerate}[label=\textbf{\textrm{I.5}}]
        \item\label{I-ass:eigenfct} $\sum_{j=1}^\infty \psi_j(t)^2=\infty$ for all $t\in[0,1]$.
    \end{enumerate}
\end{I-assumption-cont}

Assumption~\ref{I-ass:eigenfct} rules out finite-rank predictor processes, and this exclusion is natural in the present problem.
If $X$ has finite rank, then each point evaluation $X(t)$ is a finite linear combination of the principal component scores and can therefore be represented by an $L^2$ coefficient whenever the corresponding eigenfunctions are evaluated at $t$.
In that case, a concurrent point-evaluation effect can be absorbed into the historical component, and the separate decomposition in model~\eqref{eq:combined} is not identifiable.

For infinite-rank predictors, Assumption~\ref{I-ass:eigenfct} is often natural.
The following lemma makes this precise by showing that the condition is automatic on intervals where the eigenfunction expansion converges pointwise to smooth functions.

\begin{lemma}\label{lem:assident}
Let Assumptions~\ref{I-ass:rvs}--\ref{I-ass:eps-structure} hold true.
Assume additionally that $\operatorname{Ker}(\Gamma_X)=\{0\}$ and that, for some interval $\mathcal I\subseteq[0,1]$, $\lim_{k\to\infty}\sum_{j=1}^k \langle f,\psi_j\rangle\psi_j(t)=f(t)$ for every continuously differentiable function $f\in C^1[0,1]$ and every $t\in\mathcal I$.
Then $\sum_{j=1}^\infty\psi_j(t)^2=\infty$ for every $t\in\mathcal I$.
\end{lemma}

\noindent
We can now state the main identification result.

\begin{theorem}[Identification]\label{th:identification}
Let Assumptions~\ref{I-ass:rvs}--\ref{I-ass:eigenfct} hold true and let $t\in[0,1]$.
Then, for any $\breve\alpha(t)\in\mathbb R$ and any
$
\breve\beta_t
=
\{\breve\beta(t,s):s\in[0,1],\;\breve\beta(t,s)=0\;\text{for all }s>t\}
\in\operatorname{Ker}(\Gamma_X)^\perp
$,
we have
\[
\E\Big(
(
\mathcal L(X^\mu)(t)
-
\breve\alpha(t)X^\mu(t)
-
\langle\breve\beta_t,X^\mu\rangle
)^2
\Big)>0
\]
whenever $\breve\alpha(t)\neq\alpha(t)$ and/or $\|\breve\beta_t-\beta_t\|>0$.
\end{theorem}

Theorem~\ref{th:identification} shows that $\alpha(t)$ and $\beta_t$ are separately identified as the unique population minimizers under squared loss.
This motivates a least-squares estimation strategy.
The identification, however, is nontrivial and leads to a nonstandard estimation problem, which we address in Section~\ref{sec:estimation}.

\begin{remark}
Under additional continuity assumptions on $\alpha$ and $\beta$, the Supplementary Material provides a uniform version of Theorem~\ref{th:identification}.
In particular, it gives conditions under which
$
\inf_{t\in[0,1]} \E(
(
\mathcal L(X^\mu)(t)
-
\breve\alpha(t)X^\mu(t)
-
\langle\breve\beta_t,X^\mu\rangle
)^2
)>0
$
whenever $\breve\alpha(t)\neq\alpha(t)$ and/or $\|\breve\beta_t-\beta_t\|>0$ uniformly over $t$.
While this uniform version is not needed for the estimation theory below, it may be useful in other settings.
\end{remark}

To construct estimators for the separately identified components, it is useful to derive separate representations of $\alpha(t)$ and $\beta_t$.
For this purpose, we decompose the predictor function $X^\mu=\{X^\mu(s):s\in[0,1]\}$ into the concurrent component $X^\mu(t)$ and its orthogonal residual component.
For a fixed $t\in[0,1]$, let
\[
\delta_t=\{\delta_t(s):s\in[0,1]\}
\]
denote the residual from the population regression of $X^\mu(s)$ on $X^\mu(t)$:
\begin{align}\label{eq:orthogonalize}
     X^\mu(s)=v_t(s)X^\mu(t)+\delta_t(s),
     \qquad s\in[0,1],
\end{align}
where
\[
v_t(s)
:=
\frac{\sigma_X(t,s)}{\sigma_X(t,t)}
=
\frac{\E(X^\mu(t)X^\mu(s))}{\E(X^\mu(t)^2)} .
\]
By construction, $v_t(t)=1$, $\delta_t(t)=0$, and $\E(X^\mu(t)\delta_t(s))=0$ for all $s\in[0,1]$.

Substituting \eqref{eq:orthogonalize} into \eqref{eq:operator} yields the orthogonalized representation
\begin{equation}\label{eq:orthogonalization}
\begin{array}{rcccccl}
Y^\mu(t) &=& \alpha(t)X^\mu(t) &+& \langle\beta_t,X^\mu\rangle &+& \varepsilon(t)\\
         &=& \alpha(t)X^\mu(t)+\langle\beta_t,v_t\rangle X^\mu(t) &+& \langle\beta_t,\delta_t\rangle &+& \varepsilon(t)\\
         &=& \alpha^\star(t)X^\mu(t) &+& \langle\beta_t,\delta_t\rangle &+& \varepsilon(t),
\end{array}
\end{equation}
where
\begin{align}\label{eq:alphastar}
    \alpha^\star(t)
    :=
    \alpha(t)+\langle\beta_t,v_t\rangle
    =
    \frac{\sigma_{XY}(t,t)}{\sigma_X(t,t)}
\end{align}
and $\sigma_{XY}(t,s):=\E(X^\mu(t)Y^\mu(s))$ denotes the cross-covariance kernel of $X$ and $Y$.

\begin{remark}
Equation~\eqref{eq:alphastar} shows that the coefficient estimated in the purely concurrent model~\eqref{eq:conc} is generally $\alpha^\star(t)$ rather than $\alpha(t)$.
Thus, when the historical component is nonzero, the concurrent-only coefficient combines the direct concurrent effect with the part of the historical effect that is linearly associated with $X(t)$.
\end{remark}

For any fixed $t\in[0,1]$, the historical coefficient function $\beta_t$ can be represented as the slope function in the scalar-on-function regression
\begin{align}\label{eq:betat}
    Y^\mu(t)=\langle\beta_t,\delta_t\rangle+\tilde\varepsilon(t),
\end{align}
where $\tilde\varepsilon(t):=\alpha^\star(t)X^\mu(t)+\varepsilon(t)$.
Indeed, Assumption~\ref{I-ass:eps-structure} gives $\E(\varepsilon(t)|X^\mu)=0$, while the construction of $\delta_t$ gives $\E(X^\mu(t)\delta_t(s))=0$ for all $s\in[0,1]$.
Consequently,
\[
\E(\delta_t(s)\tilde\varepsilon(t))=0
\qquad\text{for all }s\in[0,1],
\]
so the slope in \eqref{eq:betat} is the historical coefficient function of model~\eqref{eq:combined}.
The concurrent coefficient can then be written as
\begin{align}\label{eq:alpha}
\alpha(t)
=
\alpha^\star(t)-\langle\beta_t,v_t\rangle
=
\frac{\sigma_{XY}(t,t)}{\sigma_X(t,t)}-\langle\beta_t,v_t\rangle .
\end{align}

We finally derive some basic properties of the orthogonalized process.
A direct calculation shows that the covariance function
$\sigma_{\delta,t}(u,w)=\E(\delta_t(u)\delta_t(w))$ is
\[
\sigma_{\delta,t}(u,w)
=
\sigma_X(u,w)-v_t(u)\sigma_X(t,t)v_t(w),
\qquad u,w\in[0,1].
\]
In particular, since $\delta_t(t)=0$, we have $\sigma_{\delta,t}(t,w)=\sigma_{\delta,t}(u,t)=0$ for all $u,w\in[0,1]$.
Let $\Gamma_{\delta,t}$ denote the covariance operator corresponding to $\sigma_{\delta,t}$.
Since $v_t(u)\sigma_X(t,t)v_t(w)$ defines a positive semidefinite kernel, it follows that, for all $b\in L^2[0,1]$,
\[
\langle b,\Gamma_{\delta,t}(b)\rangle
=
\langle b,\Gamma_X(b)\rangle
-
\sigma_X(t,t)\langle b,v_t\rangle^2
\leq
\langle b,\Gamma_X(b)\rangle .
\]
The eigenvalues
$\lambda_{\delta,t,1}\geq\lambda_{\delta,t,2}\geq\cdots$ of $\Gamma_{\delta,t}$ and
$\lambda_1\geq\lambda_2\geq\cdots$ of $\Gamma_X$ therefore satisfy
\begin{align}\label{eq:eigenvaluerelation}
    \lambda_{\delta,t,j}\leq\lambda_j
    \qquad\text{for all }j\in\mathbb N.
\end{align}
This relation will be useful in the estimation theory.

\section{Estimation} \label{sec:estimation}

This section proposes estimators for the concurrent coefficient $\alpha(t)$ and the historical coefficient $\beta_t$, as defined in Section~\ref{sec:identification}, at fixed time points $t\in(0,1]$. Let $(X_1,Y_1),\ldots,(X_n,Y_n)$ denote $n$ i.i.d.\ copies of $(X,Y)$.

Our estimation strategy is based on the orthogonalized representation derived in Section~\ref{sec:identification}.
For a fixed time point $t$, the historical coefficient $\beta_t$ is not estimated from a regression on the original predictor trajectory $X$, but from the scalar-on-function regression of $Y^\mu(t)$ on the residualized process $\delta_t$ in \eqref{eq:betat}.
The key difference from a standard functional linear regression problem is that the relevant predictor $\delta_t$ depends on the unknown covariance structure of $X$ and must therefore be estimated from the data.
Once a feasible version of this orthogonalization has been constructed, we estimate $\beta_t$ using the smoothing-spline regularization of \citet{crambes2009}.
The concurrent coefficient is then recovered from the identity $\alpha(t)=\alpha^\star(t)-\langle\beta_t,v_t\rangle$ in \eqref{eq:alpha}.

In practice, the functions $X_i(s)$ and $Y_i(s)$, $s\in[0,1]$, are often observed at $p$ equidistant grid points
$0=s_1<\cdots<s_p=1$, with $s_j=(j-1)/(p-1)$ for $j=1,\ldots,p$.
For densely observed functional data, such a representation may be obtained after preprocessing or presmoothing. If $X_i$ and $Y_i$ are fully observed, $p$ may be chosen arbitrarily large. For simplicity, we assume that the concurrent time point $t$ of interest is one of the grid points, that is, $t\in\{s_1,\ldots,s_p\}$. For this fixed $t$, the corresponding historical grid points are denoted by
$s_1<\cdots<s_{p_t}=t$, where $p_t\le p$.

To introduce the feasible estimator, let
\begin{align}\label{eq:demeanedcurves}
	 X_{i}^{\bar \mu}(t) := X_i(t) - \bar{X}(t)\quad\text{and}\quad  Y_{i}^{\bar \mu}(t) := Y_i(t) - \bar{Y}(t),
\end{align}
with $\bar{X}(t) = n^{-1}\sum_{i=1}^n X_i(t)$ and $\bar{Y}(t) = n^{-1}\sum_{i=1}^n Y_i(t)$, denote the demeaned sample curves, and let
\begin{align*}
	\widehat \sigma_X(t,s) := \frac{1}{n} \sum_{i=1}^n X_{i}^{\bar \mu}(t)  X_{i}^{\bar \mu}(s)\quad\text{and}\quad
	\widehat \sigma_{XY}(t,s) := \frac{1}{n} \sum_{i=1}^n X_{i}^{\bar \mu}(t) Y_{i}^{\bar \mu}(s)
\end{align*}
denote the sample covariance kernels. The sample version of $\delta_t(s)$ is then given by
\begin{align} \label{eq:sampledelta}
	\widehat \delta_{i,t}(s) :=  X_i^{\bar \mu}(s) - \widehat v_t(s)  X_i^{\bar \mu}(t), \quad  \widehat v_t(s):= \frac{\widehat \sigma_X(t,s)}{\widehat \sigma_X(t,t)},\quad i=1, \ldots, n,
\end{align}
and the estimator of $\alpha^\star(t)$ by
\begin{align*}
	\widehat \alpha^\star(t) := \frac{\widehat \sigma_{XY}(t,t)}{\widehat \sigma_X(t,t)},\quad t\in[0,1].
\end{align*}

Given the empirically residualized predictors $\widehat\delta_{i,t}$, we estimate the nonzero part of the historical coefficient function $\beta_t$ by a smoothing-spline regularized least-squares criterion, which is determined by minimizing
\begin{equation}\label{splinedef}
\frac{1}{n}\sum_{i=1}^n \bigg( Y_i^{\bar \mu}(t) - \frac{1}{p}\sum_{j=1}^{p_t} b(s_j) \widehat{\delta}_{i,t}(s_j)\bigg)^2
+\rho\bigg(\frac{1}{p_t}\sum_{j=1}^{p_t} \pi_b^2(s_j) +\int_{0}^{t}  \Big(\frac{\partial^m b(s)}{\partial s^m}\Big)^2 \dd s\bigg)
\end{equation}
over all functions $b$ in the Sobolev space $W^{m,2}[0,t]\subset L^2[0,t]$, where $\rho>0$ denotes the smoothing parameter. Here, $\pi_b(s)=\sum_{l=1}^m\gamma_{b,l}s^{l-1}$ denotes the best possible approximation of $(b(s_1),\dots,b(s_{p_t}))$ by a polynomial of degree $m-1$, i.e., $\sum_{j=1}^{p_t} (b(s_j)-\pi_b(s_j))^2=\min_{\gamma_1,\dots,\gamma_m}\sum_{j=1}^{p_t} (b(s_j)-\sum_{l=1}^m\gamma_{l}s^{l-1})^2$.
The nonstandard term $p_t^{-1}\sum_{j=1}^{p_t} \pi_b^2(s_j)$ in the roughness penalty of \eqref{splinedef} guarantees the existence of a unique solution without any additional assumptions on the predictor functions $\widehat{\delta}_{i,t}$.

Following \cite{crambes2009}, the unique solution to the minimization problem in \eqref{splinedef} is an element of the space of natural splines of order $2m$ with knots at $s_{1},\dots, s_{p_t}$, with basis $\bm g_t(s) = (g_1(s), \ldots, g_{p_t}(s))^\top$ (see Appendix 5.8 of \citealt{Book_Eubank_1999} for a discussion of possible basis functions). Let
 \begin{align*}
	 \bm G_t = \begin{pmatrix}
		g_1(s_1) & \ldots & g_{p_t}(s_1) \\
		\vdots & & \vdots \\
		g_1(s_{p_t}) & \ldots & g_{p_t}(s_{p_t})
	\end{pmatrix}
	\quad\text{and}\quad
	\bm D_t = \begin{pmatrix}
		\widehat \delta_{1,t}(s_1) & \ldots & \widehat \delta_{1,t}(s_{p_t}) \\
		\vdots & & \vdots \\
		\widehat \delta_{n,t}(s_1) & \ldots & \widehat \delta_{n,t}(s_{p_t})
	\end{pmatrix}
\end{align*}
denote the $(p_t\times p_t)$ matrix of discretized basis functions and the $(n \times p_t)$ matrix of discretized predictor functions, and let $\bm Y_t = (Y_1^{\bar \mu}(t), \ldots, Y_n^{\bar \mu}(t))^\top$ denote the vector of centered response variables. Moreover, consider the $(p_t \times p_t)$ matrix $\bm A_{m,t} := \bm P_{m,t} + p_t \bm A_{m,t}^*$, where $\bm P_{m,t}$ denotes the matrix projecting into the space of all (discretized) polynomials of
degree $m-1$, i.e.
$\bm P_{m,t} = \bm W_{m,t} (\bm W_{m,t}^\top\bm W_{m,t})^{-1} \bm W_{m,t}^\top$ with
\begin{align*}
	 \bm W_{m,t} &= \begin{pmatrix}
	1 & s_1 & \ldots & s_1^{m-1} \\
	\vdots & \vdots & & \vdots \\
	1 & s_{p_t} & \ldots & s_{p_t}^{m-1}
\end{pmatrix},
\end{align*}
and where $\bm A_{m,t}^*$ denotes the regularization matrix
\begin{align*}
	\bm A_{m,t}^* := \bm G_t (\bm G_t^\top \bm G_t)^{-1} \left(\int_0^{t} \left( \frac{\partial^m \bm g_t(s)}{\partial s^m} \right) \left( \frac{\partial^m \bm g_t(s)}{\partial s^m} \right)^\top \dd s\right) (\bm G_t^\top \bm G_t)^{-1} \bm G_t^\top.
\end{align*}
Finally, note that for any vector $\bm{w}_t = (w(s_1),\dots,w(s_{p_t}))^\top\in\mathbb{R}^{p_t}$, there exists a unique natural spline interpolant $s_{\bm w_t}(s)=\bm g_t(s)^\top(\bm G_t^\top \bm G_t)^{-1}\bm G_t^\top\bm w_t$ that interpolates $\bm w_t$ at the grid points $s_1,\dots,s_{p_t}$.

Minimizing the objective function in \eqref{splinedef} over all functions $b$ in the Sobolev space $W^{m,2}[0,t]$ is then equivalent to the following minimization problem
\begin{align*}
\widehat{\mathbf b}_t=\argmin_{\mathbf{b} \in \mathbb R^{p_t}} \bigg\lbrace \frac{1}{n} \Big(\bm Y_t - \frac{1}{p} \bm D_t \mathbf{b} \Big)^\top\Big(\bm Y_t - \frac{1}{p} \bm D_t \mathbf{b}\Big) + \frac{\rho}{p_t} \mathbf{b}^\top \bm A_{m,t} \mathbf{b} \bigg\rbrace,
\end{align*}
which can be solved if there are at least $m$ distinct grid points in $[0,t]$. Then, \begin{align}\label{eq:betahat}
\widehat{\mathbf b}_t  = \frac{1}{n} \bigg( \frac{1}{np} \bm D_t^\top \bm D_t + \rho \frac{p}{p_t} \bm A_{m,t} \bigg)^{-1} \bm D_t^\top \bm Y_t.
\end{align}
The solution to the minimization problem \eqref{splinedef} is then given by $s_{\widehat{\mathbf b}_t}=\{s_{\widehat{\mathbf b}_t}(s):s\in[0,t]\}$ which allows us to define the smoothing spline estimator for $\beta_t$ as
$$
\widehat \beta_t(s) := \begin{cases}
	s_{\widehat{\mathbf{b}}_t}(s)& \text{if} \ 0 \leq s \leq t \ \text{and}\ t>(m-1)/(p-1) \\
		0 & \text{otherwise}.
	\end{cases}
$$
To select the smoothing parameter $\rho>0$ in practice, the generalized cross-validation (GCV) estimator proposed by \cite{crambes2009} can be used.

To estimate the concurrent parameter function $\alpha(t)$ we use the plug-in estimator
$$
\widehat \alpha(t) = \widehat \alpha^{\star}(t)-\langle \widehat \beta_t,  \widehat v_t\rangle = \frac{\widehat \sigma_{XY}(t,t)}{\widehat \sigma_X(t,t)} - \int_0^t \widehat \beta_t(s) \frac{\widehat \sigma_X(t,s)}{\widehat \sigma_X(t,t)} \dd s,
$$
and an estimator for the intercept function $\alpha_0(t)=E(Y_i(t))-\alpha(t)E(X_i(t))- \langle \beta_t, E(X_i)\rangle$ is given by
$$
\widehat\alpha_0(t)=\bar Y(t) -\widehat\alpha(t)\bar X(t)-\langle \widehat \beta_t, \bar X \rangle.
$$

Note that $\beta_t$ is identified in the population regression \eqref{eq:betat} based on the infeasible residual process $\delta_t$ while our implementation replaces $v_t$ and $\delta_{t}$ by their empirical counterparts $\widehat v_t$ and $\widehat \delta_{i,t}$ given in \eqref{eq:sampledelta}.
This substitution enforces the sample analogue of the orthogonality property $\E(X^\mu(t)\delta_t(s))=0$. Indeed, for every $s\in[0,1]$,
$$
\frac{1}{n}\sum_{i=1}^n X_i^{\bar\mu}(t)\,\widehat\delta_{i,t}(s)
=\widehat\sigma_X(t,s)-\widehat v_t(s)\widehat\sigma_X(t,t)=0,
$$
and hence also $\frac{1}{n}\sum_{i=1}^n X_i^{\bar\mu}(t)\,\widehat\delta_{i,t}(s_j)=0$ for all grid points $s_j\le t$.
Consequently, the empirical version of the predictor decomposition \eqref{eq:orthogonalize} is
$$
X_i^{\bar\mu}(s)=\widehat v_t(s)\,X_i^{\bar\mu}(t)+\widehat\delta_{i,t}(s),
$$
and the model for the centered responses can be written as
\begin{align*}
Y_i^{\bar \mu}(t)
&= \alpha(t) X_i^{\bar \mu}(t) + \langle \beta_t, X_i^{\bar \mu} \rangle + \varepsilon_i^{\bar \mu}(t)\\
&=\widetilde \alpha^\star(t) X_i^{\bar \mu}(t)
+\langle \beta_t,\widehat \delta_{i,t} \rangle + \varepsilon_i^{\bar \mu}(t),
\end{align*}
where $\varepsilon_i^{\bar \mu}(t):=\varepsilon_i(t)-n^{-1}\sum_{j=1}^n\varepsilon_j(t)$ and $\widetilde \alpha^\star(t):=\alpha(t)+\langle \beta_t,\widehat v_t\rangle$.
Define the infeasible partialled-out response
$$
\widetilde Y_i^{\bar\mu}(t):= Y_i^{\bar\mu}(t)-\widetilde\alpha^\star(t)X_i^{\bar\mu}(t)
= \langle \beta_t,\widehat\delta_{i,t}\rangle+\varepsilon_i^{\bar\mu}(t),
$$
so that the relevant infeasible scalar-on-function regression is
\begin{align}\label{eq:infeasiblemodel}
\widetilde Y_i^{\bar\mu}(t) = \langle \beta_t,\widehat\delta_{i,t}\rangle+\varepsilon_i^{\bar\mu}(t).
\end{align}
Importantly, using the observable response $Y_i^{\bar\mu}(t)$ in \eqref{splinedef} yields exactly the same smoothing-spline estimator as using the unobservable $\widetilde Y_i^{\bar\mu}(t)$.
To see this, observe that the first term of the objective function \eqref{splinedef} satisfies
\begin{align*}
&\frac{1}{n}\sum_{i=1}^n \bigg( Y_i^{\bar \mu}(t) - \frac{1}{p}\sum_{j=1}^{p_t} b(s_j)\widehat \delta_{i,t}(s_j))\bigg)^2 -
 \frac{1}{n}\sum_{i=1}^n \bigg( \widetilde Y_i^{\bar \mu}(t) - \frac{1}{p}\sum_{j=1}^{p_t} b(s_j)\widehat \delta_{i,t}(s_j)) \bigg)^2 \\
 &=\frac{1}{n}\sum_{i=1}^n  (\widetilde \alpha^\star(t) X_i^{\bar \mu}(t))^2 +
  \frac{2}{n} \sum_{i=1}^n \widetilde{\alpha}^\star(t) X_i^{\bar \mu}(t) \bigg( \widetilde Y_i^{\bar \mu}(t) - \frac{1}{p}\sum_{j=1}^{p_t} b(s_j)\widehat \delta_{i,t}(s_j)) \bigg)  \\
 &=\frac{1}{n}\sum_{i=1}^n  (\widetilde \alpha^\star(t) X_i^{\bar \mu}(t))^2 +
  \frac{2}{n} \sum_{i=1}^n \widetilde{\alpha}^\star(t) X_i^{\bar \mu}(t) \varepsilon_i^{\bar \mu}(t)
\end{align*}
since $\sum_{i=1}^n X_i^{\bar \mu}(t) \widehat \delta_{i,t}(s)=0$ for all $0\leq s\leq 1$.
Note that the two terms on the right-hand side of this equation do not depend on $b(\cdot)$.
Therefore, minimizing \eqref{splinedef} with response $Y_i^{\bar\mu}(t)$ is equivalent to minimizing the same criterion with response $\widetilde Y_i^{\bar\mu}(t)$, and with
$\widetilde{\bm Y}_t=(\widetilde{Y}_{1}^{\bar \mu}(t),\dots,\widetilde{Y}_{n}^{\bar \mu}(t))^\top$ we have the following useful equivalent expressions for our estimator $\widehat{\mathbf{b}}_t$ in \eqref{eq:betahat}:
$$
	\widehat{\mathbf{b}}_t  = \frac{1}{n} \bigg( \frac{1}{np} \bm D_t^\top \bm D_t + \rho \frac{p}{p_t} \bm A_{m,t} \bigg)^{-1} \bm D_t^\top \bm Y_t=\frac{1}{n} \bigg( \frac{1}{np} \bm D_t^\top \bm D_t + \rho \frac{p}{p_t} \bm A_{m,t} \bigg)^{-1} \bm D_t^\top  \widetilde{ \bm Y}_t.
$$
Thus, $\widehat\beta_t$ can be viewed as the smoothing-spline estimator in the infeasible regression \eqref{eq:infeasiblemodel}.

\section{Asymptotic Results} \label{sec:consistency}

Our theoretical results are established under Hölder-type path regularity assumptions on the sample paths of the predictor variable $X$.
For any $d\in \mathbb N_0$ and $\kappa\in(0,1]$, a function $f\colon[0,1]\to\mathbb R$ is called $(d,\kappa)$-Hölder continuous if its $\ell$-th derivative $f^{(\ell)}$ exists and is continuous for $\ell = 0, \ldots, d$, and if there exists $C < \infty$ such that its $d$-th derivative satisfies
$$
	|f^{(d)}(t)-f^{(d)}(s)| \leq C |t-s|^\kappa \quad \text{for all} \ t,s \in [0,1].
$$
The $(d,\kappa)$-Hölder norm is defined by
$$
	\|f\|_{d,\kappa} := \max_{0\leq \ell \leq d} \|f^{(\ell)}\|_\infty + [f^{(d)}]_\kappa,
$$
where
$$
	\|f^{(\ell)}\|_\infty := \sup_{t \in [0,1]} |f^{(\ell)}(t)|, \qquad [f^{(d)}]_\kappa := \sup_{s \neq t} \frac{|f^{(d)}(t)-f^{(d)}(s)|}{|t-s|^{\kappa}}.
$$

Hölder regularity is used repeatedly in the asymptotic analysis. In particular, it implies bounds on the eigenvalue decay of covariance operators. The following lemma describes this connection and also gives the corresponding bound for the covariance operator of the orthogonalized predictor.

\begin{lemma}\label{lem:eigenvaluedecay}
If $\E(\|X\|_{d,\kappa}^2) < \infty$, then the eigenvalues $(\lambda_k)_{k\geq1}$ of $\Gamma_X$ satisfy
$$
\lambda_k = O\big(k^{-(2d+\kappa+1)}\big),
\qquad \text{as} \ k\to\infty.
$$
Moreover, for each fixed $t\in[0,1]$, the eigenvalues $(\lambda_{\delta,t,k})_{k\geq1}$ of $\Gamma_{\delta,t}$ satisfy
$$
\lambda_{\delta,t,k}
=
O\big(k^{-(2d+\kappa+1)}\big),
\qquad \text{as} \ k\to\infty.
$$
\end{lemma}

Compared to classical functional regression problems \citep[cf.][]{Yao_Mueller_Wang_2005,Hall_Horowitz_2007,crambes2009}, our estimation approach faces the additional difficulty that the orthogonalized scalar-on-function regression in \eqref{eq:betat} involves an unobserved functional predictor. The ideal predictor is
$$
\delta_{i,t}(s)
=
X_i^\mu(s)-v_t(s)X_i^\mu(t),
$$
whereas the feasible estimator uses the plug-in version
$$
\widehat\delta_{i,t}(s)
=
X_i^{\bar\mu}(s)-\widehat v_t(s)X_i^{\bar\mu}(t).
$$
Thus, the asymptotic analysis has to relate the feasible residualized predictors $\{\widehat\delta_{i,t}\}_{i=1}^n$ to their infeasible population counterparts $\{\delta_{i,t}\}_{i=1}^n$. The following lemma shows that the plug-in residuals inherit the relevant Hölder regularity from the original predictor trajectories.

\begin{lemma}\label{lem:deltaHoelderproperties}
Let $X_1,\ldots,X_n$ be i.i.d.\ copies of $X$.
Suppose that $\E(\|X\|_{d,\kappa}^6)<\infty$ and that the variance of $X$ is uniformly bounded away from zero, i.e.\
$\inf_{s\in[0,1]}\sigma_X(s,s)>0$.
Then $\E(\|\widehat\delta_{i,t}\|_{d,\kappa}^2)<\infty$ for every $i=1,\ldots,n$ and $t\in[0,1]$.
\end{lemma}

For the asymptotic analysis of our parameter estimators, we require additional assumptions.
These assumptions are used for the rate results below and are stronger than the assumptions needed for identification in Section~\ref{sec:identification}.
In particular, they impose further regularity conditions on the sample paths, moments, parameter functions, and on the eigenvalue and eigenfunction behavior of the covariance operator of $X$.

\begin{A-assumption}\
	\begin{enumerate}[label=\textbf{A.\arabic*}]
\item\label{A-ass:rvs} $X$ and $Y$ are $C[0,1]$-valued random variables.

\item \label{A-ass:X-structure}  Assume that there exist $d\in\mathbb N_0$ and $\kappa\in(0,1]$ such that
$\E(\|X\|_{d,\kappa}^6)<\infty$ and $\E(\|Y\|_{d,\kappa}^{3})<\infty$.
Moreover, $\inf_{t\in[0,1]}\sigma_X(t,t)>0$.
Finally, there exists a constant $C<\infty$ such that, for every
$t\in[0,1]$ and all $g,h\in L^2[0,1]$,
\begin{align}\label{eq:kurtosis-ass-X-adapt}
&\E\Big[
\big(\langle X^\mu,g\rangle-X^\mu(t)\langle v_t,g\rangle\big)^2
\big(\langle X^\mu,h\rangle-X^\mu(t)\langle v_t,h\rangle\big)^2
\Big]  \notag\\
&\qquad\le
C\,
\E\Big[
\big(\langle X^\mu,g\rangle-X^\mu(t)\langle v_t,g\rangle\big)^2
\Big]\,
\E\Big[
\big(\langle X^\mu,h\rangle-X^\mu(t)\langle v_t,h\rangle\big)^2
\Big].
\end{align}

\item\label{A-ass:param}
Let  $\alpha_0,\alpha\in C[0,1]$ and $\beta\in C(\Delta)$ where $\Delta = \{(t,s): 0 \leq s \leq t \leq 1\}$. Let $m\geq 2$ be the smoothing spline order. For each $t\in[0,1]$, the function $s\mapsto \beta(t,s)$ is $m$-times differentiable on $[0,t]$ with $\beta^{(m)}(t,\cdot) \in L^2[0,t]$ and $\beta(t,\cdot)\in \operatorname{Ker}(\Gamma_X)^\perp$.

\item\label{A-ass:eps-structure} $\varepsilon$ is independent of $X$ and satisfies $\E(\varepsilon(t)) = 0$ for all $t \in [0,1]$.

\item\label{A-ass:eigenfct}
The eigenvalues of $\Gamma_X$ satisfy $\lambda_k=O(k^{-2})$. Moreover,
there exist $k_0\ge1$, $C_\lambda>0$, and $\tau\ge 2d+\kappa+1$ such that $\lambda_k\ge C_\lambda k^{-\tau}$ for all $k\ge k_0$.
The mean squared eigenfunctions are bounded away from zero such that for each fixed $t\in[0,1]$, there exist $k_{0,t}\ge1$ and $C_{\psi,t}>0$ such that $\frac{1}{k}\sum_{j=1}^k\psi_j(t)^2\ge C_{\psi,t}$ for all $k\ge k_{0,t}$.
\item\label{A-ass:iid}
The observations $(X_i,Y_i)_{i=1,\dots,n}$ are i.i.d.\ copies of $(X,Y)$.
\end{enumerate}
\end{A-assumption}

Assumptions \ref{A-ass:rvs}--\ref{A-ass:eigenfct} together imply Assumptions \ref{I-ass:rvs}--\ref{I-ass:eigenfct}.
Assumption \ref{A-ass:rvs} implies joint measurability of $X$ and $Y$ by Theorem 7.4.2 in \cite{Hsing_2015_Book} and, together with Assumption \ref{A-ass:X-structure}, controls the smoothness of the sample paths.
Our framework accommodates both smooth paths ($d$ large) and rough paths ($d=0$ with small $\kappa$).
The moment bound in the Hölder norm directly implies $\sup_{t \in [0,1]} \E(|X^{(\ell)}(t)|^6) < \infty$ for $0 \leq \ell \leq d$ and ensures that the pathwise Hölder constant
$$
	K_X := \sup_{s \neq t} \frac{|X^{(d)}(t)-X^{(d)}(s)|}{|t-s|^{\kappa}}
$$
satisfies $\E(K_X^6) < \infty$.
The additional variance-positivity condition excludes degenerate predictors.
Condition \eqref{eq:kurtosis-ass-X-adapt} is a uniform kurtosis-type condition
for the linear functionals of $X$ that arise after orthogonalizing $X$ with
respect to its value $X(t)$. It requires products of squared orthogonalized
linear functionals to be bounded by the product of their second moments. In the
special case $g=h$, this reduces to the usual fourth-moment condition used in
functional linear regression (see, e.g., \citealt{Hall_Horowitz_2007} and \citealt{crambes2009}). The condition holds, for example, if $X$ is a
Gaussian process.
Assumption~\ref{A-ass:param} imposes the required regularity on the parameter functions.
In particular, for each fixed $t$, the nonzero part $s\mapsto \beta(t,s)$ on $[0,t]$ belongs to the Sobolev space $W^{m,2}[0,t]$ on which the smoothing-spline estimator is defined.
Assumption \ref{A-ass:eps-structure} imposes standard conditions on the regression error.
Assumption~\ref{A-ass:eigenfct} complements Assumption~\ref{I-ass:eigenfct} by specifying how slow eigenvalues need to converge and how fast the cumulative squared eigenfunctions need to diverge, which is required to derive the rate of consistency of $\widehat{\alpha}(t)$.
The upper bound $\lambda_k=O(k^{-2})$ is a technical summability condition that is automatically satisfied for continuously differentiable predictor trajectories by Lemma~\ref{lem:eigenvaluedecay}, and it is also compatible with nondifferentiable processes such as Brownian-motion-type predictors, whose covariance eigenvalues decay at order $k^{-2}$.
Thus, this condition does not by itself exclude rough functional predictors, although for very rough paths it has to be verified as an additional eigenvalue regularity condition.
Lemma~\ref{lem:eigenvaluedecay} explains the condition $\tau\ge 2d+\kappa+1$ in the lower eigenvalue bound of Assumption~\ref{A-ass:eigenfct}.
Assumption \ref{A-ass:iid} is a standard sampling condition in regression analysis.

We formulate our consistency results with respect to the $\Gamma_{\delta,t}$-specific semi-norm, which is the natural semi-norm for the scalar-on-function regression in \eqref{eq:betat} and focuses only on the identifiable components of $\beta_t$:
\begin{align*}
	\|\widehat\beta_t-\beta_t\|_{\Gamma_{\delta,t}}&=\langle\widehat\beta_t-\beta_t, \Gamma_{\delta,t}(\widehat\beta_t-\beta_t)\rangle^{1/2}  = \E\big( \langle \widehat \beta_t-\beta_t,\delta_{t}\rangle^2 \big)^{1/2}.
\end{align*}

\begin{theorem} \label{thm:consistency}
Suppose Assumptions \ref{A-ass:rvs}--\ref{A-ass:iid} hold. Let $p = p(n)$ be
an integer sequence with $p(n) \to \infty$ such that $\sqrt{n}/p(n) = O(1)$ if
$d \ge 1$ and $n^{1/(2\kappa)}/p(n) = O(1)$ if $d = 0$. Define
$\varphi = (2d + \kappa + 2m + 1)/(2d + \kappa + 2m + 2)$ and choose the
smoothing parameter $\rho = \rho(n)$ so that $\rho(n) \sim n^{-\varphi}$.
Then, as $n \to \infty$, for every fixed $0 < t \le 1$,
\begin{itemize}
  \item[(a)] $\|\widehat\beta_t - \beta_t\|_{\Gamma_{\delta,t}}
         = O_P\bigl(n^{-\varphi/2}\bigr)$;
  \item[(b)] $\widehat\alpha(t) - \alpha(t)
         = O_P\bigl(n^{-\varphi/(2\tau)}\bigr)$.
\end{itemize}
\end{theorem}

Theorem~\ref{thm:consistency} highlights a smoothness-driven trade-off between the estimation of the historical coefficient $\beta_t$ and the concurrent coefficient $\alpha(t)$. When $X$ has smoother trajectories, corresponding to larger values of $d$, the convergence rate for $\widehat\beta_t$ approaches the parametric rate, since $\varphi/2\to 1/2$ as $d\to\infty$.
At the same time, the rate for $\widehat\alpha(t)$ depends on the accuracy with which the correction term $\langle \beta_t,v_t\rangle$ in the identity $\alpha(t)=\alpha^\star(t)-\langle \beta_t,v_t\rangle$ can be estimated.
This step is controlled by the eigenvalue lower bound in
Assumption~\ref{A-ass:eigenfct}. Since Hölder regularity implies the upper
bound $\lambda_k=O(k^{-(2d+\kappa+1)})$, the lower bound
$\lambda_k\ge C_\lambda k^{-\tau}$ can hold only for exponents
$\tau\ge 2d+\kappa+1$.
Thus, smoother trajectories typically lead to a weaker rate bound for $\widehat\alpha(t)$. Conversely, for rougher predictors, such as $d=0$ with small $\kappa$, the convergence rate for $\widehat\alpha(t)$ improves, while the historical estimator $\widehat\beta_t$ converges more slowly.

We now analyze the prediction error for a new random function $X_{n+1}$ with the same distribution as $X$ and independent of the training sample $V_n:=\{(X_i,Y_i):i=1,\dots,n\}$, which is given by
$$m_t(X_{n+1}) :=E(Y(t)|X=X_{n+1})=\alpha_0(t)+\alpha(t)X_{n+1}(t)
+\langle  \beta_t,  X_{n+1}\rangle.$$

\begin{theorem}\label{thm:predictionerror}
Suppose that the conditions of Theorem \ref{thm:consistency} are satisfied, and let $n \to \infty$.
\begin{itemize}
	\item[(a)] For any fixed $0<t\leq 1$, the in-sample error satisfies:
$$
	\frac{1}{n} \sum_{i=1}^n \left(m_t(X_i) -\widehat\alpha_0(t)-\widehat\alpha(t)X_i(t)
-\langle \widehat\beta_t, X_i \rangle\right)^2
= O_P(n^{-\varphi}).
$$
	\item[(b)] For any fixed $0<t\leq 1$, the conditional out-of-sample error satisfies:
	\begin{align*}
	&\mathbb{E}\bigg(\Big(m_t(X_{n+1})-\widehat\alpha_0(t)-\widehat\alpha(t)X_{n+1}(t)
	-\langle \widehat\beta_t,X_{n+1}\rangle\Big)^2\,\Big|\,V_n\bigg)=O_P(n^{-\varphi}).
	\end{align*}
\end{itemize}
\end{theorem}

Although Theorems~\ref{thm:consistency}--\ref{thm:predictionerror} provide convergence rates for $\widehat{\beta}_t$ and $\widehat{\alpha}$, deriving non-degenerate weak limits for the full structural estimators is substantially more delicate. This is not surprising, given the ill-posed inverse nature of functional linear regression, where limiting distributions for functional coefficient estimators typically require additional structure, such as restrictions on the cross-covariance operator \citep[see][]{Crambes.2013,otto2025functional}.

We therefore state a distributional result for the well-posed concurrent-only coefficient
$$
\alpha^\star(t)=\frac{\sigma_{XY}(t,t)}{\sigma_X(t,t)}.
$$
As discussed in Section~\ref{sec:identification}, $\alpha^\star(t)$ is the coefficient in the best linear prediction of $Y(t)$ from the concurrent value $X(t)$. It generally differs from the structural concurrent coefficient $\alpha(t)$, since it also contains the part of the historical effect that is linearly associated with $X(t)$. Thus, the following result should not be interpreted as a weak limit for $\widehat\alpha$. Rather, it provides a Gaussian-process approximation for the reduced concurrent coefficient $\alpha^\star$, which underlies the simultaneous confidence bands for $\alpha^\star$ considered in Sections~\ref{sec:simulations} and~\ref{sec:applications}.

\begin{theorem}\label{thm:alphaStar}
Suppose that the conditions of Theorem \ref{thm:consistency} are satisfied, and assume in addition that there exists
$q>\max\{6,3/\tilde\kappa\}$, with
$\tilde\kappa=\kappa$ if $d=0$ and $\tilde\kappa=1$ if $d\ge1$, such that $\E(\|X\|_{d,\kappa}^{q})<\infty$ and $\E(\|Y\|_{d,\kappa}^{q/2})<\infty$.
Let $n \to \infty$.
\begin{enumerate}[label=(\alph*)]
\item\label{res:alphaStar}
The estimator $\widehat \alpha^\star(t) = \widehat \sigma_{XY}(t,t)/\widehat \sigma_X(t,t)$ satisfies
$$
\sqrt{n}\left(\widehat{\alpha}^\star - \alpha^\star\right)\to_D\mathcal{GP}(0,\sigma_{\widehat{\alpha}^\star})\quad\text{in}\quad C[0,1],
$$
with mean zero and covariance function
$$
\sigma_{\widehat{\alpha}^\star}
=\Big\{\sigma_{\widehat{\alpha}^\star}(s,t) = \frac{\sigma_Z(s,t)}{\sigma_X(s,s) \sigma_X(t,t)}: s,t\in[0,1]\Big\},
$$
where
$$
\sigma_Z(s,t) = \mathrm{Cov}(Z_i(s), Z_i(t)), \quad Z_i(t) = X_i^{\mu}(t)(Y_i^{\mu}(t) - \alpha^\star(t) X_i^{\mu}(t)).
$$
\item\label{res:UnifConsCeCXX}
The covariance function $\sigma_{\widehat{\alpha}^\star}(s,t)$ can be estimated uniformly consistently,
$$
\sup_{s,t\in[0,1]}\left|\widehat{\sigma}_{\widehat{\alpha}^\star}(s,t) - \sigma_{\widehat{\alpha}^\star}(s,t)\right|
= o_P(1),
$$
where, for $s,t \in [0,1]$,
$$
	\widehat{\sigma}_{\widehat{\alpha}^\star}(s,t)
= \frac{\widehat \sigma_Z(s,t)}{\widehat \sigma_X(s,s) \widehat \sigma_X(t,t)},
$$
with
$$
	\widehat \sigma_Z(s,t) = \frac{1}{n} \sum_{i=1}^n \widehat Z_i(s) \widehat Z_i(t), \quad \widehat Z_i(t) = X_i^{\bar \mu}(t)(Y_i^{\bar \mu}(t) - \widehat \alpha^\star(t) X_i^{\bar \mu}(t)).
$$
\end{enumerate}
\end{theorem}

\section{Simulations}\label{sec:simulations}

We generate random functions $(Y_i,X_i,\varepsilon_i)_{i=1}^n\overset{\operatorname{iid}}{\sim}(Y,X,\varepsilon)$,
$$
Y(t)=\alpha_0(t)+\alpha(t)X(t)+\int_0^t\beta(t,s) X(s) \dd s+\varepsilon(t), \quad t\in[0,1],
$$
evaluated at $p=51$ equidistant grid points $s_j=t_j=(j-1)/(p-1)$, $j=1,\dots,p$, to emulate fully observed functional data, where $X$ is a Gaussian process $X\sim\mathcal{GP}(\mu,\sigma_X)$ with mean function $\mu(t)=10 t^3 - 15 t^4 + 6 t^5$ and Mat\'ern covariance function $\sigma_X(s,t)\equiv C(s,t|\nu=3.5,\varsigma=1)$, and where $\varepsilon$ is a zero mean Gaussian process $\varepsilon\sim\mathcal{GP}(0,\sigma_\varepsilon)$ with Mat\'ern covariance function $\sigma_\varepsilon(s,t)\equiv C(s,t|\nu=2.5,\varsigma=0.1)$,
$$
C(s,t|\nu,\varsigma)=\mathbbm{1}_{(|t-s|=0)}\varsigma^2 + \mathbbm{1}_{(|t-s|>0)}\varsigma^2\big(2^{1-\nu}/\operatorname{Gam}(\nu)\big)\allowbreak \big(\sqrt{2\nu}|t-s|\big)^\nu K_\nu\big(\sqrt{2\nu}|t-s|\big),
$$
where $\operatorname{Gam}(\cdot)$ denotes the gamma function, $K_\nu$ is the modified Bessel function of the second kind, and where $\nu\geq 0$ controls the roughness of the sample paths ($\nu>1$ leads to continuously differentiable sample paths). We consider the following two data generating processes (DGP) and estimation procedures:
\begin{description}
	\item[DGP1] $\alpha_0(t)=2$, $\alpha(t)=10+5\,(10 t^3 - 15 t^4 + 6 t^5)$, $\beta(t,s)=\mathbbm{1}_{(s\leq t)}\dfrac{5 s}{t}$ for all $t\in[0,1]$
	\item[DGP2] $\alpha_0(t)=2$, $\alpha(t)=5+5\,(10 t^3 - 15 t^4 + 6 t^5)$, $\beta(t,s)=\mathbbm{1}_{(s\leq t)}\dfrac{\alpha(t) s}{t}$ for all $t\in[0,1]$
\end{description}
DGP2 is the substantially harder case since $\beta(t,s)\propto\alpha(t)$ renders the two effects maximally entangled. The parameter functions $\alpha(t)$ and $\beta(t,s)$ of DGP1 and DGP2 are shown in Figure \ref{fig:DGP1andDGP2}.

\begin{figure}[!tbp]

	\resizebox{\linewidth}{!}{\input{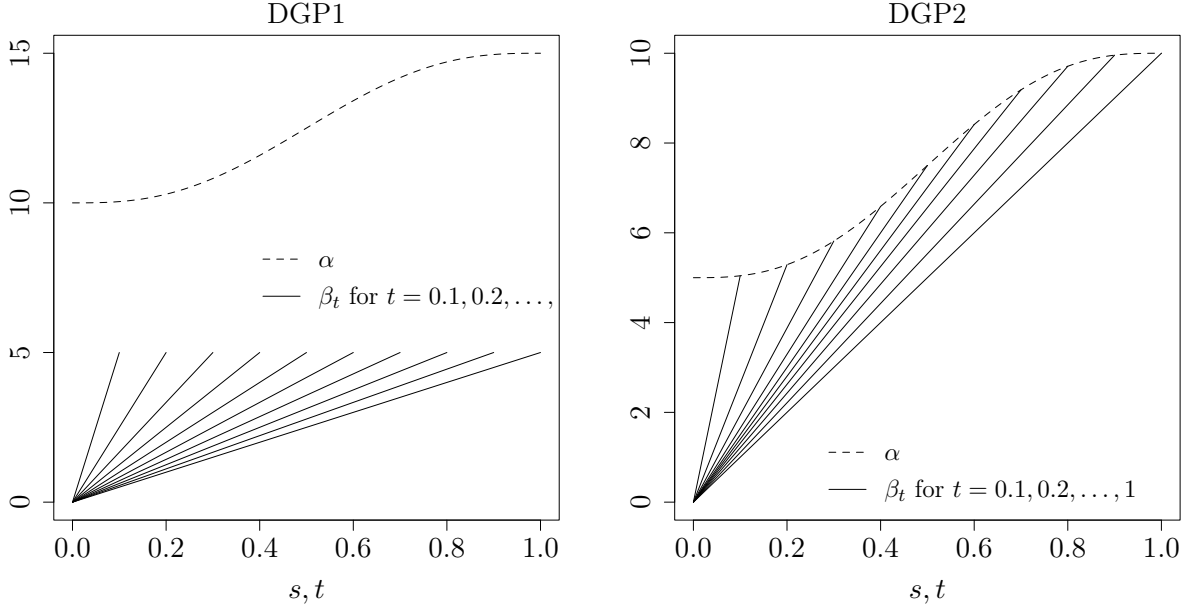}}
	\caption{Parameter functions of DGP1 and DGP2.}\label{fig:DGP1andDGP2}
\end{figure}

We consider the following three estimation procedures applied in a small sample size scenario ($n=50$) and a large sample size scenario ($n=500$):
\begin{description}
 \item[Historical] The smoothing splines approach of \cite{crambes2009} for estimating the historical functional linear regression model \eqref{eq:hist} of $Y(t_j)$ on $X(s)$, $s\leq t_j$ for each $t_j$.
 \item[Concurrent] Pointwise linear regression for estimating the concurrent functional linear regression model \eqref{eq:conc} of $Y(t_j)$ on $X(t_j)$  at each $t_j$.
 \item[Combined] Our combined concurrent and historical functional linear regression model \eqref{eq:combined} estimated using the estimation procedure described in Section \ref{sec:estimation}.
\end{description}

To evaluate the estimation procedures, we use the L2-distances between the estimates $\widehat\alpha$ and $\widehat\beta_t$ and the true concurrent and historical parameters,
$$
\|\widehat\alpha - \alpha\|^2\quad\text{and}\quad\frac{1}{|\mathcal{T}|}\sum_{t\in\mathcal{T}}\|\widehat\beta_t - \beta_t\|^2,
$$
with $\mathcal{T}=\{0.1,0.2,\dots,1\}$, and report their means and standard deviations over $1000$ Monte Carlo replications (Table~\ref{TAB:SIM_1}) for both DGPs and both sample size scenarios.

Table~\ref{TAB:SIM_1} confirms the omitted-variable-bias phenomena discussed in the introduction.
The historical model, which omits the concurrent component, yields errors $|\mathcal{T}|^{-1}\sum_{t\in\mathcal{T}}\|\widehat\beta_t - \beta_t\|^2$ on average of $1169$ and $378$ for DGP1 and DGP2 at $n=50$, with virtually no reduction at $n=500$, reflecting a persistent bias.
Similarly, the average estimation error $\|\widehat\alpha - \alpha\|^2$ of the concurrent model are at $2.0$ and $6.7$ across both sample sizes in DGP1 and DGP2.
The Combined model eliminates both biases, with estimation errors that shrink substantially from $n=50$ to $n=500$ for all parameters and both DGPs.

\begin{table}[!tbp]
\centering
\caption{Means (and standard deviations) of the L2-distances between the estimates and the true parameters over 1000 Monte Carlo replications.}
\label{TAB:SIM_1}
\begin{tabular}{rl S[table-format=4.2]@{\,}l S[table-format=4.2]@{\,}l c S[table-format=4.2]@{\,}l S[table-format=4.2]@{\,}l}
\toprule
&&\multicolumn{4}{c}{$\|\widehat{\alpha}-\alpha\|^2$}&&
 \multicolumn{4}{c}{$|\mathcal{T}|^{-1}\sum_{t\in\mathcal{T}}\|\widehat{\beta}_t-\beta_t\|^2$}\\
\cmidrule(lr){3-6}\cmidrule(lr){8-11}
&&\multicolumn{2}{c}{$n=50$}& \multicolumn{2}{c}{$n=500$}&& \multicolumn{2}{c}{$n=50$}& \multicolumn{2}{c}{$n=500$}\\
\midrule
DGP1
&  Historical  & \multicolumn{2}{c}{---} & \multicolumn{2}{c}{---}  && 1111.29 &(35.92) & 1145.26 &(11.25)\\
&  Concurrent  & 2.02 &(0.15)& 2.01 &(0.05)&& \multicolumn{2}{c}{---} & \multicolumn{2}{c}{---}\\
&  Combined    & \bfseries 0.07 &(0.05)& \bfseries 0.02 &(0.01)&& \bfseries 0.85 &(0.52) &\bfseries 0.40 &(0.19)\\
\midrule
DGP2
&  Historical  & \multicolumn{2}{c}{---} & \multicolumn{2}{c}{---}  && 362.37 &(12.41) & 390.86 &(3.86)\\
&  Concurrent  & 6.71 &(0.53)& 6.70 &(0.16)&& \multicolumn{2}{c}{---} & \multicolumn{2}{c}{---}\\
&  Combined    & \bfseries 0.08 &(0.07)& \bfseries 0.03 &(0.02)&& \bfseries 0.98 &(0.86) &\bfseries 0.43 &(0.23)\\
\bottomrule
\end{tabular}

\end{table}

Theorem \ref{thm:alphaStar} implies that the concurrent-only estimator $\widehat{\alpha}^\star(t)$ is asymptotically normal in a Gaussian process sense, and hence can be used for inference on $\alpha^\star(t)$ using the following simultaneous confidence band based on the estimated covariance function $\widehat{\sigma}_{\widehat{\alpha}^\star}(s,t)$:
\begin{equation}\label{eq:SCB}
\operatorname{SCB}^{1-\gamma}_{\alpha^\star}(t)
=
\left[
\widehat{\alpha}^\star(t)
\pm
\frac{1}{\sqrt{n-2}}\widehat q_{1-\gamma}
\right],
\qquad t\in[0,1],
\end{equation}
where $\widehat q_{1-\gamma}$ denotes the $(1-\gamma)$-quantile of
$\sup_{t\in[0,1]}|\mathcal Z(t)|$, where
$\mathcal Z\sim\mathcal{GP}(0,\widehat\sigma_{\widehat{\alpha}^\star})$
and $\widehat\sigma_{\widehat{\alpha}^\star}$ is defined in
Theorem~\ref{thm:alphaStar}.

There are no closed form expressions for supremum quantiles of Gaussian processes, but they can be approximated using Monte Carlo methods or Kac-Rice formula-based approximations. In the following we use the Kac-Rice formula-based approximation of \cite{LiRe2023} as implemented in the \textsf{R}-package \texttt{ffscb} \citep{ffscb_Rpgk}, which is computationally efficient and provides accurate quantile approximations. We evaluate the finite-sample performance of the confidence band \eqref{eq:SCB} in terms of its coverage probability for
\begin{equation*}
\alpha^\star(t)=\alpha(t)+\int_0^t\beta(t,s)\frac{\sigma_X(s,t)}{\sigma_X(t,t)}\dd s
\end{equation*}
over $10,000$ simulations
\begin{equation*}
	\frac{1}{10,000}\sum_{r = 1}^{10,000}\mathbbm{1}_{\left(\alpha^\star(t) \,\in\, \operatorname{SCB}^{1-\gamma}_{\alpha^\star,r}(t) \;\text{for all} \; t\in[0,1]\right)},
\end{equation*}
where $\operatorname{SCB}^{1-\gamma}_{\alpha^\star,r}(t)$ is the confidence band \eqref{eq:SCB} computed in the $r$-th simulation.

Table \ref{tab:coverage} reports the coverage probabilities for different values of $\gamma\in\{0.05, 0.10\}$. While the confidence band is slightly optimistic in the small sample size scenario ($n=50$), it achieves coverage probabilities effectively equal to the nominal level in the large sample size scenario ($n=500$) for both DGPs and both values of $\gamma$.

\begin{table}[!tbp]
\centering
\caption{Coverage probabilities for $\operatorname{SCB}^{1-\gamma}_{\alpha^\star}$.}
\label{tab:coverage}
\begin{tabular}{r S[table-format=1.2]@{\,} S[table-format=1.2]@{\,} c S[table-format=1.2]@{\,} S[table-format=1.2]@{\,}}
\toprule
Nominal&\multicolumn{2}{c}{$1-\gamma=0.90$}&&\multicolumn{2}{c}{$1-\gamma=0.95$}\\
\cmidrule(lr){2-3}\cmidrule(lr){5-6}
&\multicolumn{1}{c}{$n=50$}& \multicolumn{1}{c}{$n=500$}&& \multicolumn{1}{c}{$n=50$}& \multicolumn{1}{c}{$n=500$}\\
\midrule
DGP1 & 0.85 & 0.89 && 0.91 & 0.95\\
DGP2 & 0.84 & 0.90 && 0.90 & 0.95\\
\bottomrule
\end{tabular}

\end{table}

\section{Applications}\label{sec:applications}

Both the concurrent functional linear regression model \eqref{eq:conc} and the historical functional linear regression model \eqref{eq:hist} are widely used in applications. For example, \cite{Abdoli_et_al_2008} use the concurrent model to estimate the time-varying effect of an on-body lift assistive device, \cite{Chang_et_al_2023} study the concurrent effect of income trajectories on intergenerational mobility, and \cite{Petrovich_et_al_2023} estimate time-varying labor supply elasticities. For the historical model, \cite{meyer2021bayesian} develop a Bayesian approach to estimate the effect of past air pollution exposure on current heart rate variability, while \cite{janssen2025learning} apply a modified historical model to hydrological data.

However, in either specification, failure to account for an existing historical or concurrent effect can induce substantial omitted-variable bias. In Section \ref{sec:appl1}, we consider an application in which the purely concurrent model is supported by the combined model, whereas the purely historical model estimates a strongly nonzero coefficient surface that largely absorbs the omitted concurrent effect. In Section \ref{sec:appl2}, we analyze a setting in which both concurrent and historical effects are present. In this case, the combined model successfully disentangles the two effects, while the concurrent and historical models produce markedly biased estimates. We also use the two case studies to illustrate the use of the confidence band \eqref{eq:SCB} for $\alpha^\star(t)$ for testing hypotheses about $\alpha(t)$ using relevance testing.

\subsection{Gait Data: Angles}\label{sec:appl1}

We revisit the well-known gait dataset (Figure \ref{fig:ApplGait}) of \cite{Olshen_et_al_1989}, since this dataset is one of the classic datasets for demonstrating the practical use of the functional concurrent linear regression model \citep{RHG_2009_book, Manrique_et_al_2018}. The data consist of hip and knee angle curves of $n=39$ children during one standardized gait cycle. We model the hip angle $Y(t)$ as the response and the knee angle $X(t)$ as its concurrent predictor.

\begin{figure}[!tbp]
		\resizebox{\linewidth}{!}{\input{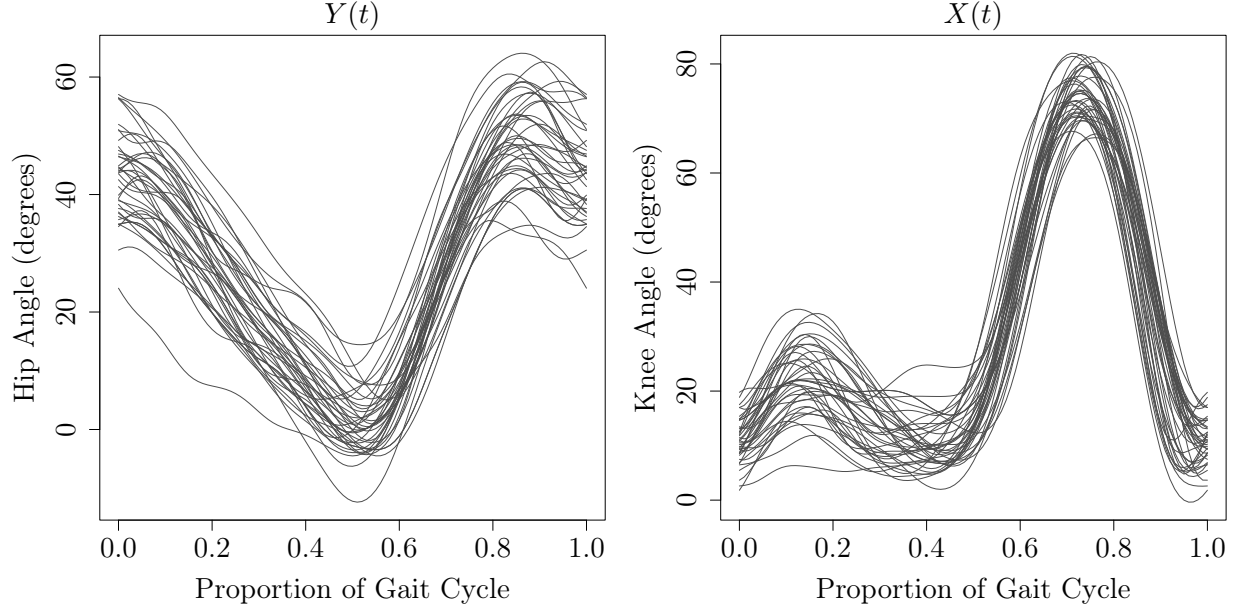}}
		\caption{Hip and knee angles of $n=39$ children during one standardized gait cycle.}\label{fig:ApplGait}
\end{figure}

Figure~\ref{fig:applConc} compares the concurrent effect estimates of the combined model \eqref{eq:combined} and the classical concurrent model \eqref{eq:conc}. Both models yield effectively identical estimates of $\widehat\alpha(t)$, indicating that there is effectively no omitted-variable bias in the concurrent model which omits the historical component.

Figure~\ref{fig:applHist} compares the historical effect estimates of the combined model \ref{eq:combined} and the classical (pure) historical model \eqref{eq:hist}. While the combined model yields effectively zero estimates of $\widehat\beta_t(s)$ throughout, the historical model produces strongly non-zero estimates of $\widehat\beta_t(s)$, which shows the bias due to the omitted concurrent effect in the historical model. This illustrates that the pure historical model is misspecified and yields a biased estimate of the historical effect, whereas the combined model correctly identifies the absence of a historical effect and yields an estimate that is consistent with the expected biomechanical structure of the data.

\begin{figure}[!tbp]
    \centering
    \begin{minipage}{0.46\linewidth}
        \centering
        \resizebox{\linewidth}{!}{\begin{tikzpicture}[x=1pt,y=1pt]
\definecolor{fillColor}{RGB}{255,255,255}
\path[use as bounding box,fill=fillColor,fill opacity=0.00] (0,0) rectangle (361.35,361.35);
\begin{scope}
\path[clip] ( 49.20, 49.20) rectangle (360.15,330.15);
\definecolor{drawColor}{RGB}{0,0,0}

\path[draw=drawColor,line width= 0.4pt,line join=round,line cap=round] ( 66.47,316.28) --
	( 69.35,307.85) --
	( 72.23,295.04) --
	( 75.11,280.69) --
	( 77.99,267.03) --
	( 80.87,255.55) --
	( 83.75,247.14) --
	( 86.63,242.15) --
	( 89.51,240.43) --
	( 92.39,241.37) --
	( 95.27,244.05) --
	( 98.15,247.40) --
	(101.02,250.49) --
	(103.90,252.62) --
	(106.78,253.42) --
	(109.66,252.77) --
	(112.54,250.71) --
	(115.42,247.33) --
	(118.30,242.74) --
	(121.18,237.04) --
	(124.06,230.34) --
	(126.94,222.69) --
	(129.82,214.09) --
	(132.70,204.57) --
	(135.57,194.29) --
	(138.45,183.64) --
	(141.33,173.31) --
	(144.21,164.11) --
	(147.09,156.76) --
	(149.97,151.63) --
	(152.85,148.72) --
	(155.73,147.70) --
	(158.61,148.06) --
	(161.49,149.31) --
	(164.37,151.02) --
	(167.25,152.88) --
	(170.12,154.67) --
	(173.00,156.10) --
	(175.88,156.75) --
	(178.76,155.97) --
	(181.64,152.96) --
	(184.52,146.87) --
	(187.40,137.19) --
	(190.28,124.07) --
	(193.16,108.54) --
	(196.04, 92.47) --
	(198.92, 78.04) --
	(201.80, 67.18) --
	(204.67, 60.99) --
	(207.55, 59.61) --
	(210.43, 62.36) --
	(213.31, 68.18) --
	(216.19, 75.98) --
	(219.07, 84.91) --
	(221.95, 94.41) --
	(224.83,104.23) --
	(227.71,114.25) --
	(230.59,124.50) --
	(233.47,135.08) --
	(236.35,146.15) --
	(239.22,157.95) --
	(242.10,170.73) --
	(244.98,184.68) --
	(247.86,199.68) --
	(250.74,215.17) --
	(253.62,229.92) --
	(256.50,242.08) --
	(259.38,249.38) --
	(262.26,249.78) --
	(265.14,242.42) --
	(268.02,228.32) --
	(270.90,209.98) --
	(273.77,190.18) --
	(276.65,170.89) --
	(279.53,153.02) --
	(282.41,136.81) --
	(285.29,122.38) --
	(288.17,109.98) --
	(291.05, 99.94) --
	(293.93, 92.40) --
	(296.81, 87.26) --
	(299.69, 84.33) --
	(302.57, 83.44) --
	(305.45, 84.50) --
	(308.32, 87.48) --
	(311.20, 92.38) --
	(314.08, 99.25) --
	(316.96,108.17) --
	(319.84,119.36) --
	(322.72,133.20) --
	(325.60,150.21) --
	(328.48,170.87) --
	(331.36,195.18) --
	(334.24,221.90) --
	(337.12,248.20) --
	(340.00,270.80) --
	(342.87,288.32) --
	(345.75,301.88) --
	(348.63,312.31);
\end{scope}
\begin{scope}
\path[clip] (  0.00,  0.00) rectangle (361.35,361.35);
\definecolor{drawColor}{RGB}{0,0,0}

\path[draw=drawColor,line width= 0.4pt,line join=round,line cap=round] ( 60.72, 49.20) -- (348.63, 49.20);

\path[draw=drawColor,line width= 0.4pt,line join=round,line cap=round] ( 60.72, 49.20) -- ( 60.72, 43.20);

\path[draw=drawColor,line width= 0.4pt,line join=round,line cap=round] (118.30, 49.20) -- (118.30, 43.20);

\path[draw=drawColor,line width= 0.4pt,line join=round,line cap=round] (175.88, 49.20) -- (175.88, 43.20);

\path[draw=drawColor,line width= 0.4pt,line join=round,line cap=round] (233.47, 49.20) -- (233.47, 43.20);

\path[draw=drawColor,line width= 0.4pt,line join=round,line cap=round] (291.05, 49.20) -- (291.05, 43.20);

\path[draw=drawColor,line width= 0.4pt,line join=round,line cap=round] (348.63, 49.20) -- (348.63, 43.20);

\node[text=drawColor,anchor=base,inner sep=0pt, outer sep=0pt, scale=  1.69] at ( 60.72, 27.60) {0.0};

\node[text=drawColor,anchor=base,inner sep=0pt, outer sep=0pt, scale=  1.69] at (118.30, 27.60) {0.2};

\node[text=drawColor,anchor=base,inner sep=0pt, outer sep=0pt, scale=  1.69] at (175.88, 27.60) {0.4};

\node[text=drawColor,anchor=base,inner sep=0pt, outer sep=0pt, scale=  1.69] at (233.47, 27.60) {0.6};

\node[text=drawColor,anchor=base,inner sep=0pt, outer sep=0pt, scale=  1.69] at (291.05, 27.60) {0.8};

\node[text=drawColor,anchor=base,inner sep=0pt, outer sep=0pt, scale=  1.69] at (348.63, 27.60) {1.0};

\path[draw=drawColor,line width= 0.4pt,line join=round,line cap=round] ( 49.20, 77.02) -- ( 49.20,319.74);

\path[draw=drawColor,line width= 0.4pt,line join=round,line cap=round] ( 49.20, 77.02) -- ( 43.20, 77.02);

\path[draw=drawColor,line width= 0.4pt,line join=round,line cap=round] ( 49.20,117.48) -- ( 43.20,117.48);

\path[draw=drawColor,line width= 0.4pt,line join=round,line cap=round] ( 49.20,157.93) -- ( 43.20,157.93);

\path[draw=drawColor,line width= 0.4pt,line join=round,line cap=round] ( 49.20,198.38) -- ( 43.20,198.38);

\path[draw=drawColor,line width= 0.4pt,line join=round,line cap=round] ( 49.20,238.84) -- ( 43.20,238.84);

\path[draw=drawColor,line width= 0.4pt,line join=round,line cap=round] ( 49.20,279.29) -- ( 43.20,279.29);

\path[draw=drawColor,line width= 0.4pt,line join=round,line cap=round] ( 49.20,319.74) -- ( 43.20,319.74);

\node[text=drawColor,rotate= 90.00,anchor=base,inner sep=0pt, outer sep=0pt, scale=  1.69] at ( 34.80, 77.02) {0.3};

\node[text=drawColor,rotate= 90.00,anchor=base,inner sep=0pt, outer sep=0pt, scale=  1.69] at ( 34.80,117.48) {0.4};

\node[text=drawColor,rotate= 90.00,anchor=base,inner sep=0pt, outer sep=0pt, scale=  1.69] at ( 34.80,157.93) {0.5};

\node[text=drawColor,rotate= 90.00,anchor=base,inner sep=0pt, outer sep=0pt, scale=  1.69] at ( 34.80,198.38) {0.6};

\node[text=drawColor,rotate= 90.00,anchor=base,inner sep=0pt, outer sep=0pt, scale=  1.69] at ( 34.80,238.84) {0.7};

\node[text=drawColor,rotate= 90.00,anchor=base,inner sep=0pt, outer sep=0pt, scale=  1.69] at ( 34.80,279.29) {0.8};

\node[text=drawColor,rotate= 90.00,anchor=base,inner sep=0pt, outer sep=0pt, scale=  1.69] at ( 34.80,319.74) {0.9};

\path[draw=drawColor,line width= 0.4pt,line join=round,line cap=round] ( 49.20, 49.20) --
	(360.15, 49.20) --
	(360.15,330.15) --
	( 49.20,330.15) --
	cycle;
\end{scope}
\begin{scope}
\path[clip] (  0.00,  0.00) rectangle (361.35,361.35);
\definecolor{drawColor}{RGB}{0,0,0}

\node[text=drawColor,anchor=base,inner sep=0pt, outer sep=0pt, scale=  1.69] at (204.67,339.93) {Concurrent Effects};

\node[text=drawColor,anchor=base,inner sep=0pt, outer sep=0pt, scale=  1.69] at (204.67,  3.60) {Proportion of Gait Cycle};
\end{scope}
\begin{scope}
\path[clip] ( 49.20, 49.20) rectangle (360.15,330.15);
\definecolor{drawColor}{RGB}{0,0,0}

\path[draw=drawColor,line width= 0.6pt,dash pattern=on 1pt off 3pt ,line join=round,line cap=round] ( 60.72,310.72) --
	( 63.60,315.32) --
	( 66.47,313.89) --
	( 69.35,306.67) --
	( 72.23,295.59) --
	( 75.11,283.05) --
	( 77.99,270.97) --
	( 80.87,260.64) --
	( 83.75,252.82) --
	( 86.63,247.83) --
	( 89.51,245.58) --
	( 92.39,245.63) --
	( 95.27,247.23) --
	( 98.15,249.54) --
	(101.02,251.74) --
	(103.90,253.22) --
	(106.78,253.62) --
	(109.66,252.80) --
	(112.54,250.77) --
	(115.42,247.64) --
	(118.30,243.58) --
	(121.18,238.76) --
	(124.06,233.34) --
	(126.94,227.42) --
	(129.82,221.04) --
	(132.70,214.18) --
	(135.57,206.86) --
	(138.45,199.23) --
	(141.33,191.63) --
	(144.21,184.51) --
	(147.09,178.35) --
	(149.97,173.45) --
	(152.85,169.89) --
	(155.73,167.52) --
	(158.61,166.08) --
	(161.49,165.29) --
	(164.37,164.95) --
	(167.25,164.89) --
	(170.12,164.99) --
	(173.00,165.04) --
	(175.88,164.66) --
	(178.76,163.24) --
	(181.64,159.94) --
	(184.52,153.91) --
	(187.40,144.52) --
	(190.28,131.79) --
	(193.16,116.56) --
	(196.04,100.52) --
	(198.92, 85.79) --
	(201.80, 74.32) --
	(204.67, 67.33) --
	(207.55, 65.11) --
	(210.43, 67.11) --
	(213.31, 72.33) --
	(216.19, 79.66) --
	(219.07, 88.21) --
	(221.95, 97.43) --
	(224.83,106.99) --
	(227.71,116.77) --
	(230.59,126.80) --
	(233.47,137.15) --
	(236.35,147.99) --
	(239.22,159.57) --
	(242.10,172.17) --
	(244.98,186.01) --
	(247.86,201.07) --
	(250.74,216.91) --
	(253.62,232.45) --
	(256.50,245.96) --
	(259.38,255.19) --
	(262.26,257.94) --
	(265.14,252.97) --
	(268.02,240.84) --
	(270.90,223.74) --
	(273.77,204.37) --
	(276.65,184.79) --
	(279.53,166.07) --
	(282.41,148.61) --
	(285.29,132.72) --
	(288.17,118.89) --
	(291.05,107.63) --
	(293.93, 99.16) --
	(296.81, 93.37) --
	(299.69, 89.98) --
	(302.57, 88.79) --
	(305.45, 89.64) --
	(308.32, 92.49) --
	(311.20, 97.31) --
	(314.08,104.14) --
	(316.96,113.04) --
	(319.84,124.22) --
	(322.72,138.02) --
	(325.60,154.92) --
	(328.48,175.35) --
	(331.36,199.21) --
	(334.24,225.27) --
	(337.12,250.76) --
	(340.00,272.52) --
	(342.87,289.16) --
	(345.75,301.62) --
	(348.63,310.72);

\path[draw=drawColor,line width= 0.4pt,line join=round,line cap=round] (108.88,312.15) -- (130.48,312.15);

\path[draw=drawColor,line width= 0.4pt,dash pattern=on 1pt off 3pt ,line join=round,line cap=round] (108.88,290.55) -- (130.48,290.55);

\node[text=drawColor,anchor=base west,inner sep=0pt, outer sep=0pt, scale=  1.56] at (141.28,306.78) {$\,\widehat{\alpha}(t)\,$ Combined Model};

\node[text=drawColor,anchor=base west,inner sep=0pt, outer sep=0pt, scale=  1.56] at (141.28,285.18) {$\widehat{\alpha}^{\star}(t)$ Concurrent Model};
\end{scope}
\end{tikzpicture}}
		\caption{Concurrent effect estimates of our combined (concurrent and historical) regression model \eqref{eq:combined} and of the classical concurrent model \eqref{eq:conc}.}
        \label{fig:applConc}
    \end{minipage}
    \hfill
    \begin{minipage}{0.46\linewidth}
        \centering
		\resizebox{\linewidth}{!}{\input{./Figures/Fig_Appl_Beta_1.tex}}
        \caption{Historical effect estimates $\widehat{\beta}_t(s)$, $s\leq t$, with $t=0.1, 0.2,\dots,1$, for the combined regression model \eqref{eq:combined} and the historical model \eqref{eq:hist}.}
        \label{fig:applHist}
    \end{minipage}
\end{figure}

In settings where the historical component vanishes, i.e., $\beta_t(s)=0$ for all $s,t\in[0,1]$, we have $\alpha^\star(t)=\alpha(t)$ and $\widehat{\alpha}^\star(t)=\widehat{\alpha}(t)$ for all $t\in[0,1]$. Thus, in absence of a historical effect, the asymptotic distribution of $\widehat{\alpha}^{\star}(t)$ established in Theorem~\ref{thm:alphaStar} can be used to conduct inference on $\alpha(t)$, which is otherwise generally infeasible. In particular, under the absence of a historical effects, the intrinsically untestable null-hypothesis
\begin{equation}\label{eq:nullhyp_alpha_1}
H_0\colon \alpha(t) = 0 \quad\text{for all}\quad t\in[0,1]
\end{equation}
becomes equivalent to the testable null-hypothesis
\begin{equation}\label{eq:nullhyp_alphaStar_1}
H_0\colon \alpha^\star(t) = 0 \quad\text{for all}\quad t\in[0,1].
\end{equation}

Figure~\ref{fig:ConfBand_1} shows the simultaneous 95 \% confidence band \eqref{eq:SCB} for $\alpha^\star(t)$, which can be used to test the null-hypothesis of no concurrent effect \eqref{eq:nullhyp_alphaStar_1}. The confidence band does not contain zero for any $t\in[0,0.29]\cup[0.57,0.73]\cup[0.9,1.0]$, providing strong evidence for non-zero concurrent effects of the knee angle on the hip angle throughout the gait cycle, which allows rejecting the null-hypothesis \eqref{eq:nullhyp_alphaStar_1}. Under the here plausible assumption of no historical effect, this also allows rejecting the null-hypothesis \eqref{eq:nullhyp_alpha_1}.

\begin{figure}[!tbp]
	\centering
	\resizebox{0.46\linewidth}{!}{\input{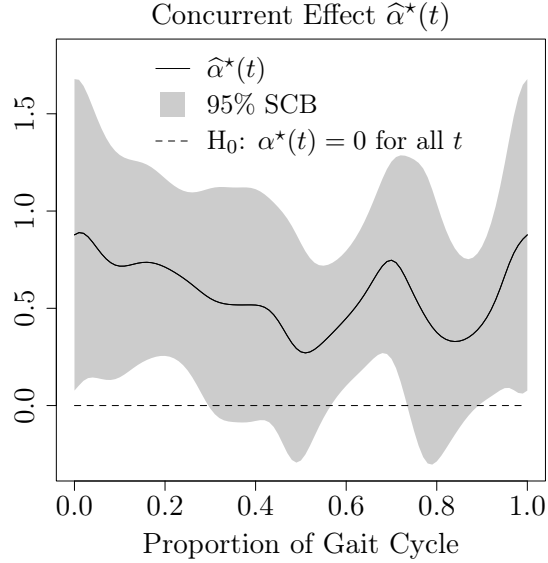}}
	\caption{Application of Section \ref{sec:appl1}: Simultaneous 95 \% confidence band (SCB) for the concurrent effect $\alpha^{\star}(t)$.}
	\label{fig:ConfBand_1}
\end{figure}

\subsection{Gait Data: Torques}\label{sec:appl2}

We analyze joint torque data, comprising ankle and knee torque curves of $n=120$ amateur runners during a standardized gait cycle (Figure~\ref{fig:ApplGait_2}).
We model the knee torque $Y(t)$ as the response and the ankle torque $X(t)$ as the predictor; further details on the data can be found in \cite{LWHB2014}. Within the gait cycle, foot and ankle motion and torque temporally precede knee motion and torque, providing a biomechanical rationale for a genuine historical effect of $X(s)$, $s < t$, on $Y(t)$, in addition to a concurrent effect.

\begin{figure}[!tbp]
		\resizebox{\linewidth}{!}{\input{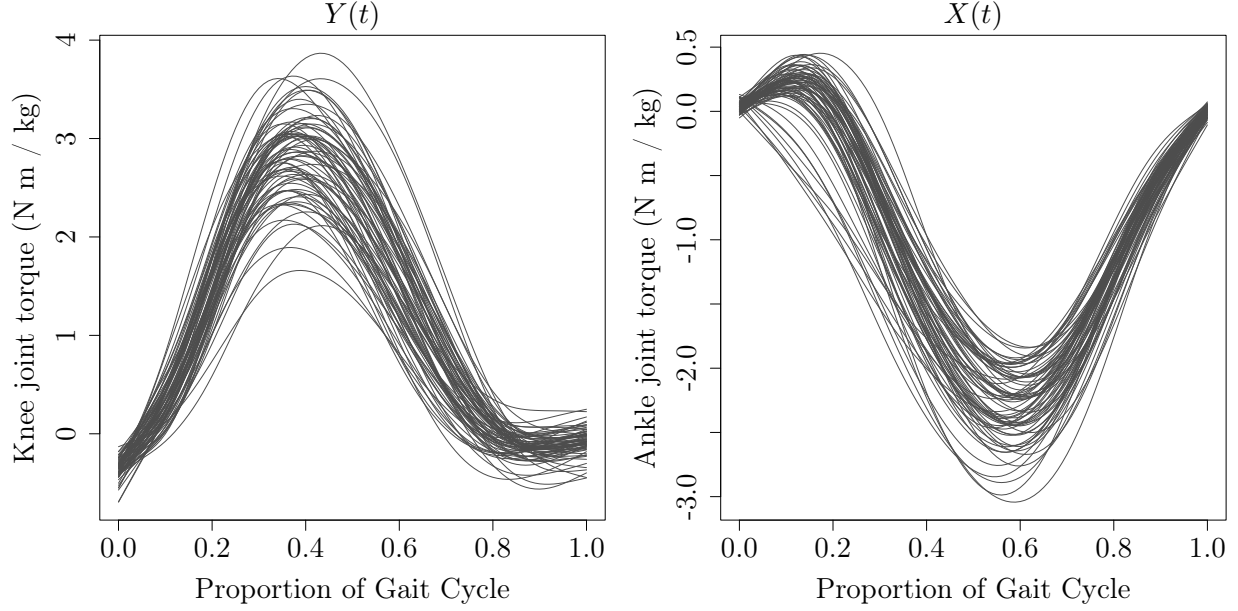}}
		\caption{Ankle and knee joint torque curves of $n=120$ amateur runners.}\label{fig:ApplGait_2}
\end{figure}

Figures~\ref{fig:applConc_2} and~\ref{fig:applHist_2} illustrate the contrasting case where both a concurrent and a historical effect are present.
The concurrent effect estimates $\widehat\alpha(t)$ of the combined model and the classical concurrent model differ substantially (Figure~\ref{fig:applConc_2}), since the concurrent-only model conflates the historical contribution with the concurrent effect.
The historical surface $\widehat\beta_t(s)$ estimated by the plain historical model is likewise distorted, as it attempts to absorb the concurrent effect on top of the genuine historical contribution (Figure~\ref{fig:applHist_2}).
The combined model disentangles both effects, yielding estimates of $\widehat\alpha(t)$ and $\widehat\beta_t(s)$ that are consistent with the expected biomechanical structure.

\begin{figure}[!tbp]
    \centering
    \begin{minipage}{0.46\linewidth}
        \centering
        \resizebox{\linewidth}{!}{\begin{tikzpicture}[x=1pt,y=1pt]
\definecolor{fillColor}{RGB}{255,255,255}
\path[use as bounding box,fill=fillColor,fill opacity=0.00] (0,0) rectangle (361.35,361.35);
\begin{scope}
\path[clip] ( 49.20, 49.20) rectangle (360.15,330.15);
\definecolor{drawColor}{RGB}{0,0,0}

\path[draw=drawColor,line width= 0.4pt,line join=round,line cap=round] ( 66.47,221.07) --
	( 69.35,196.39) --
	( 72.23,182.90) --
	( 75.11,174.53) --
	( 77.99,169.05) --
	( 80.87,165.33) --
	( 83.75,162.56) --
	( 86.63,160.05) --
	( 89.51,157.19) --
	( 92.39,153.55) --
	( 95.27,149.06) --
	( 98.15,143.98) --
	(101.02,138.77) --
	(103.90,133.87) --
	(106.78,129.60) --
	(109.66,126.10) --
	(112.54,123.43) --
	(115.42,121.59) --
	(118.30,120.57) --
	(121.18,120.32) --
	(124.06,120.73) --
	(126.94,121.68) --
	(129.82,123.03) --
	(132.70,124.66) --
	(135.57,126.46) --
	(138.45,128.38) --
	(141.33,130.36) --
	(144.21,132.38) --
	(147.09,134.40) --
	(149.97,136.37) --
	(152.85,138.26) --
	(155.73,140.00) --
	(158.61,141.51) --
	(161.49,142.71) --
	(164.37,143.50) --
	(167.25,143.81) --
	(170.12,143.55) --
	(173.00,142.68) --
	(175.88,141.16) --
	(178.76,138.97) --
	(181.64,136.09) --
	(184.52,132.54) --
	(187.40,128.32) --
	(190.28,123.48) --
	(193.16,118.06) --
	(196.04,112.15) --
	(198.92,105.85) --
	(201.80, 99.27) --
	(204.67, 92.60) --
	(207.55, 86.01) --
	(210.43, 79.73) --
	(213.31, 73.97) --
	(216.19, 68.96) --
	(219.07, 64.89) --
	(221.95, 61.91) --
	(224.83, 60.13) --
	(227.71, 59.61) --
	(230.59, 60.31) --
	(233.47, 62.16) --
	(236.35, 65.05) --
	(239.22, 68.81) --
	(242.10, 73.27) --
	(244.98, 78.24) --
	(247.86, 83.58) --
	(250.74, 89.14) --
	(253.62, 94.81) --
	(256.50,100.53) --
	(259.38,106.26) --
	(262.26,111.97) --
	(265.14,117.68) --
	(268.02,123.38) --
	(270.90,129.09) --
	(273.77,134.85) --
	(276.65,140.66) --
	(279.53,146.53) --
	(282.41,152.48) --
	(285.29,158.51) --
	(288.17,164.61) --
	(291.05,170.77) --
	(293.93,176.97) --
	(296.81,183.21) --
	(299.69,189.46) --
	(302.57,195.71) --
	(305.45,201.93) --
	(308.32,208.11) --
	(311.20,214.28) --
	(314.08,220.44) --
	(316.96,226.64) --
	(319.84,232.95) --
	(322.72,239.42) --
	(325.60,246.16) --
	(328.48,253.25) --
	(331.36,260.75) --
	(334.24,268.65) --
	(337.12,276.79) --
	(340.00,284.79) --
	(342.87,291.90) --
	(345.75,296.93) --
	(348.63,298.38);
\end{scope}
\begin{scope}
\path[clip] (  0.00,  0.00) rectangle (361.35,361.35);
\definecolor{drawColor}{RGB}{0,0,0}

\path[draw=drawColor,line width= 0.4pt,line join=round,line cap=round] ( 60.72, 49.20) -- (348.63, 49.20);

\path[draw=drawColor,line width= 0.4pt,line join=round,line cap=round] ( 60.72, 49.20) -- ( 60.72, 43.20);

\path[draw=drawColor,line width= 0.4pt,line join=round,line cap=round] (118.30, 49.20) -- (118.30, 43.20);

\path[draw=drawColor,line width= 0.4pt,line join=round,line cap=round] (175.88, 49.20) -- (175.88, 43.20);

\path[draw=drawColor,line width= 0.4pt,line join=round,line cap=round] (233.47, 49.20) -- (233.47, 43.20);

\path[draw=drawColor,line width= 0.4pt,line join=round,line cap=round] (291.05, 49.20) -- (291.05, 43.20);

\path[draw=drawColor,line width= 0.4pt,line join=round,line cap=round] (348.63, 49.20) -- (348.63, 43.20);

\node[text=drawColor,anchor=base,inner sep=0pt, outer sep=0pt, scale=  1.69] at ( 60.72, 27.60) {0.0};

\node[text=drawColor,anchor=base,inner sep=0pt, outer sep=0pt, scale=  1.69] at (118.30, 27.60) {0.2};

\node[text=drawColor,anchor=base,inner sep=0pt, outer sep=0pt, scale=  1.69] at (175.88, 27.60) {0.4};

\node[text=drawColor,anchor=base,inner sep=0pt, outer sep=0pt, scale=  1.69] at (233.47, 27.60) {0.6};

\node[text=drawColor,anchor=base,inner sep=0pt, outer sep=0pt, scale=  1.69] at (291.05, 27.60) {0.8};

\node[text=drawColor,anchor=base,inner sep=0pt, outer sep=0pt, scale=  1.69] at (348.63, 27.60) {1.0};

\path[draw=drawColor,line width= 0.4pt,line join=round,line cap=round] ( 49.20,116.92) -- ( 49.20,323.06);

\path[draw=drawColor,line width= 0.4pt,line join=round,line cap=round] ( 49.20,116.92) -- ( 43.20,116.92);

\path[draw=drawColor,line width= 0.4pt,line join=round,line cap=round] ( 49.20,185.63) -- ( 43.20,185.63);

\path[draw=drawColor,line width= 0.4pt,line join=round,line cap=round] ( 49.20,254.35) -- ( 43.20,254.35);

\path[draw=drawColor,line width= 0.4pt,line join=round,line cap=round] ( 49.20,323.06) -- ( 43.20,323.06);

\node[text=drawColor,rotate= 90.00,anchor=base,inner sep=0pt, outer sep=0pt, scale=  1.69] at ( 34.80,116.92) {-1};

\node[text=drawColor,rotate= 90.00,anchor=base,inner sep=0pt, outer sep=0pt, scale=  1.69] at ( 34.80,185.63) {0};

\node[text=drawColor,rotate= 90.00,anchor=base,inner sep=0pt, outer sep=0pt, scale=  1.69] at ( 34.80,254.35) {1};

\node[text=drawColor,rotate= 90.00,anchor=base,inner sep=0pt, outer sep=0pt, scale=  1.69] at ( 34.80,323.06) {2};

\path[draw=drawColor,line width= 0.4pt,line join=round,line cap=round] ( 49.20, 49.20) --
	(360.15, 49.20) --
	(360.15,330.15) --
	( 49.20,330.15) --
	cycle;
\end{scope}
\begin{scope}
\path[clip] (  0.00,  0.00) rectangle (361.35,361.35);
\definecolor{drawColor}{RGB}{0,0,0}

\node[text=drawColor,anchor=base,inner sep=0pt, outer sep=0pt, scale=  1.69] at (204.67,339.93) {Concurrent Effect};

\node[text=drawColor,anchor=base,inner sep=0pt, outer sep=0pt, scale=  1.69] at (204.67,  3.60) {Proportion of Gait Cycle};
\end{scope}
\begin{scope}
\path[clip] ( 49.20, 49.20) rectangle (360.15,330.15);
\definecolor{drawColor}{RGB}{0,0,0}

\path[draw=drawColor,line width= 0.6pt,dash pattern=on 1pt off 3pt ,line join=round,line cap=round] ( 60.72,189.25) --
	( 63.60,223.56) --
	( 66.47,219.07) --
	( 69.35,204.88) --
	( 72.23,194.58) --
	( 75.11,188.04) --
	( 77.99,184.06) --
	( 80.87,181.78) --
	( 83.75,180.69) --
	( 86.63,180.37) --
	( 89.51,180.49) --
	( 92.39,180.79) --
	( 95.27,181.08) --
	( 98.15,181.24) --
	(101.02,181.24) --
	(103.90,181.07) --
	(106.78,180.80) --
	(109.66,180.48) --
	(112.54,180.19) --
	(115.42,179.99) --
	(118.30,179.93) --
	(121.18,180.02) --
	(124.06,180.28) --
	(126.94,180.68) --
	(129.82,181.20) --
	(132.70,181.81) --
	(135.57,182.48) --
	(138.45,183.18) --
	(141.33,183.89) --
	(144.21,184.58) --
	(147.09,185.25) --
	(149.97,185.88) --
	(152.85,186.46) --
	(155.73,186.99) --
	(158.61,187.47) --
	(161.49,187.89) --
	(164.37,188.26) --
	(167.25,188.57) --
	(170.12,188.83) --
	(173.00,189.02) --
	(175.88,189.16) --
	(178.76,189.23) --
	(181.64,189.22) --
	(184.52,189.13) --
	(187.40,188.94) --
	(190.28,188.67) --
	(193.16,188.28) --
	(196.04,187.80) --
	(198.92,187.20) --
	(201.80,186.50) --
	(204.67,185.69) --
	(207.55,184.78) --
	(210.43,183.76) --
	(213.31,182.65) --
	(216.19,181.45) --
	(219.07,180.17) --
	(221.95,178.84) --
	(224.83,177.47) --
	(227.71,176.08) --
	(230.59,174.70) --
	(233.47,173.34) --
	(236.35,172.04) --
	(239.22,170.81) --
	(242.10,169.67) --
	(244.98,168.65) --
	(247.86,167.75) --
	(250.74,167.01) --
	(253.62,166.45) --
	(256.50,166.07) --
	(259.38,165.92) --
	(262.26,166.00) --
	(265.14,166.36) --
	(268.02,167.00) --
	(270.90,167.97) --
	(273.77,169.27) --
	(276.65,170.92) --
	(279.53,172.95) --
	(282.41,175.34) --
	(285.29,178.12) --
	(288.17,181.25) --
	(291.05,184.74) --
	(293.93,188.56) --
	(296.81,192.69) --
	(299.69,197.11) --
	(302.57,201.80) --
	(305.45,206.73) --
	(308.32,211.91) --
	(311.20,217.32) --
	(314.08,222.98) --
	(316.96,228.94) --
	(319.84,235.25) --
	(322.72,242.01) --
	(325.60,249.35) --
	(328.48,257.44) --
	(331.36,266.45) --
	(334.24,276.54) --
	(337.12,287.74) --
	(340.00,299.72) --
	(342.87,311.30) --
	(345.75,319.65) --
	(348.63,319.74);

\path[draw=drawColor,line width= 0.4pt,line join=round,line cap=round] (108.88,312.15) -- (130.48,312.15);

\path[draw=drawColor,line width= 0.4pt,dash pattern=on 1pt off 3pt ,line join=round,line cap=round] (108.88,290.55) -- (130.48,290.55);

\node[text=drawColor,anchor=base west,inner sep=0pt, outer sep=0pt, scale=  1.56] at (141.28,306.78) {$\,\widehat{\alpha}(t)\,$ Combined Model};

\node[text=drawColor,anchor=base west,inner sep=0pt, outer sep=0pt, scale=  1.56] at (141.28,285.18) {$\widehat{\alpha}^{\star}(t)$ Concurrent Model};
\end{scope}
\end{tikzpicture}}
        \caption{Concurrent effect estimates of our combined (concurrent and historical) regression model \eqref{eq:combined} and of the classical concurrent model \eqref{eq:conc}.}
        \label{fig:applConc_2}
    \end{minipage}
    \hfill
    \begin{minipage}{0.46\linewidth}
        \centering
        \resizebox{\linewidth}{!}{\input{./Figures/Fig_Appl_Beta_2.tex}}
        \caption{Historical effect estimates $\widehat{\beta}_t(s)$, $s\leq t$, with $t=0.1, 0.2,\dots,1$, for the combined regression model \eqref{eq:combined} and the historical model \eqref{eq:hist}.}
        \label{fig:applHist_2}
    \end{minipage}
\end{figure}

In the presence of a non-zero historical effect, we have $\alpha^\star(t)\neq \alpha(t)$ and $\widehat{\alpha}^\star(t)\neq \widehat{\alpha}(t)$ for at least some $t\in[0,1]$. Under this situation, the infeasible inference on $\alpha(t)$ is not equivalent to feasible inference on $\alpha^\star(t)$ as in Section \ref{sec:appl1}. However, we still may use the asymptotic distribution of $\widehat{\alpha}^\star(t)$ to conduct inference on $\alpha(t)$ using the idea of relevance testing; namely, in situations where it is plausible to have access (by expert knowledge or previous studies) to a relevance band $\mathcal{B}(t)=[\mathcal{B}_\ell(t), \mathcal{B}_u(t)]$ such that
\begin{equation}\label{eq:relevance_band}
\mathcal{B}_\ell(t) \leq \int_0^t \beta_t(s) \frac{\sigma_X(t,s)}{\sigma_X(t,t)} \dd s \leq \mathcal{B}_u(t)\quad\text{for all}\quad t\in[0,1].
\end{equation}

Observe that the intrinsically untestable null-hypothesis of interest
\begin{equation}\label{eq:nullhyp_alpha_2}
H_0\colon \alpha(t) = 0 \quad\text{for all}\quad t\in[0,1]
\end{equation}
is equivalent to
\begin{equation}\label{eq:nullhyp_alphaStar_2}
H_0\colon \alpha^{\star}(t) = \int_0^t \beta_t(s) \frac{\sigma_X(t,s)}{\sigma_X(t,t)} \dd s \quad\text{for all}\quad t\in[0,1],
\end{equation}
since $\alpha(t) = \alpha^{\star}(t) - \int_0^t \beta_t(s) \sigma_X(t,s)/\sigma_X(t,t) \dd s$. The null-hypothesis \eqref{eq:nullhyp_alphaStar_2} is generally infeasible since the right-hand-side is typically unknown. However, if there is expert knowledge that allows to specify a relevance band $\mathcal{B}(t)$ as in \eqref{eq:relevance_band}, then we can formulate the following testable null- and alternative hypotheses
\begin{align}
H_0\colon & \alpha^{\star}(t) \in \mathcal{B}(t) \quad\text{for all}\quad t\in[0,1]\label{eq:nullhyp_alphaStar_relevance}\\
H_1\colon &\alpha^{\star}(t) \notin \mathcal{B}(t) \quad\text{for some}\quad t\in[0,1].\label{eq:althyp_alphaStar_relevance}
\end{align}

Clearly, by the definition of the relevance band $\mathcal{B}(t)$, we have that if
$\alpha^{\star}(t) \notin \mathcal{B}(t)$ for some $t\in[0,1]$, then $\alpha^{\star}(t) \neq \int_0^t \beta_t(s) \frac{\sigma_X(t,s)}{\sigma_X(t,t)} \dd s$ for the same $t\in[0,1]$. Thus, rejecting the null-hypothesis \eqref{eq:nullhyp_alphaStar_relevance} implies that the generally untestable null-hypothesis \eqref{eq:nullhyp_alphaStar_2} and its equivalent null-hypothesis \eqref{eq:nullhyp_alpha_2} can also be rejected.

To construct a plausible relevance band $\mathcal{B}(t)$, we divide the data into a pre-estimation sample ($n=40$) and a test sample ($n=80$). We use the pre-estimation sample with $n=40$ to obtain an estimate of $\int_0^t \beta_t(s) \sigma_X(t,s)/\sigma_X(t,t) \dd s$ using $\int_0^t \widehat{\beta}_t(s) \widehat{\sigma}_X(t,s)/\widehat{\sigma}_X(t,t) \dd s$. In consultation with experts, we then specify a relevance band $\mathcal{B}(t)$ around this estimate, which is used to conduct the relevance test \eqref{eq:nullhyp_alphaStar_relevance}--\eqref{eq:althyp_alphaStar_relevance} on the test sample with $n=80$ observations using the confidence band \eqref{eq:SCB} for $\alpha^\star(t)$.

Figure~\ref{fig:ConfBand_2} shows the simultaneous 95 \% confidence band \eqref{eq:SCB} for $\alpha^\star(t)$ and the relevance band $\mathcal{B}(t)$ against which we compare the simultaneous confidence band. Both bands overlap for the totality of the gait cycle, indicating that we cannot reject the null-hypothesis \eqref{eq:nullhyp_alphaStar_relevance} of $\alpha^\star(t)$ being within the relevance band $\mathcal{B}(t)$, and hence cannot reject the null-hypothesis \eqref{eq:nullhyp_alpha_2} of no concurrent effect $\alpha(t)=0$ for all $t\in[0,1]$.

\begin{figure}[!tbp]
	\centering
	\resizebox{0.46\linewidth}{!}{\input{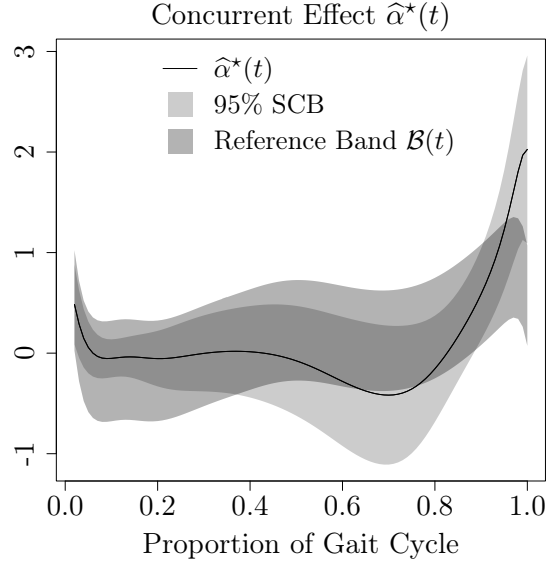}}
	\caption{Application of Section \ref{sec:appl2}: 95 \% SCB for the concurrent effect $\alpha^{\star}(t)$ and the relevance band $\mathcal{B}(t)$.}
	\label{fig:ConfBand_2}
\end{figure}

\section{Conclusion}\label{sec:conclusion}

We have studied a function-on-function regression model that combines a historical effect of past predictor values with a concurrent pointwise effect. Since the concurrent term is a point evaluation while the historical term is an $L^2$-integral effect, the two components are not automatically separable. We characterized this ambiguity and showed that the concurrent and historical coefficients are separately identified when the covariance eigenfunctions of the predictor are not pointwise square-summable.

Building on this identification result, we proposed an estimator based on an orthogonalized representation of the predictor process. The resulting estimation problem differs from standard functional linear regression because the relevant residualized predictor is not observed and must be estimated from the data. Our asymptotic analysis establishes consistency rates for the historical and concurrent estimators and reveals a smoothness-driven trade-off: smoother predictor trajectories lead to faster convergence of the historical component, whereas rougher trajectories yield improved rates for the concurrent component.

The simulation study and the two gait-data applications illustrate the practical importance of separating concurrent and historical effects. In particular, omitting one component can lead to omitted-variable bias in the estimated coefficient of the other component, whereas the combined model can disentangle the two types of effects. The weak convergence result for the well-posed coefficient $\alpha^\star$ further provides a basis for simultaneous confidence bands for the reduced concurrent coefficient, while full inference for the structural coefficients remains an intrinsically more delicate problem.

\paragraph{Funding.}
The first and second authors gratefully acknowledge financial support by the Deutsche Forschungsgemeinschaft (DFG, German Research Foundation) under Germany's Excellence Strategy – EXC-2047/1 – 390685813.
The third author gratefully acknowledges financial support from the Deutsche Forschungsgemeinschaft (DFG, German Research Foundation), grant number 511905296.

\paragraph{\textsf{R}-package \texttt{chflr}.}
The \textsf{R}-package \texttt{chflr} implements the estimation procedure for the combined historical and concurrent functional linear regression model described in Section \ref{sec:estimation}, and is available at \url{https://github.com/lidom/chflr\_repo}.

\clearpage
\part*{Supplement to ``Combining Concurrent and Historical Functional Linear Regression''}

This supplementary material provides technical proofs of all lemmas and theorems of the main paper.

\appendix
\section{Proofs for Section 2 of the main paper -- Identifiability}

\subsection{Proof of Lemma 2.2 of the main paper}

Suppose first that (i) holds.
Let $\xi_k=\langle X^\mu,\psi_k\rangle$ denote the principal component scores.
Taking covariances with $\xi_k$ on both sides of $\alpha(t)X^\mu(t)=\langle\breve\beta_t,X^\mu\rangle$ gives $\alpha(t)\lambda_k\psi_k(t)=\lambda_k\langle\breve\beta_t,\psi_k\rangle$ for every $k\in\mathbb N$.
Since $\lambda_k>0$, it follows that $\langle\breve\beta_t,\psi_k\rangle=\alpha(t)\psi_k(t)$ for every $k\in\mathbb N$.
Because $\breve\beta_t\in\operatorname{Ker}(\Gamma_X)^\perp$, it has the basis representation $\breve\beta_t=\sum_{k=1}^\infty\langle\breve\beta_t,\psi_k\rangle\psi_k$.
Hence $\|\breve\beta_t\|^2=\sum_{k=1}^\infty\langle\breve\beta_t,\psi_k\rangle^2=\alpha(t)^2\sum_{k=1}^\infty\psi_k(t)^2$.
Since $\|\breve\beta_t\|<\infty$ and $\alpha(t)\neq0$, this implies $\sum_{k=1}^\infty\psi_k(t)^2<\infty$.

Conversely, suppose that statement (ii) holds and define $\breve\beta_t=\sum_{k=1}^\infty\alpha(t)\psi_k(t)\psi_k$.
Then $\breve\beta_t\in\operatorname{Ker}(\Gamma_X)^\perp\subseteq L^2[0,1]$, since $\|\breve\beta_t\|^2=\alpha(t)^2\sum_{k=1}^\infty\psi_k(t)^2<\infty$.
Moreover, $\langle\breve\beta_t,\psi_k\rangle=\alpha(t)\psi_k(t)$ for every $k\in\mathbb N$.
Using the Karhunen-Lo\`eve expansion of $X^\mu$, we obtain $\langle\breve\beta_t,X^\mu\rangle=\sum_{k=1}^\infty\langle\breve\beta_t,\psi_k\rangle\langle X^\mu,\psi_k\rangle=\sum_{k=1}^\infty\alpha(t)\psi_k(t)\langle X^\mu,\psi_k\rangle=\alpha(t)X^\mu(t)$.
\hfill$\square$

\subsection{Proof of Lemma 2.3 of the main paper}

We do a proof by contradiction.
Suppose that there exists $t\in\mathcal I$ such that $\sum_{j=1}^\infty \psi_j(t)^2<\infty$.
Then there exists $K\in\mathbb N$ such that $\sum_{j=K}^\infty \psi_j(t)^2<1/4$.
For $\theta>0$, choose a continuously differentiable function $f_{\theta,t}\in\operatorname{Ker}(\Gamma_X)^\perp$ such that $f_{\theta,t}(t)\geq1$, $\int_0^1 f_{\theta,t}(s)^2\,ds=1$, and $f_{\theta,t}(s)=0$ for all $s\in[0,1]$ with $|t-s|\geq\theta$.
Then $\sum_{j=1}^\infty \langle f_{\theta,t},\psi_j\rangle^2\leq1$.
Moreover, for each fixed $j\in\mathbb N$, $\langle f_{\theta,t},\psi_j\rangle\to0$ as $\theta\to0$.
Hence there exists $\theta^*>0$ such that
$
|\sum_{j=1}^{K-1}\langle f_{\theta^*,t},\psi_j\rangle\psi_j(t)|<1/2.
$
By the assumed pointwise convergence of the eigenfunction expansion,
$
f_{\theta^*,t}(t)
=
\sum_{j=1}^\infty \langle f_{\theta^*,t},\psi_j\rangle\psi_j(t)
$.
Consequently,
\begin{align*}
1
&\leq f_{\theta^*,t}(t)\\
&\leq
\bigg|\sum_{j=1}^{K-1}\langle f_{\theta^*,t},\psi_j\rangle\psi_j(t)\bigg|
+
\bigg|\sum_{j=K}^\infty \langle f_{\theta^*,t},\psi_j\rangle\psi_j(t)\bigg|\\
&\leq
\frac12
+
\bigg(
\sum_{j=K}^\infty \langle f_{\theta^*,t},\psi_j\rangle^2
\sum_{j=K}^\infty \psi_j(t)^2
\bigg)^{1/2}\\
&<
\frac12+\frac12
=
1,
\end{align*}
which is impossible.
\hfill$\square$

\subsection{Proof of Theorem 2.4 of the main paper}

First, consider the case $\breve\alpha(t)\neq\alpha(t)$ and $\|\breve\beta_t-\beta_t\|>0$.
We use a proof by contradiction.
Suppose that
$
\E((\mathcal L(X^\mu)(t)-\breve\alpha(t)X^\mu(t)-\langle\breve\beta_t,X^\mu\rangle)^2)=0
$.
Then
\[
\E\left(((\alpha(t)-\breve\alpha(t))X^\mu(t)+\langle\beta_t-\breve\beta_t,X^\mu\rangle)^2\right)=0,
\]
and hence
$(\alpha(t)-\breve\alpha(t))X^\mu(t)=\langle\breve\beta_t-\beta_t,X^\mu\rangle
$ almost surely.
The characterization in Lemma 2.2 of the main paper is invariant under multiplication of $X^\mu(t)$ by a nonzero scalar.
Applied with the scalar $\alpha(t)-\breve\alpha(t)$, it yields $\sum_{k=1}^\infty\psi_k(t)^2<\infty$, contradicting Assumption I.5 of the main paper.
Thus the assertion holds in this case.

Next, consider the case where $\|\breve\beta_t-\beta_t\|=0$ and $\breve\alpha(t)\neq\alpha(t)$.
Then
$$
\E\left((\mathcal L(X^\mu)(t)-\breve\alpha(t)X^\mu(t)-\langle\breve\beta_t,X^\mu\rangle)^2\right)
=
(\alpha(t)-\breve\alpha(t))^2\sigma_X(t,t)>0,
$$
because $\sigma_X(t,t)>0$ by Assumption I.2 of the main paper.

It remains to consider the case where $\breve\alpha(t)=\alpha(t)$ and $\|\breve\beta_t-\beta_t\|>0$.
In this case,
\begin{align*}
\E\left((\mathcal L(X^\mu)(t)-\breve\alpha(t)X^\mu(t)-\langle\breve\beta_t,X^\mu\rangle)^2\right)
&=
\E\left(\langle\beta_t-\breve\beta_t,X^\mu\rangle^2\right) \\
&=
\sum_{k=1}^\infty\lambda_k\langle\beta_t-\breve\beta_t,\psi_k\rangle^2>0.
\end{align*}
The last inequality follows because $\beta_t-\breve\beta_t\in\operatorname{Ker}(\Gamma_X)^\perp$, $\|\beta_t-\breve\beta_t\|>0$, and all eigenvalues on $\operatorname{Ker}(\Gamma_X)^\perp$ are positive.
This completes the proof.
\hfill$\square$

\subsection{Uniform version of Theorem 2.4 of the main paper}

\begin{theorem}\label{thm:UniformIdentification}
Let Assumptions I.1--I.5 of the main paper hold true and consider another pair of coefficient functions $\breve\alpha:[0,1]\to\mathbb R$ and $\breve\beta:[0,1]^2\to\mathbb R$ with $\breve\beta_t:=\breve\beta(t,\cdot)\in\operatorname{Ker}(\Gamma_X)^\perp$ and $\breve\beta(t,s)=0$ for all $s>t$.
Assume that $\alpha$ and $\breve\alpha$ are continuous on $[0,1]$ and that $\beta$ and $\breve\beta$ are continuous on $\{(t,s):0\leq s\leq t\leq1\}$.
Let $I\subseteq[0,1]$ be compact.
If, for every $t\in I$, at least one of the two inequalities $\breve\alpha(t)\neq\alpha(t)$ or $\|\breve\beta_t-\beta_t\|\neq0$ holds, then
$$
  \inf_{t\in I}
  \E\Big((\mathcal L(X^\mu)(t)-\breve\alpha(t)X^\mu(t)-\langle\breve\beta_t,X^\mu\rangle)^2\Big)>0.
$$
\end{theorem}

\begin{proof}
Define
$$
f(t)
:=
\mathcal L(X^\mu)(t)-\breve\alpha(t)X^\mu(t)-\langle\breve\beta_t,X^\mu\rangle
=
(\alpha(t)-\breve\alpha(t))X^\mu(t)+\langle\beta_t-\breve\beta_t,X^\mu\rangle .
$$
Then
\begin{align*}
g(t)
:=
\E(f(t)^2)
&=
(\alpha(t)-\breve\alpha(t))^2\sigma_X(t,t)
+2(\alpha(t)-\breve\alpha(t))\langle\sigma_X(t,\cdot),\beta_t-\breve\beta_t\rangle  \\
&\quad
+\langle\beta_t-\breve\beta_t,\Gamma_X(\beta_t-\breve\beta_t)\rangle .
\end{align*}
Because $\alpha-\breve\alpha$ is continuous on $[0,1]$, $\beta-\breve\beta$ is continuous on $\{(r,s):0\leq s\leq r\leq1\}$, and $\sigma_X$ is continuous on $[0,1]^2$ by Assumption I.2 of the main paper, it follows that $g$ is continuous on $[0,1]$.
Fix $t\in I$ and suppose that $g(t)=0$.
Then
\begin{equation}\label{eq:rep-eq1}
(\alpha(t)-\breve\alpha(t))X^\mu(t)
=
\langle\breve\beta_t-\beta_t,X^\mu\rangle
\quad\text{a.s.}
\end{equation}
If $\alpha(t)=\breve\alpha(t)$, then \eqref{eq:rep-eq1} gives $\langle\breve\beta_t-\beta_t,X^\mu\rangle=0$ a.s., which implies $\breve\beta_t-\beta_t\in\operatorname{Ker}(\Gamma_X)$.
However, since $\breve\beta_t-\beta_t\in\operatorname{Ker}(\Gamma_X)^\perp$, it follows that $\|\breve\beta_t-\beta_t\|=0$, which contradicts the condition of the theorem.

If $\alpha(t)\neq\breve\alpha(t)$, then \eqref{eq:rep-eq1} represents a nonzero scalar multiple of the point evaluation $X^\mu(t)$ as an $L^2$-linear functional of $X^\mu$.
The characterization in Lemma 2.2 of the main paper is invariant under nonzero rescaling of this point evaluation.
Therefore, \eqref{eq:rep-eq1} implies $\sum_{k=1}^\infty\psi_k(t)^2<\infty$, contradicting Assumption I.5 of the main paper.
Hence $g(t)>0$ for every $t\in I$.
Since $g$ is continuous on the compact set $I$, it attains a strictly positive minimum on $I$.
Thus $\inf_{t\in I}g(t)>0$, which proves the assertion.
\end{proof}

\section{Proofs for Section 4 of the main paper -- Asymptotic Results}

\subsection{Auxiliary lemmas}

\begin{lemma} \label{lem:kerneleigenvaluedecay}
Let $K:[0,1]^2\to\mathbb R$ be a continuous, symmetric and positive semi‑definite kernel function.
Fix $d\in\mathbb N_0$ and assume the mixed derivative
$$
K_d(t,s)=\partial_t^{\,d}\partial_s^{\,d}K(t,s)
$$
exists and is continuous on $[0,1]^2$.
If there exist $0<\alpha\le1$ and $A\in L^1[0,1]$ such that
$$
|K_d(t,s+h)-K_d(t,s)|\le A(t)\,|h|^\alpha\qquad
 \text{for all} \ t,s\in[0,1],\;s+h\in[0,1],
$$
then the eigenvalues $(\lambda_k)_{k\ge1}$ of the associated integral operator, arranged in non-increasing order, satisfy
$$
\lambda_k = O\big(k^{-(2d+1+\alpha)}\big)\qquad(\text{as} \ k\to\infty).
$$
\end{lemma}

\begin{proof}
See \cite{cochran1988differentiable}, Theorem 1 together with their Main Theorem.
\end{proof}

\begin{lemma} \label{lem:studentized}
Let $X_1, \ldots, X_n$ be a sequence of i.i.d.\ copies of a stochastic process $X$ on $[0,1]$ with $\inf_{t \in [0,1]} \V(X(t)) > 0$ and $\sup_{t\in [0,1]} \E(|X(t)|^q) < \infty$ for some $q \geq 4$. Define the studentized variables $X_i^{\hat{\sigma}}(t) = (X_i(t) - \bar{X}(t))/\sqrt{\widehat{\sigma}_X(t,t)}$, $t \in [0,1]$. Then
$\sup_{n \geq 2} \sup_{1 \leq i \leq n} \sup_{t\in[0,1]}
\mathbb E|X_i^{\widehat\sigma}(t)|^{q/2}
<\infty$.
\end{lemma}

\begin{proof}
Note that bounding the numerator and denominator separately using standard Cauchy-Schwarz arguments leads to cases where the denominator is not bounded away from zero.
Therefore, we fix some $0 < \zeta < 0.5$ and treat the convenient event
$$
	G_n = \Big\{ |\bar X(t) - \E(X(t))| \leq \zeta \Big\} \cap \Big\{\widehat \sigma_X(t,t) \geq (1-\zeta) \V(X(t)) \Big\}
$$
separately from the troublesome event $G_n^c$, and we aim to show that $\E(|X_i^{\hat \sigma}(t)|^{q/2} \mathbbm{1}_{G_n}) < \infty$ and $\E(|X_i^{\hat \sigma}(t)|^{q/2} \mathbbm{1}_{G_n^c}) < \infty$.
For the convenient event, by \eqref{eq:elementaryinequality}
\begin{align*}
	\E(|X_i^{\hat \sigma}(t)|^{q/2} \mathbbm{1}_{G_n})
	&\leq \E\bigg( \Big| \frac{|X_i(t) - \E(X(t))| + |\E(X(t)) - \bar X(t) |}{\sqrt{\widehat \sigma_X(t,t)}} \Big|^{q/2}  \mathbbm{1}_{G_n}\bigg) \\
	&\leq \E\bigg( \frac{\big(|X_i(t)-\E(X(t))|+\zeta\big)^{q/2}}{(1-\zeta)^{q/4}\V(X(t))^{q/4}}  \bigg) \\
	&\leq 2^{q/2-1} \frac{\E(|X_i(t) - \E(X(t))|^{q/2}) + \zeta^{q/2}}{(1-\zeta)^{q/4} \V(X(t))^{q/4}}  < \infty
\end{align*}
because $q$-th moments of $X(t)$ are bounded and $\V(X(t)) > 0$.
For the troublesome event $G_n^c$, note that
$$
	|X_i^{\hat \sigma}(t)|^{q/2}
	= \bigg(\frac{|X_i(t) - \bar X(t)|^2}{\widehat \sigma_X(t,t)} \bigg)^{q/4}
	\leq n^{q/4},
$$
because $|X_i(t) - \bar X(t)|^2 \leq n \widehat \sigma_X(t,t)$.
Furthermore,
$$
	\widehat \sigma_X(t,t) = \frac{1}{n} \sum_{j=1}^n (X_j(t) - \E(X(t)))^2 - (\bar X(t) - \E(X(t)))^2,
$$
and let $Z_j(t) := (X_j(t) - \E(X(t)))^2 - \V(X(t))$.
Then,
\begin{align}
	&\E(|X_i^{\hat \sigma}(t)|^{q/2} \mathbbm{1}_{G_n^c})
	\leq n^{q/4} P(G_n^c) \nonumber \\
	&\leq n^{q/4} P\Big(|\bar X(t) - \E(X(t))| > \zeta\Big) +  n^{q/4} P\Big(\widehat \sigma_X(t,t) < (1-\zeta) \V(X(t))\Big) \nonumber \\
	&= n^{q/4} P\bigg( \Big|\sum_{j=1}^n (X_j(t) - \E(X(t))) \Big|^{q/2} > n^{q/2} \zeta^{q/2} \bigg) \label{eq:studentized1}  \\
	&\quad + n^{q/4} P\bigg( \frac{1}{n} \sum_{j=1}^n Z_j(t) - (\bar X(t) - \E(X(t)))^2 < -\zeta \V(X(t)) \bigg), \label{eq:studentized2}
\end{align}
and it remains to show that \eqref{eq:studentized1} and \eqref{eq:studentized2} are bounded.
For \eqref{eq:studentized1} we apply Markov's inequality:
\begin{align}
	&n^{q/4} P\bigg( \Big|\sum_{j=1}^n (X_j(t) - \E(X(t))) \Big|^{q/2} > n^{q/2} \zeta^{q/2} \bigg) \notag \\
	&\quad \leq \frac{\E(|\sum_{j=1}^n (X_j(t) - \E(X(t)))|^{q/2})}{\zeta^{q/2} n^{q/4}}. \label{eq:studentized1a}
\end{align}
To bound the numerator of \eqref{eq:studentized1a}, we apply Rosenthal's inequality (e.g., equation B.50 in \citealt{hansen2022econometrics}), which states that for any i.i.d. sequence $V_j$ of real-valued random variables with $\E(V_j) = 0$ then, for any $p \geq 2$ there exists a constant $R_p$ such that
\begin{align} \label{eq:rosenthal}
		\E\bigg(\Big|\sum_{j=1}^n V_j \Big|^{p}\bigg) \leq R_p \bigg( \Big( \sum_{j=1}^n \E(V_j^2) \Big)^{p/2} + \sum_{j=1}^n \E(|V_j|^p) \bigg).
\end{align}
Applied to $p=q/2$, we obtain
$$
E\bigg(\Big|\sum_{j=1}^n (X_j(t) - \E(X(t)))\Big|^{q/2}\bigg) \leq R_{q/2} n^{q/4} (\V(X(t))^{q/4} + \E(|X(t)|^{q/2})) = O(n^{q/4}),
$$
and because the denominator of \eqref{eq:studentized1a} is proportional to $n^{q/4}$, \eqref{eq:studentized1} is bounded.
To bound \eqref{eq:studentized2}, note that
\begin{align*}
& \frac{1}{n} \sum_{j=1}^n Z_j(t) - (\bar X(t) - \E(X(t)))^2 < -\zeta \V(X(t)) \\
	\Rightarrow \quad &  \frac{1}{n} \sum_{j=1}^n Z_j(t) < -\frac{\zeta}{2} \V(X(t)) \quad \text{or} \quad (\bar X(t) - \E(X(t)))^2 > \frac{\zeta}{2} \V(X(t))
\end{align*}
This holds because if both fail then $\frac{1}{n} \sum_{j=1}^n Z_j(t) - (\bar X(t) - \E(X(t)))^2 \geq -\zeta \V(X(t))$, contradicting our property of the previous line.
Therefore, \eqref{eq:studentized2} can be bounded as follows:
\begin{align}
	&P\bigg( \frac{1}{n} \sum_{j=1}^n Z_j(t) - (\bar X(t) - \E(X(t)))^2 < -\zeta \V(X(t)) \bigg) \nonumber \\
	&\leq P\Big( \Big|  \frac{1}{n} \sum_{j=1}^n Z_j(t) \Big| > \frac{\zeta}{2} \V(X(t)) \Big) + P\Big( \big| \bar X(t) - \E(X(t)) \big|^2 > \frac{\zeta}{2} \V(X(t)) \Big). \label{eq:studentized2b}
\end{align}
The second summand of \eqref{eq:studentized2b} is bounded using identical arguments as for \eqref{eq:studentized1}.
For the first summand of \eqref{eq:studentized2b}, abbreviate $z = \V(X(t))\zeta/2$.
Then, Markov's and Rosenthal's inequalities applied to the $(q/2)$-th power yield
\begin{align*}
	n^{q/4} P\bigg( \Big|  \frac{1}{n} \sum_{j=1}^n Z_j(t) \Big| >  z \bigg)
	&= n^{q/4} P\bigg( \Big| \sum_{j=1}^n Z_j(t) \Big|^{q/2} > n^{q/2} z^{q/2} \bigg) \\
	&\leq \frac{\E(|\sum_{j=1}^n Z_j(t)|^{q/2})}{n^{q/4} z^{q/2}} \\
	&\leq R_{q/2} z^{-q/2} (\E(Z_j(t)^{2})^{q/4} + n^{1-q/4} \E(|Z_j(t)|^{q/2})) \\
	&< \infty,
\end{align*}
where the last step holds because $\E(|Z_j(t)|^{q/2}) < \infty$ follows from $\E(|X(t)|^{q}) < \infty$.
Therefore,
\begin{align*}
	\E(|X_i^{\hat \sigma}(t)|^{q/2}) = \E(|X_i^{\hat \sigma}(t)|^{q/2} \mathbbm{1}_{G_n}) + \E(|X_i^{\hat \sigma}(t)|^{q/2} \mathbbm{1}_{G_n^c}) < \infty,
\end{align*}
and since all bounds in the above steps are uniform in $t$ and $n$, the assertion follows.
\end{proof}

\subsection{Proof of Lemma 4.1 of the main paper}

The $d$-th mixed partial derivative of $\sigma_X(t,s)$ is
$$
	\sigma_X^{(d)}(t,s) = \frac{\partial^d}{\partial s^d}\frac{\partial^d}{\partial t^d}
  \sigma_X(t,s) = \Cov(X^{(d)}(t), X^{(d)}(s)).
$$
Since $\E(\| X \|_{d,\kappa}^2) < \infty$, $X^{(d)}$ has $\kappa$-Hölder sample paths almost surely and we have
$$
	|X^{(d)}(t) - X^{(d)}(s)| \leq K_X |t-s|^\kappa, \qquad K_X = [X^{(d)}]_\kappa = \sup_{t \neq s} \frac{|X^{(d)}(t) - X^{(d)}(s)|}{|t-s|^{\kappa}},
$$
where $E(K_X^2) < \infty$ and $\sup_{r \in [0,1]} \E(X^{(d)}(r)^2) < \infty$.
Then, for every fixed $s$, by Cauchy-Schwarz,
\begin{align*}
	|\sigma_X^{(d)}(t+h,s) - \sigma_X^{(d)}(t,s)|
	&= \Big|\E\Big( (X^{(d)}(t+h) - X^{(d)}(t))(X^{(d)}(s) - \E(X^{(d)}(s))) \Big)\Big| \\
	&\leq \E\Big( |X^{(d)}(t+h) - X^{(d)}(t)|\ |X^{(d)}(s) - \E(X^{(d)}(s))| \Big) \\
	&\leq |h|^\kappa \E\Big( K_X\ |X^{(d)}(s) - \E(X^{(d)}(s))| \Big) \\
	&\leq |h|^\kappa \sqrt{\E(K_X^2)} \sqrt{ \sup_{r \in [0,1]} \E(X^{(d)}(r)^2) } \\
	&= A |h|^\kappa,
\end{align*}
where the deterministic function
$$
A = \sqrt{\E(K_X^2)} \sqrt{ \sup_{r \in [0,1]} \E(X^{(d)}(r)^2) }
$$
is finite and does not depend on $t$. Hence $A(\cdot) \in L^1[0,1]$, and the assertion follows with Lemma \ref{lem:kerneleigenvaluedecay} together with the fact that $\lambda_{\delta,t,j} \leq \lambda_j$ by equation (2.7) of the main paper.
\hfill$\square$

\subsection{Proof of Lemma 4.2 of the main paper}

The following elementary inequality will be used here and in subsequent proofs several times:
\begin{align} \label{eq:elementaryinequality}
\bigg|\sum_{i=1}^n x_i\bigg|^{p} \leq n^{p-1} \sum_{i=1}^n |x_i|^p, \quad p \geq 1, \quad x_i \in \mathbb R.
\end{align}
With $X_i^{\hat \sigma}(t) = (X_i(t)-\bar X(t))/\sqrt{\widehat \sigma_X(t,t)}$ denoting the studentized predictor variable, we have
\begin{align*}
	X_i^{\bar \mu}(t) \widehat v_t(s)
	&= \frac{1}{n} \sum_{j=1}^n \frac{X_j^{\bar \mu}(s) X_j^{\bar \mu}(t)}{\widehat \sigma_X(t,t)} X_i^{\bar \mu}(t)  \\
	&= \frac{1}{n} \sum_{j=1}^n \frac{X_j(s) X_j^{\bar \mu}(t)}{\widehat \sigma_X(t,t)} X_i^{\bar \mu}(t)
= \frac{1}{n} \sum_{j=1}^n  X_j(s) X_j^{\hat \sigma}(t) X_i^{\hat \sigma}(t) .
\end{align*}
Thus, our predictor variable admits the representation
\begin{align*}
	\widehat{\delta}_{i,t}(s) &= X_i(s) - \bar X(s) - X_i^{\hat \sigma}(t) W_{n,t}(s), \quad W_{n,t}(s) := \frac{1}{n} \sum_{j=1}^n X_j(s) X_j^{\hat \sigma}(t).
\end{align*}
By the triangle inequality and \eqref{eq:elementaryinequality},
$$
	\|\widehat{\delta}_{i,t}\|_{d,\kappa}^2 \leq 3 \Big( \|X_i\|_{d,\kappa}^2 + \|\bar X\|_{d,\kappa}^2 + |X_i^{\hat \sigma}(t)|^2 \, \|W_{n,t}\|_{d,\kappa}^2 \Big).
$$
Since $\E(\|X_i\|_{d,\kappa}^2) < \infty$ and $\E(\|\bar X\|_{d,\kappa}^2)< \infty$, it remains to show $\E(|X_i^{\hat \sigma}(t)|^2 \, \|W_{n,t}\|_{d,\kappa}^2) < \infty$.

By the Cauchy-Schwarz inequality together with the fact that $\frac{1}{n} \sum_{j=1}^n X_j^{\hat \sigma}(t)^2 = 1$ we get for $0 \leq \ell \leq d$,
$$
	\|W_{n,t}^{(\ell)}\|_\infty \leq \bigg( \frac{1}{n} \sum_{j=1}^n \| X_j^{(\ell)}\|_\infty^2 \bigg)^{1/2}, \qquad [W_{n,t}^{(d)}]_\kappa \leq \bigg( \frac{1}{n} \sum_{j=1}^n [X_j^{(d)}]_\kappa^2 \bigg)^{1/2},
$$
which implies
\begin{align*}
	\|W_{n,t}\|_{d,\kappa}^2 &\leq 2 \Big( \max_{0\leq \ell \leq d} \|W_{n,t}^{(\ell)}\|_\infty^2 + [W_{n,t}^{(d)}]_\kappa^2 \Big) \\
	&\leq \frac{2}{n} \sum_{j=1}^n \max_{0\leq \ell \leq d} \| X_j^{(\ell)}\|_\infty^2 + [X_j^{(d)}]_\kappa^2 \\
	&\leq \frac{2}{n} \sum_{j=1}^n \|X_j\|_{d,\kappa}^2.
\end{align*}
Therefore, by Hölder's inequality, finite sixth moments, and Lemma \ref{lem:studentized},
$$
	\E(|X_i^{\hat \sigma}(t)|^2 \, \|W_{n,t}\|_{d,\kappa}^2) \leq \frac{2}{n} \sum_{j=1}^n \E(|X_i^{\hat \sigma}(t)|^3)^{2/3} \E(\|X_j\|_{d,\kappa}^6)^{1/3} < \infty,
$$
and the assertion follows.
\hfill$\square$

\subsection{Auxiliary results for the proof of Theorem 4.3 of the main paper}

\begin{lemma} \label{lem:consistency.aux1}
Let $\widehat \sigma_{\delta,t}(u,w)= \frac{1}{n} \sum_{i=1}^n\widehat \delta_{i,t}(u)\,\widehat \delta_{i,t}(w)$ and let $\widehat{\Gamma}_{\delta,t}$ be the corresponding integral operator.
Then, under the conditions of Theorem 4.3 of the main paper,
$$
	\langle \widehat \beta_t - \beta_t, \widehat \Gamma_{\delta,t}(\widehat \beta_t - \beta_t) \rangle = O_P(n^{-\varphi}) \quad \text{as} \ n \to \infty.
$$
\end{lemma}

\begin{proof}
We apply Theorem~2 of \citet{crambes2009} to the scalar-on-function regression
with response $\widetilde Y_i^{\bar\mu}(t)$ and predictor
$\widehat\delta_{i,t}$, $i=1,\ldots,n$.
Recall from Section 3 of the main paper that using the observable response
$Y_i^{\bar\mu}(t)$ in the smoothing-spline criterion yields the same minimizer as
using the infeasible response $\widetilde Y_i^{\bar\mu}(t)$, because
$\sum_{i=1}^n X_i^{\bar\mu}(t)\widehat\delta_{i,t}(s)=0$ for all $s\in[0,1]$.
Thus the estimator $\widehat\beta_t$ can be analyzed as the smoothing-spline
estimator in the regression
$
    \widetilde Y_i^{\bar\mu}(t)
    =
    \langle \beta_t,\widehat\delta_{i,t}\rangle
    +
    \varepsilon_i^{\bar\mu}(t)
$.

The criterion in Section 3 of the main paper uses the Riemann-sum weight $1/p$ inherited from the original grid on $[0,1]$, rather than $1/p_t$.
For fixed $t>0$, however, $p_t/p\to t$, so this change only affects multiplicative constants.
Equivalently, in the finite-dimensional representation the smoothing parameter is replaced by the parameter $\rho p/p_t$, which is of the same order as $\rho$.

Theorem~2 of \citet{crambes2009} is stated for a general random design of
functional predictors and does not require independence of the predictor curves.
Its finite-sample bounds are conditional on the design matrix, while the regularity
assumptions on the predictor process are used to control the discretization and
empirical approximation terms. We therefore verify the corresponding assumptions
for the estimated residualized predictors $\widehat\delta_{i,t}$.
In the notation adapted to the interval $[0,t]$, these conditions are:
\begin{itemize}
\item[] (CKS-1) \ $\beta_t$ is $m$-times differentiable and its $m$-th derivative belongs to $L^2[0,t]$.
\item[] (CKS-2) \ There exists some constant $\tilde \kappa$, $0 < \tilde \kappa \leq 1$, such that for every $\epsilon > 0$, there exists a constant $C_1 < \infty$ such that
\begin{align*}
	P(|\widehat \delta_{i,t}(r) - \widehat \delta_{i,t}(s)| \leq C_1 |r-s|^{\tilde \kappa}, \ \forall r,s \in [0,t]) \geq 1-\epsilon.
\end{align*}
\item[] (CKS-3) \ For some constants $\tilde q > 0$ and $C_2 < \infty$ and all $K \in \mathbb N$ there is a $K$-dimensional linear subspace $S_K$ of $L^2[0,t]$ with
$$
	E\bigg( \inf_{f \in \mathcal S_K} \sup_{s \in [0,t]} |\widehat \delta_{i,t}(s) - f(s)|^2 \bigg) \leq C_2 K^{-2\tilde q}.
$$
\item[] (CKS-S) With $\rho_t:=\rho p/p_t$, we have $\rho_t\to0$ and $1/(n\rho_t)\to0$ as $n\to\infty$.
\item[] (CKS-G) $np_t^{-2 \tilde \kappa} = O(1)$ as $n,p \to \infty$.
\end{itemize}
Once these conditions are verified, Theorem~2 of \citet{crambes2009} gives
$$
    \langle \widehat\beta_t-\beta_t,
    \widehat\Gamma_{\delta,t}(\widehat\beta_t-\beta_t)\rangle
    =
    O_P(n^{-\varphi}),
$$
with $\varphi=(2d+\kappa+2m+1)/(2d+\kappa+2m+2)$.

(CKS-1) follows from Assumption A.3 of the main paper.
(CKS-2) follows with $\tilde \kappa = 1$ if $d \geq 1$ and  $\tilde \kappa = \kappa$ if $d = 0$ because $\widehat \delta_{i,t}$ has $C^{d,\kappa}[0,1]$-valued sample paths a.s.\ by Lemma 4.2 of the main paper.
That is, $\tilde \kappa = \max(\kappa, \min(1,d))$.
Specifically, if $d \geq 1$ then $\widehat \delta_{i,t}$ has $C^{0,1}[0,t]$-valued sample paths a.s., and, if $d = 0$, then $\widehat \delta_{i,t}$ has $C^{0,\kappa}[0,t]$-valued paths a.s.
Therefore $\widehat \delta_{i,t}$ has $C^{0,\tilde \kappa}[0,t]$-valued paths a.s.\ with the random Hölder constant
$$
	K = \sup_{\substack{r \neq s \\ r,s \in [0,t]}} \frac{|\widehat \delta_{i,t}(r) - \widehat \delta_{i,t}(s)|}{|r-s|^{\tilde \kappa}} \leq \|\widehat \delta_{i,t}\|_{0,\tilde \kappa}
$$
so that
$$
	|\widehat \delta_{i,t}(r) - \widehat \delta_{i,t}(s)| \leq K |r-s|^{\tilde \kappa} \qquad \text{for all} \ r,s \in [0,t], \quad \text{almost surely}.
$$
Since $K$ is finite almost surely, choose $C_1 < \infty$ such that $P(K \leq C_1) \geq 1-\epsilon$. Then,
\begin{align*}
	P(|\widehat \delta_{i,t}(r) - \widehat \delta_{i,t}(s)| \leq C_1 |r-s|^{\tilde \kappa}, \ \forall r,s \in [0,t])
	= P(K \leq C_1) \geq 1-\epsilon.
\end{align*}
To verify (CKS-3), we employ Lemma 1 of \cite{crambes2009}, which indicates that a sufficient condition for (CKS-3) with $\tilde q = d + \gamma$ is:
\begin{align} \label{eq:CKS3suffcond}
	\E\bigg(\sup_{|r-s| \leq h} |\widehat \delta^{(d)}_{i,t}(r) - \widehat \delta^{(d)}_{i,t}(s)|^2 \bigg) \leq C_2 h^{2\gamma}\qquad \text{for any} \ h > 0, \quad \text{for some} \ \gamma > 0.
\end{align}
Following the proof of Lemma 4.2 of the main paper, our predictor variable admits the representation $\widehat{\delta}_{i,t}(s) = X_i(s) - \bar X(s) - \frac{1}{n} \sum_{j=1}^n X_j(s) X_j^{\hat \sigma}(t) X_i^{\hat \sigma}(t)$ and has the $d$-th derivative
$$
	\widehat{\delta}_{i,t}^{(d)}(s)
= X_i^{(d)}(s) - \frac{1}{n} \sum_{j=1}^n X_j^{(d)}(s) - \frac{1}{n} \sum_{j=1}^n X_j^{(d)}(s) X_j^{\hat \sigma}(t) X_i^{\hat \sigma}(t).
$$
With inequality \eqref{eq:elementaryinequality},
\begin{align*}
	|\widehat{\delta}_{i,t}^{(d)}(r) - \widehat{\delta}_{i,t}^{(d)}(s)|^{2}  &\leq 3 |X_i^{(d)}(r) - X_i^{(d)}(s)|^{2} + \frac{3}{n} \sum_{j=1}^n |X_j^{(d)}(r) - X_j^{(d)}(s)|^{2} \\
	&\qquad + 3 |X_i^{\hat \sigma}(t)|^2 \Big| \frac{1}{n}\sum_{j=1}^n (X_j^{(d)}(r) - X_j^{(d)}(s))X_j^{\hat \sigma}(t) \Big|^{2},
\end{align*}
where Cauchy--Schwarz and the fact that
$n^{-1}\sum_{j=1}^n|X_j^{\hat\sigma}(t)|^2=1$ give
\begin{align*}
\left|\frac{1}{n}\sum_{j=1}^n
\big(X_j^{(d)}(r)-X_j^{(d)}(s)\big)
X_j^{\hat\sigma}(t)\right|^2
&\leq \frac{1}{n}\sum_{j=1}^n
\left|X_j^{(d)}(r)-X_j^{(d)}(s)\right|^2.
\end{align*}
Thus, by Hölder's inequality,
\begin{align*}
	&\E\bigg(\sup_{|r-s| \leq h} |\widehat \delta^{(d)}_{i,t}(r) - \widehat \delta^{(d)}_{i,t}(s)|^2 \bigg) \\
	&\leq 3 \E\bigg(\sup_{|r-s| \leq h} |X_i^{(d)}(r) - X_i^{(d)}(s)|^{2}\bigg) + \frac{3}{n} \sum_{j=1}^n \E\bigg(\sup_{|r-s| \leq h} |X_j^{(d)}(r) - X_j^{(d)}(s)|^{2}\bigg) \\
	&\qquad + \frac{3}{n}\sum_{j=1}^n  \E\bigg(\sup_{|r-s| \leq h} |X_j^{(d)}(r) - X_j^{(d)}(s)|^6 \bigg)^{1/3} \E\Big( |X_i^{\hat \sigma}(t)|^{3} \Big)^{2/3}.
\end{align*}
Because $X$ has uniformly bounded sixth moments, Lemma \ref{lem:studentized} implies $E( |X_i^{\hat \sigma}(t)|^{3}) < \infty$, and by the i.i.d.\ property from Assumption A.6 of the main paper,
\begin{align} \label{eq:CKS3.aux}
E\bigg(\sup_{|r-s| \leq h} |\widehat \delta^{(d)}_{i,t}(r) - \widehat \delta^{(d)}_{i,t}(s)|^2 \bigg) \leq \tilde C_2  \E\bigg(\sup_{|r-s| \leq h} |X^{(d)}(r) - X^{(d)}(s)|^{6}\bigg)^{1/3}
\end{align}
For the remaining term of \eqref{eq:CKS3.aux}, for any $h > 0$ and $r,s$ with $|r-s| \leq h$,
$$
	|X^{(d)}(r) - X^{(d)}(s)| \leq L |r-s|^\kappa \leq L h^\kappa,
$$
and
$$
	\sup_{|r-s| \leq h} |X^{(d)}(r) - X^{(d)}(s)| \leq L h^\kappa,
$$
where the random Hölder constant
$$
	L = \sup_{r \neq s} \frac{|X^{(d)}(r) - X^{(d)}(s)|}{|r-s|^{\kappa}} \leq \|X\|_{d,\kappa}
$$
satisfies $\E(L^6) < \infty$ by Assumption A.2 of the main paper.
Taking sixth powers and expectations gives
$$
	\E\bigg(\sup_{|r-s| \leq h} |X^{(d)}(r) - X^{(d)}(s)|^6 \bigg) \leq \E(L^6) h^{6\kappa}.
$$
Therefore, with \eqref{eq:CKS3.aux},
$$
	E\bigg(\sup_{|r-s| \leq h} |\widehat \delta^{(d)}_{i,t}(r) - \widehat \delta^{(d)}_{i,t}(s)|^2 \bigg) \leq \tilde C \E(L^6)^{1/3} h^{2\kappa},
$$
so condition \eqref{eq:CKS3suffcond} holds with $C_2 = \tilde C \E(L^6)^{1/3}$ and $\gamma = \kappa$.
Consequently, (CKS-3) holds with $\tilde q = d + \kappa$.
Finally, since $p/p_t=O(1)$ for fixed $t>0$, our condition on $\rho$ implies the (CKS-S) condition for $\rho_t=\rho p/p_t$.
Further, note that, for any fixed $t \in (0,1]$, our conditions imply that $p_t \geq c \cdot p$ for some constant $c$, uniformly in $n$, so that $n^{1/(2\kappa)}/p = O(1)$ implies the (CKS-G) condition.
\end{proof}

\begin{lemma} \label{lem:consistency.aux2}
	Let
	\begin{align*}
	\widehat \sigma_{\delta,t}(r,s)
	=\frac{1}{n}\sum_{i=1}^n
	\widehat \delta_{i,t}(r)\widehat \delta_{i,t}(s),\quad 
	\widetilde \sigma_{\delta,t}(r,s)
	=\frac{1}{n}\sum_{i=1}^n
	\delta_{i,t}(r)\delta_{i,t}(s),
	\end{align*}
	where
	\begin{align*}
	\widehat \delta_{i,t}(s)
	=X_i^{\bar \mu}(s)-\widehat v_t(s)X_i^{\bar \mu}(t),\quad
	\delta_{i,t}(s)
	=X_i^{\mu}(s)-v_t(s)X_i^{\mu}(t).
	\end{align*}
	Let $\widehat \Gamma_{\delta,t}$ and $\widetilde \Gamma_{\delta,t}$ be the corresponding integral operators. Then, for any $f \in L^2[0,1]$, under Assumptions A.1, A.2, and A.6 of the main paper,
$$
	\langle f, \widehat \Gamma_{\delta,t}(f) \rangle - \langle f, \widetilde \Gamma_{\delta,t}(f) \rangle = O_P(n^{-1})  \quad \text{as} \ n \to \infty.
$$
\end{lemma}

\begin{proof}
We first derive an identity for the difference between the two covariance kernels.
Since $\widehat\sigma_X(r,t)=\widehat\sigma_X(t,t)\widehat v_t(r)$, the feasible kernel is
$$
\widehat\sigma_{\delta,t}(r,s)
=
\widehat\sigma_X(r,s)
-
\widehat\sigma_X(t,t)\widehat v_t(r)\widehat v_t(s).
$$
For the infeasible kernel, write $\bar X^\mu(s)=\bar X(s)-\E(X(s))$ and $\bar\delta_t(s)=\bar X^\mu(s)-v_t(s)\bar X^\mu(t)$, and note that $\widetilde\sigma_X(r,s)=\widehat\sigma_X(r,s)+\bar X^\mu(r)\bar X^\mu(s)$.
Then,
\begin{align*}
\widetilde\sigma_{\delta,t}(r,s)
&=
\widetilde\sigma_X(r,s)
-
v_t(s)\widetilde\sigma_X(r,t)
-
v_t(r)\widetilde\sigma_X(s,t)
+
v_t(r)v_t(s)\widetilde\sigma_X(t,t) \\
&=
\widehat\sigma_X(r,s)
-
v_t(s)\widehat\sigma_X(r,t)
-
v_t(r)\widehat\sigma_X(s,t)
+
v_t(r)v_t(s)\widehat\sigma_X(t,t)  \\
&\quad
+
\big( \bar X^\mu(r)-v_t(r)\bar X^\mu(t) \big)
\big( \bar X^\mu(s)-v_t(s)\bar X^\mu(t) \big) \\
&=
\widehat\sigma_X(r,s)
-
\widehat\sigma_X(t,t)v_t(s)\widehat v_t(r)
-
\widehat\sigma_X(t,t)v_t(r)\widehat v_t(s)
+
\widehat\sigma_X(t,t)v_t(r)v_t(s) 
+
\bar\delta_t(r)\bar\delta_t(s).
\end{align*}
Subtracting the last display from the expression for
$\widehat\sigma_{\delta,t}(r,s)$ gives
\begin{align*}
\widehat\sigma_{\delta,t}(r,s)-\widetilde\sigma_{\delta,t}(r,s)
&=
-\widehat\sigma_X(t,t)
\big(\widehat v_t(r)\widehat v_t(s)
-v_t(s)\widehat v_t(r)
-v_t(r)\widehat v_t(s)
+v_t(r)v_t(s)\big) 
-\bar\delta_t(r)\bar\delta_t(s) \\
&=
-\widehat\sigma_X(t,t)
\big(\widehat v_t(r)-v_t(r)\big)
\big(\widehat v_t(s)-v_t(s)\big)
-
\bar\delta_t(r)\bar\delta_t(s).
\end{align*}
It follows that
$$
\|\widehat\Gamma_{\delta,t}-\widetilde\Gamma_{\delta,t}\|_{\mathcal S}
\leq
|\widehat\sigma_X(t,t)|\,\|\widehat v_t-v_t\|^2
+
\|\bar\delta_t\|^2,
$$
where $\|\cdot\|_{\mathcal S}$ denotes the Hilbert--Schmidt norm.

By the i.i.d.\ property and since $X$ has uniformly bounded fourth moments, standard law of large numbers arguments imply $\|\bar X - \E(X)\| = O_P(n^{-1/2})$ and $\|\widehat \Gamma_X - \Gamma_X\|_{\mathcal S} = O_P(n^{-1/2})$.
Hence $\widehat\sigma_X(t,t)=O_P(1)$ and
$\widehat\sigma_X(t,t)-\sigma_X(t,t)=O_P(n^{-1/2})$. Since
$\sigma_X(t,t)>0$, we obtain
$
\|\widehat v_t-v_t\|=O_P(n^{-1/2})
$.
Moreover,
$$
\|\bar\delta_t\|
\leq
\|\bar X^\mu\|+|\bar X^\mu(t)|\,\|v_t\|
=
O_P(n^{-1/2}),
$$
because $\|v_t\|<\infty$ for fixed $t$. Therefore,
$$
\|\widehat\Gamma_{\delta,t}-\widetilde\Gamma_{\delta,t}\|_{\mathcal S}
=
O_P(n^{-1}).
$$
Finally, by Cauchy--Schwarz,
$$
|
\langle f,\widehat\Gamma_{\delta,t}(f)\rangle
-
\langle f,\widetilde\Gamma_{\delta,t}(f)\rangle
|
\leq
\|f\|^2
\|\widehat\Gamma_{\delta,t}-\widetilde\Gamma_{\delta,t}\|_{\mathcal S}
=
O_P(n^{-1})
$$
for every fixed $f\in L^2[0,1]$.
\end{proof}

\subsection{Proof of Theorem 4.3(a) of the main paper}

\begin{proof}
Let $f_n=\widehat\beta_t-\beta_t$.
From Lemmas~\ref{lem:consistency.aux1} and \ref{lem:consistency.aux2}, together with $\|f_n\|=O_P(1)$, it follows that
\begin{align} \label{eq:betaconsistency.aux1}
	A_n
	:=
	\langle f_n,\widetilde\Gamma_{\delta,t}(f_n)\rangle
	=
	O_P(n^{-\varphi}),
\end{align}
where $\widetilde\Gamma_{\delta,t}$ is the integral operator with kernel
$$
	\widetilde \sigma_{\delta,t}(r,s)
	=
	\frac{1}{n}\sum_{i=1}^n \delta_{i,t}(r)\delta_{i,t}(s),
	\qquad
	\delta_{i,t}(s)=X_i^\mu(s)-X_i^\mu(t)v_t(s).
$$
Let $(\lambda_{\delta,t,k},\psi_{\delta,t,k})_{k\in\mathbb N}$ denote the positive
eigenvalues and corresponding orthonormal eigenfunctions of $\Gamma_{\delta,t}$,
arranged in non-increasing order. Define
$$
	b_{t,k}:=\langle f_n,\psi_{\delta,t,k}\rangle,
	\qquad
	\zeta_{i,t,k}:=\langle \delta_{i,t},\psi_{\delta,t,k}\rangle .
$$
Then
$$
	\E(\zeta_{i,t,k})=0,
	\qquad
	\E(\zeta_{i,t,k}\zeta_{i,t,l})
	=
	\lambda_{\delta,t,k}\mathbbm 1_{\{k=l\}},
$$
and
\begin{align} \label{eq:betaconsistency.aux2}
	B_n
	:=
	\langle f_n,\Gamma_{\delta,t}(f_n)\rangle
	=
	\sum_{k=1}^\infty b_{t,k}^2\lambda_{\delta,t,k}.
\end{align}
Moreover,
$$
	A_n
	=
	\frac{1}{n}\sum_{i=1}^n
	\bigg(\sum_{k=1}^\infty b_{t,k}\zeta_{i,t,k}\bigg)^2.
$$
Therefore,
\begin{align} \label{eq:betaconsistency.aux3}
	A_n-B_n
	&=
	\sum_{k=1}^\infty\sum_{l=1}^\infty
	b_{t,k}b_{t,l}
	\bigg(
	\frac{1}{n}\sum_{i=1}^n\zeta_{i,t,k}\zeta_{i,t,l}
	-
	\lambda_{\delta,t,k}\mathbbm 1_{\{k=l\}}
	\bigg).
\end{align}
We next control the right-hand side of \eqref{eq:betaconsistency.aux3}. For $k\leq l$, set
$$
	\xi_{t,k,l}
	=
	\frac{1}{\sqrt n(\lambda_{\delta,t,k}\lambda_{\delta,t,l})^{1/2}}
	\sum_{i=1}^n
	\Big(
	\zeta_{i,t,k}\zeta_{i,t,l}
	-
	\lambda_{\delta,t,k}\mathbbm 1_{\{k=l\}}
	\Big).
$$
By Assumption A.2 of the main paper, applied with
$g=\psi_{\delta,t,k}$ and $h=\psi_{\delta,t,l}$, we have
$$
	\E(\zeta_{i,t,k}^2\zeta_{i,t,l}^2)
	\leq
	C\lambda_{\delta,t,k}\lambda_{\delta,t,l}.
$$
Consequently,
$$
	\sup_{k\leq l}\E(\xi_{t,k,l}^2)<\infty.
$$
Using symmetry of the double sum in \eqref{eq:betaconsistency.aux3}, we obtain
\begin{align} \label{eq:betaconsistency.aux4}
	|A_n-B_n|
	&\leq
	\frac{2}{\sqrt n}
	\bigg|
	\sum_{k=1}^\infty\sum_{l=k}^\infty
	b_{t,k}b_{t,l}
	(\lambda_{\delta,t,k}\lambda_{\delta,t,l})^{1/2}
	\xi_{t,k,l}
	\bigg|.
\end{align}
Let $K_n=\lfloor n^{1/2}\rfloor$. Splitting the double sum into the parts
$k\leq K_n$ and $k>K_n$ and applying Cauchy-Schwarz gives
\begin{align}
	|A_n-B_n|
	&\leq
	\frac{2}{\sqrt n}
	\bigg(
	\sum_{k=1}^{K_n}\sum_{l=k}^\infty
	\lambda_{\delta,t,k}b_{t,k}^2b_{t,l}^2
	\bigg)^{1/2}
	\bigg(
	\sum_{k=1}^{K_n}\sum_{l=k}^\infty
	\lambda_{\delta,t,l}\xi_{t,k,l}^2
	\bigg)^{1/2} \nonumber \\
	&\quad+
	\frac{2}{\sqrt n}
	\bigg(
	\sum_{k=K_n+1}^\infty\sum_{l=k}^\infty
	b_{t,k}^2b_{t,l}^2
	\bigg)^{1/2}
	\bigg(
	\sum_{k=K_n+1}^\infty\sum_{l=k}^\infty
	\lambda_{\delta,t,k}\lambda_{\delta,t,l}\xi_{t,k,l}^2
	\bigg)^{1/2}.
	\label{eq:betaconsistency.aux5}
\end{align}
For the first factor in the first term of
\eqref{eq:betaconsistency.aux5}, we have
$$
	\sum_{k=1}^{K_n}\sum_{l=k}^\infty
	\lambda_{\delta,t,k}b_{t,k}^2b_{t,l}^2
	\leq
	\bigg(\sum_{k=1}^\infty \lambda_{\delta,t,k}b_{t,k}^2\bigg)
	\bigg(\sum_{l=1}^\infty b_{t,l}^2\bigg)
	=
	B_n\|f_n\|^2.
$$
For the second factor, the uniform second-moment bound for $\xi_{t,k,l}$ yields
$$
	\E\bigg(
	\sum_{k=1}^{K_n}\sum_{l=k}^\infty
	\lambda_{\delta,t,l}\xi_{t,k,l}^2
	\bigg)
	\leq
	C\sum_{k=1}^{K_n}\sum_{l=k}^\infty \lambda_{\delta,t,l}.
$$
Since $\lambda_{\delta,t,l}\leq \lambda_l$ and $\lambda_l=O(l^{-2})$ by
Assumption A.5 of the main paper,
$$
	\sum_{k=1}^{K_n}\sum_{l=k}^\infty \lambda_{\delta,t,l}
	\leq
	C\sum_{k=1}^{K_n}\sum_{l=k}^\infty l^{-2}
	=
	O(\log(K_n))
	=
	O(\log(n)).
$$
Hence
$$
	\bigg(
	\sum_{k=1}^{K_n}\sum_{l=k}^\infty
	\lambda_{\delta,t,l}\xi_{t,k,l}^2
	\bigg)^{1/2}
	=
	O_P((\log(n))^{1/2}).
$$
Since $\|f_n\|=O_P(1)$, the first term in \eqref{eq:betaconsistency.aux5} is therefore
\begin{align} \label{eq:betaconsistency.lowfreq}
	O_P\Big(n^{-1/2}B_n^{1/2}(\log(n))^{1/2}\Big).
\end{align}
For the second term in \eqref{eq:betaconsistency.aux5}, we have
$$
	\sum_{k=K_n+1}^\infty\sum_{l=k}^\infty b_{t,k}^2b_{t,l}^2
	\leq
	\|f_n\|^4
	=
	O_P(1).
$$
Moreover,
\begin{align*}
	\E\bigg(
	\sum_{k=K_n+1}^\infty\sum_{l=k}^\infty
	\lambda_{\delta,t,k}\lambda_{\delta,t,l}\xi_{t,k,l}^2
	\bigg)
	&\leq
	C\sum_{k=K_n+1}^\infty\sum_{l=k}^\infty
	\lambda_{\delta,t,k}\lambda_{\delta,t,l}
	\leq
	C\bigg(\sum_{k=K_n+1}^\infty\lambda_{\delta,t,k}\bigg)^2.
\end{align*}
Again using $\lambda_{\delta,t,k}\leq\lambda_k=O(k^{-2})$, we get
$$
	\sum_{k=K_n+1}^\infty\lambda_{\delta,t,k}
	=
	O(K_n^{-1}).
$$
Thus the second term in \eqref{eq:betaconsistency.aux5} is
\begin{align} \label{eq:betaconsistency.highfreq}
	O_P(n^{-1/2}K_n^{-1})
	=
	O_P(n^{-1}).
\end{align}
Combining \eqref{eq:betaconsistency.aux5},
\eqref{eq:betaconsistency.lowfreq}, and \eqref{eq:betaconsistency.highfreq}, we obtain
\begin{align} \label{eq:betaconsistency.transfer}
	|A_n-B_n|
	=
	O_P\Big(n^{-1/2}B_n^{1/2}(\log(n))^{1/2}\Big)
	+
	O_P(n^{-1}).
\end{align}
Since $A_n=O_P(n^{-\varphi})$ by \eqref{eq:betaconsistency.aux1}, it follows that
$$
	B_n
	\leq
	A_n+|A_n-B_n|
	=
	O_P(n^{-\varphi})
	+
	O_P\Big(n^{-1/2}B_n^{1/2}(\log(n))^{1/2}\Big)
	+
	O_P(n^{-1}).
$$
Using $2ab\leq a^2+b^2$, the middle term can be absorbed into the left-hand side:
$$
	O_P\Big(n^{-1/2}B_n^{1/2}(\log(n))^{1/2}\Big)
	\leq
	\frac12 B_n + O_P(n^{-1}\log(n)).
$$
Consequently,
$$
	B_n
	=
	O_P(n^{-\varphi})+O_P(n^{-1}\log(n)).
$$
Since $\varphi<1$, we have $n^{-1}\log(n)=o(n^{-\varphi})$. Therefore,
$$
	\langle \widehat\beta_t-\beta_t,
	\Gamma_{\delta,t}(\widehat\beta_t-\beta_t)\rangle
	=
	B_n
	=
	O_P(n^{-\varphi}),
$$
which proves the assertion.
\end{proof}

\subsection{Proof of Theorem 4.3(b) of the main paper}

The difference of the parameter and the estimator of interest can be decomposed as
\begin{align*}
	\widehat \alpha(t) - \alpha(t)
	&= \widehat \alpha^\star(t) - \langle \widehat \beta_t, \widehat v_t \rangle - \alpha^\star(t) + \langle \beta_t, v_t \rangle  \\
	&= (\widehat \alpha^\star(t) - \alpha^\star(t)) + \langle \widehat \beta_t, v_t - \widehat v_t \rangle + \langle \beta_t - \widehat \beta_t, v_t \rangle.
\end{align*}
It is easy to verify that $\widehat \alpha^\star(t) - \alpha^\star(t)$ and $\langle \widehat \beta_t, v_t - \widehat v_t \rangle$ are $O_P(n^{-1/2})$ since $X_1, \ldots X_n$ are i.i.d.\ and have enough bounded moments.
Therefore, it remains to show that $\langle \beta_t - \widehat \beta_t, v_t \rangle = O_P(n^{-\varphi/(2\tau)})$.
Define $b_{j,t} := \langle \beta_t - \widehat \beta_t, \psi_j \rangle$ and $\tilde b_{j,t} := b_{j,t} - \langle \beta_t - \widehat \beta_t, v_t \rangle \psi_j(t)$.
Then, for any $k \in \mathbb N$ with $\sum_{j=1}^k (\psi_j(t))^2 > 0$, we have
\begin{align}
	\langle \beta_t - \widehat \beta_t, v_t \rangle^2
	&= \frac{\sum_{j=1}^{k} \langle \beta_t - \widehat \beta_t, v_t \rangle^2 (\psi_j(t))^2}{\sum_{j=1}^{k} (\psi_j(t))^2} \nonumber \\
	&= \frac{\sum_{j=1}^k (\tilde b_{j,t} - b_{j,t})^2}{\sum_{j=1}^k (\psi_j(t))^2} \leq \frac{\Big(\sqrt{\sum_{j=1}^{k} \tilde b_{j,t}^2 } + \sqrt{\sum_{j=1}^{k} b_{j,t}^2 }\,\Big)^2}{\sum_{j=1}^{k} (\psi_j(t))^2},  \label{eq:thm3aux1}
\end{align}
which follows from the Cauchy-Schwarz inequality.
Three terms appear in the right-hand side expression of \eqref{eq:thm3aux1}.
For the first term, note that
$$
	\sum_{j=1}^\infty \lambda_j \psi_j(t)^2 = \sigma_X(t,t),
	\quad \langle \beta_t - \widehat \beta_t, v_j \rangle = \frac{\sum_{j=1}^\infty \lambda_j \psi_j(t) b_{j,t}}{\sigma_X(t,t)},
$$
and
$$
	\langle \beta_t - \widehat \beta_t, \Gamma_X(\beta_t - \widehat \beta_t) \rangle = \sum_{j=1}^\infty \lambda_j b_{j,t}^2
$$
which implies
\begin{align*}
	\sum_{j=1}^\infty \lambda_{j} \tilde b_{j,t}^2
	&= \sum_{j=1}^\infty \lambda_{j} b_{j,t}^2 - 2 \langle \beta_t - \widehat \beta_t, v_t \rangle \sum_{j=1}^\infty \lambda_{j} b_{j,t} \psi_j(t) + \langle \beta_t - \widehat \beta_t, v_t \rangle^2 \sum_{j=1}^\infty \lambda_{j} (\psi_j(t))^2 \\
	&= \langle \beta_t - \widehat \beta_t, \Gamma_X(\beta_t - \widehat \beta_t) \rangle - \langle \beta_t - \widehat \beta_t, v_t \rangle^2 \sigma_X(t,t)  \\
	&= \langle \beta_t - \widehat \beta_t, \Gamma_{\delta,t}(\beta_t - \widehat \beta_t) \rangle = O_P(n^{-\varphi}),
\end{align*}
by Theorem 4.3(a) of the main paper and the fact that $\langle f, \Gamma_{\delta,t}(f) \rangle = \langle f, \Gamma_X(f)\rangle - \sigma_X(t,t) \langle f, v_t \rangle^2$ for any $f \in L^2[0,1]$.
Therefore, by Assumption A.5 of the main paper,
\begin{align*}
	\sum_{j=1}^k \tilde b_{j,t}^2
	\leq  \lambda_{k}^{-1} \sum_{j=1}^k \lambda_{j} \tilde b_{j,t}^2
	 \leq \lambda_{k}^{-1} \sum_{j=1}^\infty \lambda_{j} \tilde b_{j,t}^2 = \lambda_{k}^{-1} \langle \beta_t - \widehat \beta_t, \Gamma_{\delta,t}(\beta_t - \widehat \beta_t) \rangle = O_P(k^\tau n^{-\varphi}).
\end{align*}
Setting $k$ proportional to $n^{\varphi/\tau}$ yields $\sum_{j=1}^k \tilde b_{j,t}^2 = O_P(1)$.
For the second term of \eqref{eq:thm3aux1}, because $(\psi_j)_{j \in \mathbb N}$ is an orthonormal sequence, by Bessel's inequality,
$$
\sum_{j=1}^k b_{j,t}^2 \leq \sum_{j=1}^\infty b_{j,t}^2
\leq \|\beta_t - \widehat \beta_t\|^2 = O_P(1).
$$
Finally, for the third term of \eqref{eq:thm3aux1}, with $k$ proportional to $n^{\varphi/\tau}$, by Assumption A.5 of the main paper, we obtain
\begin{align*}
	\bigg(\sum_{j=1}^{k} \psi_j(t)^2\bigg)^{-1} \leq C \frac{1}{k} = O(n^{-\varphi/\tau}).
\end{align*}
Therefore, $\langle \beta_t - \widehat \beta_t, v_t \rangle = O_P(n^{-\varphi/(2\tau)})$, and the assertion follows.
\hfill$\square$

\subsection{Proof of Theorem 4.4(a) of the main paper}

First note that $\alpha(t)=\widetilde\alpha^\star(t)-\langle \widehat v_t,\beta_t\rangle$ and $\widehat\alpha(t)=\widehat\alpha^\star(t)-\langle \widehat v_t,\widehat\beta_t\rangle$. With
$$
R_1(X_i):=\widetilde\alpha^\star(t)(X_i(t)-\E(X(t)))-
\widehat\alpha^\star(t) X_i^{\bar \mu}(t),
$$
$d_X(t) := \bar X(t) - \E(X(t)) =O_p(n^{-1/2})$ and
the fact that $\widehat \alpha^\star(t) - \alpha^\star(t) = O_P(n^{-1/2})$ by straightforward arguments, we get
$$
\frac{1}{n} \sum_{i=1}^nR_1(X_i)^2=O_p(n^{-1}).
$$
Furthermore,
$R_2:=\E(Y(t))-\bar Y(t)=O_P(n^{-1/2})$, $R_3:=\langle \beta_t,d_X \rangle=O_P(n^{-1/2})$, and
$R_4:=\langle \widehat v_t,\beta_t\rangle d_X(t)=O_P(n^{-1/2})$. Therefore,
\begin{align*}
&	\frac{1}{n} \sum_{i=1}^n \left( \alpha_0(t)+ \alpha(t)X_i(t)+\langle X_i, \beta_t\rangle-\widehat\alpha_0(t)-\widehat\alpha(t)X_i(t)
-\langle X_i, \widehat\beta_t\rangle\right)^2\\
&=\frac{1}{n} \sum_{i=1}^n \left(\langle \beta_t-\widehat\beta_t,\widehat \delta_{i,t}\rangle+R_1(X_i)+R_2+R_3-R_4\right)^2
\\
&= \frac{1}{n} \sum_{i=1}^n \langle \beta_t-\widehat\beta_t,\widehat \delta_{i,t}\rangle^2+O_P\bigg(n^{-1/2}\Big(\frac{1}{n} \sum_{i=1}^n \langle \beta_t-\widehat\beta_t,\widehat \delta_{i,t}\rangle^2\Big)^{1/2}+n^{-1} \bigg),
\end{align*}
and, with
$$
	\frac{1}{n} \sum_{i=1}^n \langle \beta_t-\widehat\beta_t,\widehat \delta_{i,t}\rangle^2 = \langle \beta_t-\widehat\beta_t, \widehat \Gamma_{\delta,t}(\beta_t-\widehat\beta_t) \rangle,
$$
the assertion follows from Lemma \ref{lem:consistency.aux1}.
\hfill$\square$

\subsection{Proof of Theorem 4.4(b) of the main paper}

Define $\widetilde \beta_t := \beta_t - \widehat \beta_t$. The expression of interest is:
\begin{align}
	&m_t(X_{n+1}) -\widehat\alpha_0(t)-\widehat\alpha(t)X_{n+1}(t) -\langle \widehat\beta_t , X_{n+1} \rangle \nonumber \\
	&= [\alpha_0(t) - \widehat \alpha_0(t)] + [\alpha(t) - \widehat \alpha(t)] X_{n+1}(t) + \langle \widetilde \beta_t, X_{n+1} \rangle. \label{eq:prederror.aux1}
\end{align}
With $d_X := \bar X - \E(X)$, $d_Y := \bar Y - \E(Y)$ and by the definition of $\alpha_0$ and $\widehat \alpha_0$ we have
\begin{align*}
	[\alpha_0(t) - \widehat \alpha_0(t)] = \widehat \alpha(t) d_X(t) + \langle \widehat \beta_t, d_X \rangle - d_Y(t) - [\alpha(t) - \widehat \alpha(t)] \E(X(t)) - \langle \widetilde \beta_t, \E(X) \rangle.
\end{align*}
Therefore, with $X_{n+1}^{\mu} := X_{n+1} - E(X)$, \eqref{eq:prederror.aux1} becomes
\begin{align} \label{eq:prederror.aux2}
	[\alpha(t) - \widehat \alpha(t)] X_{n+1}^{\mu}(t) + \langle \widetilde \beta_t, X_{n+1}^{\mu} \rangle + \widehat \alpha(t) d_X(t) + \langle \widehat \beta_t, d_X \rangle - d_Y(t).
\end{align}
To rearrange \eqref{eq:prederror.aux2} further, note that
$\alpha(t)=\widetilde\alpha^\star(t)-\langle \beta_t, \widehat v_t\rangle$ and $\widehat\alpha(t)=\widehat\alpha^\star(t)-\langle \widehat\beta_t, \widehat v_t\rangle$, which implies
$$
	[\alpha(t) - \widehat \alpha(t)] X_{n+1}^{\mu}(t) + \langle \widetilde \beta_t, X_{n+1}^{\mu} \rangle
	= \langle \widetilde \beta_t, X_{n+1}^\mu - X_{n+1}^\mu(t) \widehat v_t \rangle +
	[\widetilde \alpha^\star(t) - \widehat \alpha^\star(t)] X_{n+1}^\mu(t),
$$
and \eqref{eq:prederror.aux2} becomes
$$
	\langle \widetilde \beta_t, X_{n+1}^\mu - X_{n+1}^\mu(t) \widehat v_t \rangle +
	[\widetilde \alpha^\star(t) - \widehat \alpha^\star(t)] X_{n+1}^\mu(t) + \widehat \alpha(t) d_X(t) + \langle \widehat \beta_t, d_X \rangle - d_Y(t).
$$
Therefore, by the elementary inequality \eqref{eq:elementaryinequality},
\begin{align*}
	&\frac{1}{16} \E\Big(\big(m_t(X_{n+1}) -\widehat\alpha_0(t)-\widehat\alpha(t)X_{n+1}(t) -\langle \widehat\beta_t , X_{n+1} \rangle\big)^2 \mid V_n\Big) \\
	&\leq \E\Big( \langle \widetilde \beta_t, X_{n+1}^\mu - X_{n+1}^\mu(t) \widehat v_t \rangle^2 \mid V_n \Big) + \E\Big( [\widetilde \alpha^\star(t) - \widehat \alpha^\star(t)]^2 X_{n+1}^\mu(t)^2 \mid V_n \Big) \\
	&\qquad  + \widehat \alpha(t)^2 d_X(t)^2 +  \langle \widehat \beta_t, d_X \rangle^2  +  d_Y(t)^2,
\end{align*}
where $V_n := \{(X_1, Y_1), \ldots, (X_n,Y_n)\}$.
By the i.i.d.\ property and the fact that the second moments are bounded, $\E(d_Y(t)^2) = O(n^{-1})$, $\E(d_X(t)^2) = O(n^{-1})$, and $\E(\|d_X\|^2) = O(n^{-1})$, so that the last three terms of the above expression are $O_P(n^{-1})$.
For the second term of the above expression, from the model representation
$$
Y_i^{\bar \mu}(t) = \widetilde \alpha^\star(t) X_i^{\bar \mu}(t) + \langle \beta_t, \widehat \delta_{i,t} \rangle + \varepsilon_i^{\bar \mu}(t),
$$
and the fact that $\sum_{i=1}^n X_i^{\bar \mu}(t) \widehat \delta_{i,t}(s) = 0$ for all $s\in[0,1]$, we have
\begin{align*}
\widehat \sigma_{XY}(t,t) =  \widetilde \alpha^\star(t) \widehat \sigma_X(t,t) + \frac{1}{n} \sum_{i=1}^n \varepsilon_i^{\bar \mu}(t) X_i^{\bar \mu}(t)
\end{align*}
which, together with $\widehat\alpha^\star(t) = \widehat \sigma_{XY}(t,t) / \widehat \sigma_X(t,t)$, implies
$$
	[\widetilde \alpha^\star(t) - \widehat \alpha^\star(t)]^2 X_{n+1}^\mu(t)^2  = \bigg( \frac{1}{n} \sum_{i=1}^n \varepsilon_i^{\bar \mu}(t) X_i^{\bar \mu}(t)\bigg)^2 \bigg(\frac{X_{n+1}^\mu(t)}{\widehat \sigma_X(t,t)}\bigg)^2
$$
Thus, since $X_{n+1}$ is independent of $V_n$,
\begin{align*}
\E\Big( [\widetilde \alpha^\star(t) - \widehat \alpha^\star(t)]^2 X_{n+1}^\mu(t)^2 \mid V_n \Big) = \bigg( \frac{1}{n} \sum_{i=1}^n \varepsilon_i^{\bar \mu}(t) X_i^{\bar \mu}(t) \bigg)^2
\frac{\sigma_X(t,t)}{\widehat \sigma_X(t,t)^2}.
\end{align*}
Moreover, with $\bar\varepsilon(t):=n^{-1}\sum_{i=1}^n\varepsilon_i(t)$ and $d_X(t)=\bar X(t)-\E(X(t))$,
\begin{align*}
\frac{1}{n} \sum_{i=1}^n \varepsilon_i^{\bar \mu}(t) X_i^{\bar \mu}(t)
=
\frac{1}{n} \sum_{i=1}^n \varepsilon_i(t) X_i^\mu(t)
-\bar\varepsilon(t)d_X(t).
\end{align*}
By Assumptions A.4 and A.6 of the main paper, the random variables $\varepsilon_i(t)X_i^\mu(t)$ are i.i.d.\ and centered. Furthermore, the model and Assumptions A.2--A.4 of the main paper imply $\E(\varepsilon(t)^2)<\infty$, and independence gives
$
\E(\varepsilon(t)^2X^\mu(t)^2)
=
\E(\varepsilon(t)^2)\E(X^\mu(t)^2)
<\infty.
$
Consequently,
$
\frac{1}{n} \sum_{i=1}^n \varepsilon_i(t)X_i^\mu(t)=O_P(n^{-1/2}).
$
Likewise, $\bar\varepsilon(t)=O_P(n^{-1/2})$ and $d_X(t)=O_P(n^{-1/2})$, so that
\begin{align*}
\frac{1}{n} \sum_{i=1}^n \varepsilon_i^{\bar \mu}(t)X_i^{\bar \mu}(t)
=O_P(n^{-1/2}).
\end{align*}
Finally, $|\widehat \sigma_X(t,t)- \sigma_X(t,t)| = O_P(n^{-1/2})$, and therefore
\begin{align*}
\E\Big( [\widetilde \alpha^\star(t) - \widehat \alpha^\star(t)]^2 X_{n+1}^\mu(t)^2 \mid V_n \Big)
=O_P(n^{-1}).
\end{align*}
It remains to show that
$$
\E\Big( \langle \widetilde \beta_t, X_{n+1}^\mu - X_{n+1}^\mu(t) \widehat v_t \rangle^2 \mid V_n\Big) = O_P(n^{-\varphi}).
$$
Let us first consider
\begin{align*}
	\E\Big(\langle \widetilde \beta_t,
	X_{n+1}^{\mu}-X_{n+1}^{\mu}(t)\widehat v_t\rangle^2
	\mid V_n\Big)
	&=\langle \widetilde \beta_t,
	\Gamma_X(\widetilde \beta_t)\rangle
	-2\langle \sigma_X(t,\cdot),\widetilde \beta_t\rangle
	\langle \widehat v_t,\widetilde \beta_t\rangle
	+\sigma_X(t,t)
	\langle \widehat v_t,\widetilde \beta_t\rangle^2.
\end{align*}
Since
$$
\langle \widetilde \beta_t , \Gamma_{\delta,t}(\widetilde \beta_t)\rangle=\langle \widetilde \beta_t, \Gamma_X(\widetilde \beta_t) \rangle
-\sigma_X(t,t) \langle v_t, \widetilde \beta_t\rangle^2
$$
and
\begin{align*}
	&\sigma_X(t,t)\langle v_t,\widetilde \beta_t\rangle^2
	-2\langle \sigma_X(t,\cdot),\widetilde \beta_t\rangle
	\langle \widehat v_t,\widetilde \beta_t\rangle
	+\sigma_X(t,t)
	\langle \widehat v_t,\widetilde \beta_t\rangle^2\\
	&=\sigma_X(t,t)
	\big(\langle v_t,\widetilde \beta_t\rangle
	-\langle \widehat v_t,\widetilde \beta_t\rangle\big)^2\\
	&=O_P(n^{-1}).
\end{align*}
we have
\begin{align*}
\E\Big(\langle \widetilde \beta_t,
X_{n+1}^{\mu}-X_{n+1}^{\mu}(t)\widehat v_t\rangle^2
\mid V_n\Big)
&=\langle \widetilde \beta_t,
\Gamma_{\delta,t}(\widetilde \beta_t)\rangle
+O_P(n^{-1}).
\end{align*}
Then, the assertion follows with Theorem 4.3(a) of the main paper.
\hfill$\square$

\subsection{Auxiliary results for the proof of Theorem 4.5(a) of the main paper}

\begin{lemma} \label{lem:thm43aux1}
Let the conditions of Theorem 4.5 of the main paper hold true. Let $H_i(t) := Z_i(t)/\sigma_X(t,t)$. Then, we have $\E(\|H_i\|_{0,\tilde \kappa}^p) < \infty$ for $p = q/3 \geq 2$ and some Hölder coefficient $\tilde \kappa \in (0,1]$ satisfying $p \tilde \kappa > 1$.
\end{lemma}

\begin{proof}
We first collect some general properties of the Hölder norm:
\begin{itemize}
	\item[(i)]  $\|\E(A)\|_{0,\kappa} \leq \E(\|A\|_{0,\kappa})$ for any $C[0,1]$-valued random variable $A$;
	\item[(ii)] $\|fg\|_{0,\kappa} \leq \|f\|_{0,\kappa} \|g\|_{0,\kappa}$ for any $f,g \in C[0,1]$;
	\item[(iii)] $\|f/g\|_{0,\kappa} \leq c^{-1} \|f\|_{0,\kappa} (1 +c^{-1} \|g\|_{0,\kappa})$ for any $f,g \in C[0,1]$ with $c= \inf_{t \in [0,1]} |g(t)| > 0$.
\end{itemize}
Property (i) holds because $\|\E(A)\|_\infty \leq \E(\|A\|_\infty) $ and $[\E(A)]_\kappa \leq \E([A]_\kappa)$ by interchanging expectation and supremum. For property (ii) notice that $[fg]_\kappa \leq [f]_\kappa \|g\|_\infty + [g]_\kappa \|f\|_\infty$ and $\|fg\|_\infty \leq \|f\|_\infty \|g\|_\infty$, which implies $\|fg\|_\infty + [fg]_\kappa  \leq (\|f\|_\infty + [f]_\kappa)(\|g\|_\infty + [g]_\kappa)$. Property (iii) follows by property (ii) together with the facts that
$$
\|1/g\|_\infty \leq \frac{1}{c}, \qquad
	[1/g]_\kappa = \sup_{t \neq s} \frac{\big|\frac{1}{g(t)} - \frac{1}{g(s)}\big|}{|t-s|^\kappa}
	=  \sup_{t \neq s} \frac{|g(s) - g(t)|}{|g(t) g(s)|\, |t-s|^\kappa} \leq \frac{[g]_\kappa}{c^2} \leq  \frac{\|g\|_{0,\kappa}}{c^2}.
$$
In our case, let $c= \inf_{t \in [0,1]} \sigma_X(t,t) > 0$.
Set $\tilde \kappa = \kappa$ if $d = 0$ and $\tilde \kappa = 1$ if $d \geq 1$. Then, by the additional moment condition in Theorem 4.5 of the main paper, $\E(\|X_i\|_{0, \tilde \kappa}^{q}) < \infty$ and $\E(\|Y_i\|_{0, \tilde \kappa}^{q/2}) < \infty$.
This holds because, for $d \geq 1$, $\E(\|X_i\|_{d,\kappa}^q) < \infty$ implies $\E(\|X_i\|_{0,1}^q) < \infty$.
Property (iii) implies
\begin{align} \label{eq:H-expectation}
	\E(\|H_i\|_{0,\tilde \kappa}^p) \leq c^{-p} \E(\|Z_i\|_{0,\tilde \kappa}^p) (1 +c^{-1} \|\sigma_X(\cdot,\cdot)\|_{0,\tilde \kappa})^p.
\end{align}
By the triangle inequality, properties (i) and (ii), and finite second moments, we have
$$
	\|\sigma_{X}(\cdot, \cdot)\|_{0,\tilde \kappa} \leq \|\E(X^2) - \E(X)^2\|_{0,\tilde \kappa} \leq \|\E(X^2)\|_{0,\tilde \kappa} + \|\E(X)\|_{0,\tilde \kappa}^2
	\leq 2 \E(\|X\|_{0,\tilde \kappa}^2) < \infty,
$$
so it remains to show that $\E(\|Z_i\|_{0,\tilde \kappa}^p) < \infty$.
With property (ii) and the triangle inequality,
\begin{align*}
	\|Z_i\|_{0,\tilde \kappa}
	\leq \|X_i^\mu Y_i^\mu\|_{0,\tilde \kappa} + \| \alpha^\star (X_i^\mu)^2\|_{0,\tilde \kappa}
	\leq \|X_i^\mu \|_{0,\tilde \kappa} \|Y_i^\mu\|_{0,\tilde \kappa} + \| \alpha^\star \|_{0,\tilde \kappa} \| X_i^\mu\|_{0,\tilde \kappa}^2,
\end{align*}
so $\| Z_i\|_{0,\tilde \kappa}^p \leq 2^{p-1} \|X_i^\mu \|_{0,\tilde \kappa}^p \|Y_i^\mu\|_{0,\tilde \kappa}^p + 2^{p-1} \| \alpha^\star \|_{0,\tilde \kappa}^p \| X_i^\mu\|_{0,\tilde \kappa}^{2p}$, and, by Hölder's inequality using the conjugate exponents $1/3$ and $2/3$,
\begin{align} \label{eq:Z-expectation}
	\E(\| Z_i\|_{0,\tilde \kappa}^p) \leq 2^{p-1} \E(\|X_i^\mu \|_{0,\tilde \kappa}^{3p})^{1/3} \E(\|Y_i^\mu\|_{0,\tilde \kappa}^{3p/2})^{2/3} + 2^{p-1} \| \alpha^\star \|_{0,\tilde \kappa}^p \E(\| X_i^\mu\|_{0,\tilde \kappa}^{2p}).
\end{align}
With $3p=q$ note that, by the additional moment condition in Theorem 4.5 of the main paper,
\begin{align}
	\E(\|X_i^\mu \|_{0,\tilde \kappa}^{3p}) &= \E(\|X_i - \E(X_i)\|_{0,\tilde \kappa}^{q}) \nonumber \\
	&\leq \E(\|X_i \|_{0,\tilde \kappa} + \E(\|X_i\|_{0,\tilde \kappa})^q) \nonumber \\
	&\leq 2^{q-1} \E( \|X_i \|_{0,\tilde \kappa}^q) + 2^{q-1} \E(\|X_i\|_{0,\tilde \kappa})^q < \infty. \label{eq:center-X}
\end{align}
By the same argument we also have $\E(\|Y_i^\mu\|_{0,\tilde \kappa}^{3p/2}) = \E(\|Y_i^\mu\|_{0,\tilde \kappa}^{q/2}) < \infty$.
Furthermore, by the triangle and Cauchy-Schwarz inequalities and properties (i) and (ii),
\begin{align*}
	\|\sigma_{XY}(\cdot, \cdot)\|_{0,\tilde \kappa}
	&\leq \|\E(X_iY_i)\|_{0,\tilde \kappa} + \|\E(X_i)\|_{0,\tilde \kappa} \|\E(Y_i)\|_{0,\tilde \kappa} \\
	&\leq \E(\|X_i\|_{0,\tilde \kappa}^2)^{1/2} \E(\|Y_i\|_{0,\tilde \kappa}^2)^{1/2} + \E(\|X_i\|_{0,\tilde \kappa}) \E(\|Y_i\|_{0,\tilde \kappa}) < \infty,
\end{align*}
and, therefore, by property (iii)
$$
	\| \alpha^\star \|_{0,\tilde \kappa} \leq c^{-1} \|\sigma_{XY}(\cdot, \cdot)\|_{0,\tilde \kappa}  (1 + c^{-1} \|\sigma_{X}(\cdot, \cdot)\|_{0,\tilde \kappa}) < \infty.
$$
Consequently, by \eqref{eq:Z-expectation},
\begin{align} \label{eq:Z-Hölderbound}
	\E(\| Z_i\|_{0,\tilde \kappa}^p) < \infty,
\end{align}
and, by \eqref{eq:H-expectation}, the assertion follows.
\end{proof}

\begin{lemma}[Kolmogorov-Chentsov Theorem] \label{lem:kolmogorov}
Let $(V(t))_{t \in [0,1]}$ be a real-valued sto\-chas\-tic process.
Suppose that for some constants $a, b, C > 0$
$$
	\E(|V(t) - V(s)|^a) \leq C |t-s|^{1+b} \qquad \text{for all} \ s,t \in [0,1].
$$
Then, for every exponent $0 < \gamma < b/a$ there exists a modification
$\widetilde V$ of $V$, meaning $\widetilde V(t) = V(t)$ almost surely for each fixed $t$, whose sample paths are almost surely $\gamma$-Hölder continuous with $\E([\widetilde V]_\gamma^a) \leq C_\gamma  < \infty$, where $C_\gamma$ depends only on $C$, $a$, $b$, and $\gamma$.
\end{lemma}

The proof of the Kolmogorov-Chentsov Theorem may be found, e.g., in \cite{kallenberg2021foundations}, Theorem 4.23.

\begin{lemma} \label{lem:thm43aux2}
Let the conditions of Theorem 4.5 of the main paper hold true. Then, as $n \to \infty$,
\begin{itemize}
	\item[(a)] $\sup_{t \in [0,1]} |\bar X(t) - \E(X(t))| = O_P(n^{-1/2})$
	\item[(b)] $\sup_{t \in [0,1]} |\bar Y(t) - \E(Y(t))| = O_P(n^{-1/2})$
	\item[(c)] $\sup_{t \in [0,1]} |\widehat \sigma_X(t,t) - \sigma_X(t,t)| = o_P(1)$
\end{itemize}
\end{lemma}

\begin{proof}
\textit{Proof of (a).}
Define
$$
  H_i^X(t) := X_i(t) - \E(X(t)), \qquad
  G_n^X(t) := \frac{1}{\sqrt n} \sum_{i=1}^n H_i^X(t)
  = \sqrt n(\bar X(t) - \E(X(t))).
$$
We need to show that $\sup_{t\in[0,1]} |G_n^X(t)| = O_P(1)$.
By Assumption A.2 of the main paper and the centering argument used for
$X_i^\mu$ in the proof of Lemma~\ref{lem:thm43aux1} (see
\eqref{eq:center-X}), there exist $p\ge2$ and $\tilde\kappa\in(0,1]$ with
$p\tilde\kappa>1$ such that $\E\big(\|H_i^X\|_{0,\tilde\kappa}^p\big) < \infty$.
Repeating the Rosenthal argument leading to \eqref{eq:Gn-incrementbound}, now
with $H_i$ replaced by $H_i^X$, we obtain for all $s,t\in[0,1]$,
\begin{equation}\label{eq:GnX-increment}
  \E(| G_n^X(t) - G_n^X(s)|^p)
  \le C_X |t-s|^{p\tilde\kappa},
\end{equation}
where $C_X<\infty$ does not depend on $n$.

Applying Lemma~\ref{lem:kolmogorov} with $a=p$, $1+b=p\tilde\kappa$ and any
$0<\gamma<b/a=\tilde\kappa-1/p$, we obtain, for each $n$, a modification
$\widetilde G_n^X$ of $G_n^X$ such that
$$
  |\widetilde G_n^X(t) - \widetilde G_n^X(s)|
  \le [\widetilde G_n^X]_\gamma |t-s|^\gamma
  \quad\text{for all } s,t\in[0,1],\ \text{a.s.},
$$
and $\E([\widetilde G_n^X]_\gamma^p) \le C_\gamma <\infty$, with
$C_\gamma$ independent of $n$. In particular, $\sup_{n\ge1} \E([\widetilde G_n^X]_\gamma^p) < \infty$ and $[\widetilde G_n^X]_\gamma = O_P(1)$.
Since $\widetilde G_n^X$ is a modification of $G_n^X$, the increment bound
\eqref{eq:GnX-increment} also holds with $G_n^X$ replaced by
$\widetilde G_n^X$. Therefore, by Theorem~22.10 in
\cite{davidson_stochastic_2021}, the sequence $\widetilde G_n^X$ is
stochastically equicontinuous and thus tight in $C[0,1]$. In particular,
$$
  \sup_{t\in[0,1]} |\widetilde G_n^X(t)| = O_P(1).
$$
Finally, $G_n^X$ and $\widetilde G_n^X$ are modifications with almost surely
continuous paths, hence indistinguishable as $C[0,1]$-valued random elements
and they induce the same probability law on $C[0,1]$. Consequently,
$$
  \sup_{t\in[0,1]} |G_n^X(t)|
  \overset{d}{=} \sup_{t\in[0,1]} |\widetilde G_n^X(t)|
  = O_P(1),
$$
and the assertion follows.

\bigskip
\noindent
\textit{Proof of (b).}
The proof is analogous to that of part (a), with $X$ replaced by $Y$.
We define $H_i^Y(t) := Y_i(t) - \E(Y(t))$ and
$G_n^Y(t) := n^{-1/2}\sum_{i=1}^n H_i^Y(t)$ and then apply the same
Rosenthal, Kolmogorov-Chentsov, and tightness arguments as for $G_n^X$.

\bigskip
\noindent
\textit{Proof of (c).}
Writing $X_i(t) - \bar X(t) = X_i^\mu(t) - \Delta_X(t)$ with $X_i^\mu(t) := X_i(t) - \E(X(t))$ and
$\Delta_X(t) := \bar X(t) - \E(X(t))$, we obtain $\widehat\sigma_X(t,t)
  = \frac{1}{n}\sum_{i=1}^n (X_i^\mu(t))^2 - (\Delta_X(t))^2$
Therefore,
$$
  \widehat\sigma_X(t,t) - \sigma_X(t,t)
  = \frac{1}{n}\sum_{i=1}^n (X_i^\mu(t))^2
      - \E((X_i^\mu(t))^2)
    - (\Delta_X(t))^2 =: A_n(t) - (\Delta_X(t))^2.
$$
Hence
\begin{equation}\label{eq:sigmax-decomp-simple}
  \sup_{t\in[0,1]}|\widehat\sigma_X(t,t) - \sigma_X(t,t)|
  \le \sup_{t\in[0,1]}|A_n(t)|
     + \sup_{t\in[0,1]}(\Delta_X(t))^2.
\end{equation}
By part (a) of this lemma we know that
$$
  \sup_{t\in[0,1]}|\Delta_X(t)|
  = \sup_{t\in[0,1]}|\bar X(t) - \E(X(t))|
  = O_P(n^{-1/2}),
$$
so that
\begin{equation}\label{eq:DeltaX-square-simple}
  \sup_{t\in[0,1]}(\Delta_X(t))^2 = O_P(n^{-1}) = o_P(1).
\end{equation}
It remains to show $\sup_{t\in[0,1]}|A_n(t)| = o_P(1)$. Define
$$
  W_i(t) := (X_i^\mu(t))^2 - \E((X_i^\mu(t))^2),
  \qquad
  A_n(t) = \frac{1}{n}\sum_{i=1}^n W_i(t).
$$
Because $W_i$ are centered $C[0,1]$-valued random elements and $\E(\|W_i\|_\infty)< \infty$, the Banach-space law of large numbers (e.g. Corollary 7.10 in \citealt{ledoux1991probability}) implies
\begin{equation}\label{eq:An-sup-op}
  \sup_{t\in[0,1]} |A_n(t)| = o_P(1).
\end{equation}
Combining \eqref{eq:sigmax-decomp-simple},
\eqref{eq:DeltaX-square-simple}, and \eqref{eq:An-sup-op} yields $\sup_{t\in[0,1]}|\widehat\sigma_X(t,t) - \sigma_X(t,t)| = o_P(1)$.
\end{proof}

\subsection{Proof of Theorem 4.5(a) of the main paper}

We first decompose the sample cross-covariance as follows:
$$
	\widehat \sigma_{XY}(t,t) = \sigma_{XY}(t,t) + R_{XY}(t) - \Delta_X(t) \Delta_Y(t), \quad \widehat \sigma_{X}(t,t) = \sigma_{X}(t,t) + R_{X}(t) - (\Delta_X(t))^2
$$
with sample mean estimation error
$$
	\Delta_Y(t) = \bar Y(t) - \E(Y_i(t)), \quad \Delta_X(t) = \bar X(t) - \E(X_i(t)),
$$
and
$$
	R_{XY}(t) = \frac{1}{n} \sum_{i=1}^n \Big(X_i^{\mu}(t) Y_i^{\mu}(t) - \sigma_{XY}(t,t)\Big), \quad
	R_{X}(t) = \frac{1}{n} \sum_{i=1}^n \Big((X_i^{\mu}(t))^2 - \sigma_{X}(t,t)\Big)
$$
being the cross-covariance estimation error when knowing the true mean, where $Y_i^\mu(t) = Y_i(t) - \E(Y_i(t))$ and $X_i^\mu(t) = X_i(t) - \E(X_i(t))$ are the population-centered processes. Recall that $\alpha^\star(t) = \sigma_{XY}(t,t)/\sigma_X(t,t)$.
Then,
\begin{align*}
		\widehat \alpha^\star(t) - \alpha^\star(t)
		&= \frac{ \widehat \sigma_{XY}(t,t) }{\widehat \sigma_X(t,t)} -  \alpha^\star(t) \\
		&= \frac{\sigma_{XY}(t,t)+R_{XY}(t)
		-\Delta_X(t)\Delta_Y(t)}{\widehat \sigma_X(t,t)}
		-\frac{\alpha^\star(t)
		\big(\sigma_X(t,t)+R_X(t)-(\Delta_X(t))^2\big)}
		{\widehat \sigma_X(t,t)}\\
		&= \frac{R_{XY}(t)-\alpha^\star(t)R_X(t)}
		{\widehat \sigma_X(t,t)}
		+\frac{\alpha^\star(t)(\Delta_X(t))^2
		-\Delta_X(t)\Delta_Y(t)}{\widehat \sigma_X(t,t)}.
\end{align*}
Note that $Z_i(t) = X_i^\mu(t) (Y_i^\mu(t) - \alpha^\star(t) X_i^\mu(t))$ can be represented as
$$
	\frac{1}{n} \sum_{i=1}^n Z_i(t) = R_{XY}(t) - \alpha^\star(t) R_X(t).
$$
Define
$$
	G_n(t) := \frac{1}{\sqrt n} \sum_{i=1}^n H_i(t), \quad H_i(t) := \frac{Z_i(t)}{\sigma_X(t,t)},
$$
and
$$
	E_n(t) := \frac{(\sigma_X(t,t) - \widehat \sigma_X(t,t))G_n(t) + \sqrt n (\alpha^\star(t)(\Delta_X(t))^2 - \Delta_X(t) \Delta_Y(t))}{\widehat \sigma_X(t,t)}.
$$
so that the expression of interest admits the decomposition
$$
	\sqrt n (\widehat \alpha^\star(t) - \alpha^\star(t) )
	= G_n(t) + E_n(t).
$$
The term $G_n(t)$ will determine the asymptotic distribution,
while the corrections due to $E_n(t)$ will be shown to be negligible uniformly in $t \in [0,1]$.

To analyze $G_n(t)$, we note that for any finite collection of time points
$0 \le t_1 < \dots < t_N \le 1$, the vector
\begin{align} \label{eq:vectorH}
\frac{1}{\sqrt{n}}\sum_{i=1}^n
\big(H_i(t_1),\dots,H_i(t_N)\big)
\end{align}
is a sum of i.i.d.\ centered random vectors with finite second moments, so by the multivariate central limit theorem, \eqref{eq:vectorH} converges in distribution to a $N$-dimensional centered Gaussian vector with a covariance matrix
whose entries are given by
$$
	\mathrm{Cov}(H_i(t_j),H_i(t_k)) = \frac{\sigma_Z(t_j,t_k)}{\sigma_X(t_j,t_j) \sigma_X(t_k,t_k)} = \sigma_{\widehat \alpha^\star}(t_j,t_k).
$$
To extend the multivariate normality result in $\mathbb{R}^N$ to the space $C[0,1]$, we need to show that the process $G_n$ is stochastically equicontinuous.

By Lemma \ref{lem:thm43aux1} we have the Hölder norm moment bound $\E(\|H_i\|_{0,\tilde \kappa}^p) < \infty$ for some $p\geq 2$ and $\tilde \kappa \in (0,1]$ with $p \tilde \kappa > 1$.
Then, for every fixed $n$ and every $s,t \in [0,1]$,
		\begin{align*}
			\E\Big( | G_n(t) - G_n(s)|^p \Big)
			&= \frac{1}{n^{p/2}} \E\bigg( \Big| \sum_{i=1}^n (H_i(t) - H_i(s)) \Big|^p \bigg) \\
			&\leq \frac{R_p}{n^{p/2}}
			\left(\left[\sum_{i=1}^n \E\Big(|H_i(t)-H_i(s)|^2\Big)\right]^{p/2}
			+\sum_{i=1}^n \E\Big(|H_i(t)-H_i(s)|^p\Big)\right)\\
			&=R_p
			\left(\E\Big(|H_1(t)-H_1(s)|^2\Big)^{p/2}
			+n^{1-p/2}\E\Big(|H_1(t)-H_1(s)|^p\Big)\right)\\
			&\leq 2 R_p  \E( |H_1(t) - H_1(s)|^p )
		\end{align*}
where the first inequality follows from Rosenthal's inequality given in equation \eqref{eq:rosenthal}, with a constant $R_p < \infty$ depending only on $p$. The second inequality follows from Jensen's inequality and the fact that $n^{1-p/2} \leq 1$.
Since $\E(|H_i(t) - H_i(s)|^p) \leq \E(\|H_i\|_{0,\tilde \kappa}^p) |t-s|^{p \tilde \kappa}$, we obtain
\begin{align} \label{eq:Gn-incrementbound}
	\E( | G_n(t) - G_n(s)|^p ) \leq C |t-s|^{p \tilde \kappa}, \quad C:= 2 R_p \E(\|H_i\|_{0,\tilde \kappa}^p) < \infty,
\end{align}
where the constant $C$ holds uniformly in $n$.

Then, by the Kolmogorov-Chentsov theorem given in Lemma~\ref{lem:kolmogorov}
with $a=p$, $b = p\tilde \kappa - 1$ and $0 < \gamma < b/a = \tilde \kappa - 1/p$,
there exists a modification $\widetilde G_n$ of $G_n$ such that
$$
  |\widetilde G_n(t) - \widetilde G_n(s)|
  \leq [\widetilde G_n]_\gamma |t-s|^\gamma
  \quad \text{for all } s,t \in [0,1], \quad \text{almost surely},
$$
and $\E([\widetilde G_n]_\gamma^p) \leq C_\gamma < \infty$, where $C_\gamma$ does not depend on $n$.
Since $C$ and $C_\gamma$ do not depend on $n$, we have
$\sup_{n \geq 1} \E([\widetilde G_n]_\gamma^p) < \infty$ and thus
$[\widetilde G_n]_\gamma = O_P(1)$ as $n \to \infty$.

For each fixed $n$, $\widetilde G_n$ is a modification of $G_n$, so they have
the same finite-dimensional distributions and the increment bound
\eqref{eq:Gn-incrementbound}
also holds with $G_n$ replaced by $\widetilde G_n$. Therefore, by Theorem 22.10
in \cite{davidson_stochastic_2021}, the sequence $\widetilde G_n$ is
stochastically equicontinuous and thus tight in $C[0,1]$. Together with the
finite-dimensional convergence of $G_n$ (and hence of $\widetilde G_n$) and
Theorem 7.5 in \cite{Billingsley_1999}, this yields
$$
  \widetilde G_n \to_D \mathcal{GP}(0,\sigma_{\widehat\alpha^\star})
  \quad \text{in } C[0,1].
$$
Since $G_n$ and $\widetilde G_n$ are modifications, they are indistinguishable as $C[0,1]$-valued random
elements and thus induce the same probability law on $C[0,1]$. In particular,
they have the same weak limit, so
$$
  G_n \to_D \mathcal{GP}(0,\sigma_{\widehat\alpha^\star})
  \quad \text{in } C[0,1], \quad \text{as } n\to\infty.
$$
By Lemma \ref{lem:thm43aux2} and the continuous mapping theorem and Slutsky's theorem for function spaces \citep[see][p.~259, Theorems 18.10 and 18.11]{vanderVaart1998}, we get
$$
\sup_{t\in[0,1]}\big|E_n(t)\big| = \sup_{t\in[0,1]}\big|\sqrt{n}(\widehat\alpha^\star(t)-\alpha^\star(t)) - G_n(t)\big|
= o_P(1),
$$
so that $\sqrt{n}(\widehat\alpha^\star-\alpha^\star)$ and $G_n$ have the same weak limit in $C[0,1]$, which completes the proof of part (a).

\subsection{Auxiliary results for the proof of Theorem 4.5(b) of the main paper}

\begin{lemma} \label{lem:thm43aux3}
Let the conditions of Theorem 4.5 of the main paper hold true. Then, as $n \to \infty$,
$$
	\sup_{s,t \in [0,1]} \bigg| \frac{1}{n}\sum_{i=1}^n
       \big(\widehat Z_i(s)\widehat Z_i(t) - Z_i(s)Z_i(t)\big) \bigg| = O_P(n^{-1/2}).
$$
\end{lemma}

\begin{proof}
Write
$$
\Delta_X(t) := \bar X(t) - \E(X(t)),\quad
\Delta_Y(t) := \bar Y(t) -  \E(Y(t)),\quad
\Delta_\alpha(t) := \widehat \alpha^\star(t) - \alpha^\star(t),
$$
and define $\Delta Z_i(t):=\widehat Z_i(t)-Z_i(t)$.
Using the expansion
\begin{align*}
	\widehat Z_i(t)
	&= X_i^{\bar \mu}(t)\big(Y_i^{\bar \mu}(t) - \widehat \alpha^\star(t) X_i^{\bar \mu}(t)\big)\\
&= (X_i^\mu(t) - \Delta_X(t))\big(Y_i^\mu(t) - \Delta_Y(t) - (\Delta_\alpha(t) + \alpha^\star(t))(X_i^\mu(t) - \Delta_X(t))\big)
\end{align*}
and subtracting $Z_i(t) = X_i^\mu(t)(Y_i^\mu(t) - \alpha^\star(t)X_i^\mu(t))$ shows that $\Delta Z_i(t)$ is a finite linear combination of terms of the form
$X_i^\mu(t) \Delta_Y(t)$, $X_i^\mu(t)^2 \Delta_\alpha(t)$, $X_i^\mu(t) \Delta_\alpha(t) \Delta_X(t)$,  $Y_i^\mu(t) \Delta_X(t)$, $X_i^\mu(t) \alpha^\star(t) \Delta_X(t)$,  $\Delta_X(t) \Delta_Y(t)$, $\Delta_\alpha(t) \Delta_X(t)$, $\Delta_\alpha(t) \Delta_X(t)^2$, $\alpha^\star(t) \Delta_X(t)^2$.
By Assumptions A.2--A.3 and Theorem 4.5(a) of the main paper, together with Lemma \ref{lem:thm43aux2}, we have
$$
	\|\Delta_X\|_\infty = O_P(n^{-1/2}), \quad  \|\Delta_Y\|_\infty = O_P(n^{-1/2}), \quad \|\Delta_\alpha\|_\infty = O_P(n^{-1/2})
$$
and
$$
	\|X_i^\mu\|_\infty = O_P(1), \quad \|Y_i^\mu\|_\infty = O_P(1), \quad \|\alpha^\star\|_\infty < \infty.
$$
Hence each of the terms from the linear combination above is $O_P(n^{-1/2})$ in the sup-norm. Therefore,
$$
\|\Delta Z_i\|_\infty = \sup_{t\in[0,1]}|\widehat Z_i(t)-Z_i(t)| = O_P(n^{-1/2}).
$$
Similarly, $\|Z_i\|_\infty = \sup_{t \in [0,1]} |Z_i(t)| = O_P(1)$.
Now use the identity
$$
\widehat Z_i(s)\widehat Z_i(t) - Z_i(s)Z_i(t)
= Z_i(s)\Delta Z_i(t) + Z_i(t)\Delta Z_i(s) + \Delta Z_i(s)\Delta Z_i(t).
$$
Hence, by Cauchy-Schwarz,
\begin{align*}
	&\sup_{s,t \in [0,1]}\bigg|\frac{1}{n}\sum_{i=1}^n
\big(\widehat Z_i(s)\widehat Z_i(t) - Z_i(s)Z_i(t)\big)\bigg| \\
&\le \frac{2}{n}\sum_{i=1}^n \|Z_i\|_\infty \|\Delta Z_i\|_\infty
+ \frac{1}{n}\sum_{i=1}^n \|\Delta Z_i\|_\infty^2 \\
&\le 2\Big(\frac{1}{n}\sum_{i=1}^n \|Z_i\|_\infty^2\Big)^{1/2}
\Big(\frac{1}{n}\sum_{i=1}^n \|\Delta Z_i\|_\infty^2\Big)^{1/2} + \frac{1}{n}\sum_{i=1}^n \|\Delta Z_i\|_\infty^2,
\end{align*}
where $\frac{1}{n}\sum_{i=1}^n \|Z_i\|_\infty^2 = O_P(1)$ follows by the law of large numbers because $Z_i$ are i.i.d., and it remains to show $\frac{1}{n}\sum_{i=1}^n \|\Delta Z_i\|_\infty^2 = O_P(n^{-1})$.
Let
$$
	R_n := \|\Delta_X\|_\infty + \|\Delta_Y\|_\infty + \|\Delta_\alpha\|_\infty = O_P(n^{-1/2})
$$
and define
$$
	M_i := 1 + \|X_i^\mu\|_\infty + \|Y_i^\mu\|_{\infty} + \|X_i^\mu\|_\infty^2.
$$
Then,
$\|\Delta Z_i\|_\infty \leq C M_i(R_n + R_n^2 + R_n^3)$ and
$$
	\frac{1}{n}\sum_{i=1}^n \|\Delta Z_i\|_\infty^2 \leq C^2(R_n + R_n^2 + R_n^3)^2 \frac{1}{n}\sum_{i=1}^n M_i^2 = O_P(n^{-1})
$$
by the law of large numbers because $M_i$ are i.i.d. with $\E(M_i^2) < \infty$ by Assumption A.2 of the main paper, and the assertion follows.
\end{proof}

\begin{lemma} \label{lem:thm43aux4}
Let the conditions of Theorem 4.5 of the main paper hold. Then, as $n \to \infty$,
$$
	\sup_{s,t \in [0,1]} \bigg| \frac{1}{n}\sum_{i=1}^n
       \big(Z_i(s)Z_i(t) - \sigma_Z(s,t)\big) \bigg| = o_P(1).
$$
\end{lemma}

\begin{proof}
Define
$$
	U_i(s,t):=Z_i(s)Z_i(t)-\sigma_Z(s,t), \qquad (s,t)\in[0,1]^2.
$$
Since $Z_i$ has continuous sample paths, $U_i$ is a random element of
$C([0,1]^2)$. Moreover, $U_1,U_2,\ldots$ are i.i.d. and centered in the
separable Banach space $C([0,1]^2)$ equipped with the supremum norm. By
Lemma~\ref{lem:thm43aux1},
$$
	\E\|Z_i\|_\infty^2<\infty.
$$
Therefore,
$$
	\E\|U_i\|_\infty
	\leq
	\E\|Z_i\|_\infty^2+\sup_{s,t\in[0,1]}|\sigma_Z(s,t)|
	\leq
	2\E\|Z_i\|_\infty^2
	<\infty.
$$
The Banach-space strong law of large numbers applied in $C([0,1]^2)$  (Corollary 7.10 in \citealt{ledoux1991probability}) yields
$$
	\bigg\|
	\frac1n\sum_{i=1}^n U_i
	\bigg\|_\infty
	=
	\sup_{s,t\in[0,1]}
	\bigg|
	\frac1n\sum_{i=1}^n
	(Z_i(s)Z_i(t)-\sigma_Z(s,t))
	\bigg|
	\to 0
	\qquad\text{a.s.}
$$
In particular, the convergence also holds in probability.
\end{proof}

\subsection{Proof of Theorem 4.5(b) of the main paper}

For any $s,t\in[0,1]$ we can write
\begin{align*}
  \widehat\sigma_{\widehat\alpha^\star}(s,t)
    - \sigma_{\widehat\alpha^\star}(s,t)
  &= \frac{\widehat\sigma_Z(s,t)}{\widehat\sigma_X(s,s)\widehat\sigma_X(t,t)}
     - \frac{\sigma_Z(s,t)}{\sigma_X(s,s)\sigma_X(t,t)} \\
  &= \frac{\widehat\sigma_Z(s,t) - \sigma_Z(s,t)}
           {\sigma_X(s,s)\sigma_X(t,t)} +\widehat \sigma_Z(s,t)
  \bigg[
  \frac{1}{\widehat\sigma_X(s,s)\widehat\sigma_X(t,t)}
  -\frac{1}{\sigma_X(s,s)\sigma_X(t,t)}
  \bigg].
\end{align*}
Using the decomposition
$$
	\widehat\sigma_Z(s,t) - \sigma_Z(s,t) = \frac{1}{n}\sum_{i=1}^n
       \big(\widehat Z_i(s)\widehat Z_i(t) - Z_i(s)Z_i(t)\big) + \frac{1}{n}\sum_{i=1}^n
       \big( Z_i(s) Z_i(t) - \sigma_Z(s,t)\big)
$$
and $\inf_{t \in [0,1]} \sigma_X(t,t) > 0$,
Lemmas \ref{lem:thm43aux3} and \ref{lem:thm43aux4} imply
$$
		\sup_{s,t \in [0,1]} \bigg|\frac{\widehat\sigma_Z(s,t) - \sigma_Z(s,t)}
           {\sigma_X(s,s)\sigma_X(t,t)}\bigg| = o_P(1).
$$
In addition, Lemma \ref{lem:thm43aux2}(c) yields
$$
	\sup_{s,t \in [0,1]} \bigg|\widehat \sigma_Z(s,t) \bigg[ \frac{1}{\widehat\sigma_X(s,s)\widehat\sigma_X(t,t)} - \frac{1}{\sigma_X(s,s)\sigma_X(t,t)} \bigg] \bigg| = o_P(1).
$$
Consequently, $\sup_{s,t \in [0,1]} |\widehat\sigma_{\widehat\alpha^\star}(s,t)
    - \sigma_{\widehat\alpha^\star}(s,t)| = o_P(1)$.
\hfill$\square$

\clearpage
\bibliographystyle{apalike}
\bibliography{bibfile}

\end{document}